\documentclass[11pt,a4paper]{article}
\usepackage[utf8]{inputenc}
\usepackage[english]{babel}
\usepackage[T1]{fontenc}
\usepackage{amsmath, amsfonts, amssymb, amsthm}
\usepackage{lmodern}
\usepackage{fullpage}
\usepackage{microtype}
\usepackage{enumerate}
\usepackage{mathtools}
\usepackage[smalltableaux=true]{ytableau}
\usepackage{tikz}
\usepackage[colorlinks=true, linkcolor=blue]{hyperref}

\numberwithin{equation}{section}

\newtheorem{theoreme}[equation]{Theorem}
\newtheorem{lemme}[equation]{Lemma}
\newtheorem{proposition}[equation]{Proposition}
\newtheorem{corollaire}[equation]{Corollary}
\newtheorem*{theoreme:orbite_1}{Theorem \ref*{theoreme:orbite_1}}

\theoremstyle{definition}
\newtheorem{definition}[equation]{Definition}

\theoremstyle{remark}
\newtheorem{remarque}[equation]{Remark}
\newtheorem{exemple}[equation]{Example}

\newcommand{\tuple}{\boldsymbol}
\newcommand{\tuplex}{} 

\renewcommand{\P}{\mathcal{P}}
\newcommand{\Y}{\mathcal{Y}}
\newcommand{\res}{\mathrm{res}}
\newcommand{\compat}{\models}
\newcommand{\T}{\mathcal{T}}
\renewcommand{\H}{\mathcal{H}}

\newcommand{\TT}{\mathfrak{T}}
\newcommand{\PP}{\mathfrak{P}}

\author{Salim \textsc{Rostam}\thanks{Laboratoire de Mathématiques de Versailles, UVSQ, CNRS, Université Paris-Saclay, 78035 Versailles, France.}}
\title{Stuttering blocks of Ariki--Koike algebras}
\date{}

\allowdisplaybreaks[3]

\newcounter{pos}
\newcommand{\scaletikz}{.6}
\newcommand{\scalelozenge}{.7}
\newcommand{\abacus}[3]{
	\begin{center}
	\begin{tikzpicture}[scale=\scaletikz, baseline=(current bounding box.center)]
		\draw[very thick] (0, 0) -- (0, -#1 - 1);
		\setcounter{pos}{0}
		\foreach \b in {#2}
			{
			\stepcounter{pos}
			\ifnum \value{pos} = 1
				\pgfmathparse{(\b - \value{pos} - Mod(\b - \value{pos}, #1)) / #1}
				\foreach \i in {-1,...,-#1}
					{
					\draw (\pgfmathresult + 3, \i) -- (0, \i);
					\draw (\pgfmathresult + 3.6, \i - .03) node {$\dots$};
					\foreach \j in {-2,...,\pgfmathresult}
						\draw[very thin] (\j + 2, \i - .1) -- ++ (0, .2);
					}
			\fi
			\pgfmathparse{Mod(\b - \value{pos}, #1)}
			\fill (
					{(\b - \value{pos} - \pgfmathresult) / #1}
					,
					{- \pgfmathresult - 1}
					) circle (.2);
			}
			\foreach \i in {1,...,#1}
			{
			\stepcounter{pos}
			\pgfmathparse{Mod(-\value{pos}, #1) == #1 - 1}
			\ifnum \pgfmathresult = 0
				\pgfmathparse{Mod(-\value{pos}, #1)}
				\fill (
					{(-\value{pos} - \pgfmathresult) / #1}
					,
					{-\pgfmathresult - 1}
					) circle (.2);
			\else
				\breakforeach
			\fi
			}
			\foreach \k in {1,2}
			\foreach \i in {1,...,#1}
				{
				\pgfmathparse{Mod(-\value{pos}, #1)}
				\fill (
					{(-\value{pos} - \pgfmathresult) / #1}
					,
					{-\pgfmathresult - 1}
					) circle (.2);
				\stepcounter{pos}
				}

			\pgfmathparse{(-\value{pos} - Mod(-\value{pos}, #1)) / #1 + 1}
			\foreach \i in {-1,...,-#1}
				{
				\draw (\pgfmathresult - 1, \i) -- (0, \i);
				\draw (\pgfmathresult - 1.5, \i - .03) node {$\dots$};
			
				\foreach \j in {-\value{pos},...,0}
					{
					\ifnum \j > \pgfmathresult
						\draw[very thin] (\j, \i - .1) -- ++ (0, .2);
					\fi;
					}
				}
		
	\end{tikzpicture} #3
	\end{center}
}

\newcommand{\abacuscore}[3]{
	\begin{center}
	\begin{tikzpicture}[scale=\scaletikz, baseline=(current bounding box.center)]
		\draw[very thick] (0, 0) -- (0, -#1 - 1);
		\setcounter{pos}{0}
		\foreach \b in {#2}
			{
			\stepcounter{pos}
			\ifnum \value{pos} = 1
				\pgfmathparse{(\b - \value{pos} - Mod(\b - \value{pos}, #1)) / #1}
				\foreach \i in {-1,...,-#1}
					{
					\draw (\pgfmathresult + 3, \i) -- (0, \i);
					\draw (\pgfmathresult + 3.6, \i - .03) node {$\dots$};
					\foreach \j in {-2,...,\pgfmathresult}
						\draw[very thin] (\j + 2, \i - .1) -- ++ (0, .2);
					}
			\fi
			\pgfmathparse{Mod(\b - \value{pos}, #1)}
			\fill (
					{(\b - \value{pos} - \pgfmathresult) / #1}
					,
					{- \pgfmathresult - 1}
					) circle (.2);
			}
			\foreach \i in {1,...,#1}
			{
			\stepcounter{pos}
			\pgfmathparse{Mod(-\value{pos}, #1) == #1 - 1}
			\ifnum \pgfmathresult = 0
				\pgfmathparse{Mod(-\value{pos}, #1)}
				\fill (
					{(-\value{pos} - \pgfmathresult) / #1}
					,
					{-\pgfmathresult - 1}
					) circle (.2);
			\else
				\breakforeach
			\fi
			}
			\foreach \k in {1,2}
			\foreach \i in {1,...,#1}
				{
				\pgfmathparse{Mod(-\value{pos}, #1)}
				\fill (
					{(-\value{pos} - \pgfmathresult) / #1}
					,
					{-\pgfmathresult - 1}
					) circle (.2);
				\stepcounter{pos}
				}

			\pgfmathparse{(-\value{pos} - Mod(-\value{pos}, #1)) / #1 + 1}
			\foreach \i in {-1,...,-#1}
				{
				\draw (\pgfmathresult - 1, \i) -- (0, \i);
				\draw (\pgfmathresult - 1.5, \i - .03) node {$\dots$};
			
				\foreach \j in {-\value{pos},...,0}
					{
					\ifnum \j > \pgfmathresult
						\draw[very thin] (\j, \i - .1) -- ++ (0, .2);
					\fi;
					}
				}
				
		\setcounter{pos}{0}
		\foreach \b in {#2}
			{
			\stepcounter{pos}
			\pgfmathparse{Mod(\b - \value{pos}, #1)}
			\draw[red] (
					{(\b - \value{pos} - \pgfmathresult) / #1 + 1}
					,
					{- \pgfmathresult - 1}
					) node[scale=\scalelozenge] {$\blacklozenge$};
			}
			\foreach \i in {1,...,#1}
			{
			\stepcounter{pos}
			\pgfmathparse{Mod(-\value{pos}, #1) == #1 - 1}
			\ifnum \pgfmathresult = 0
				\pgfmathparse{Mod(-\value{pos}, #1)}
				\draw[red] (
					{(-\value{pos} - \pgfmathresult) / #1 + 1}
					,
					{-\pgfmathresult - 1}
					) node[scale=\scalelozenge] {$\blacklozenge$};
			\else
				\breakforeach
			\fi
			}
			\foreach \k in {1,2}
			\foreach \i in {1,...,#1}
				{
				\pgfmathparse{Mod(-\value{pos}, #1)}
				\draw[red] (
					{(-\value{pos} - \pgfmathresult) / #1 + 1}
					,
					{-\pgfmathresult - 1}
					) node[scale=\scalelozenge] {$\blacklozenge$};
				\stepcounter{pos}
				}

		\setcounter{pos}{0}
		\foreach \b in {#2}
			{
			\stepcounter{pos}
			\pgfmathparse{Mod(\b - \value{pos}, #1)}
			\fill (
					{(\b - \value{pos} - \pgfmathresult) / #1}
					,
					{- \pgfmathresult - 1}
					) circle (.2);
			}
			\foreach \i in {1,...,#1}
			{
			\stepcounter{pos}
			\pgfmathparse{Mod(-\value{pos}, #1) == #1 - 1}
			\ifnum \pgfmathresult = 0
				\pgfmathparse{Mod(-\value{pos}, #1)}
				\fill (
					{(-\value{pos} - \pgfmathresult) / #1}
					,
					{-\pgfmathresult - 1}
					) circle (.2);
			\else
				\breakforeach
			\fi
			}
			\foreach \k in {1,2}
			\foreach \i in {1,...,#1}
				{
				\pgfmathparse{Mod(-\value{pos}, #1)}
				\fill (
					{(-\value{pos} - \pgfmathresult) / #1}
					,
					{-\pgfmathresult - 1}
					) circle (.2);
				\stepcounter{pos}
				}
	\end{tikzpicture} #3
	\end{center}
}

\begin{document}
\maketitle

\begin{abstract}
We study a shift action defined on multipartitions and on residue multisets of their Young diagrams. We prove that the minimal orbit cardinality among all multipartitions associated with a given multiset depends only on the orbit cardinality of the multiset. Using abaci, this problem reduces to a convex optimisation problem over the integers with linear constraints. We solve it by proving an existence theorem for binary matrices with prescribed row, column and block sums. Finally, we give some applications to the representation theory of the Hecke algebra of the complex reflection group $G(r,p,n)$.
\end{abstract}

\section{Introduction}

It is known since Frobenius that the irreducible representations $\{\mathcal{D}^\lambda\}_{\lambda}$ of the symmetric group on $n$ letters $\mathfrak{S}_n$ over a field of characteristic $0$ are parametrised by the \emph{partitions} of $n$, that is, sequences $\lambda = (\lambda_0 \geq \dots \geq \lambda_{h-1} > 0)$ of positive integers with $\lvert \lambda \rvert \coloneqq \lambda_0 + \dots + \lambda_{h-1} = n$. When the ground field is of prime characteristic $p$, the irreducible representations $\{\mathcal{D}^\lambda\}_\lambda$ are now indexed by the \emph{$p$-regular} partitions of $n$. However, in this case some representations may not be written as a direct sum of irreducible ones. Hence, we are also interested in the \emph{blocks} of the group algebra, that is, indecomposable two-sided ideals. Blocks also partition both sets of irreducible and indecomposable representations. Brauer and Robinson proved that these blocks are parametrised by the partitions of $n$ that are \emph{$p$-cores} (see~\textsection\ref{subsection:partitions}), proving  the so-called ``Nakayama's Conjecture''. We refer to~\cite{james-kerber} for more details about the representation theory of the symmetric group.

More generally, we can consider a \emph{Hecke algebra} of the complex reflection group $G(r,1,n) \simeq (\mathbb{Z}/r\mathbb{Z})\wr \mathfrak{S}_n$. Let $F$ be a field and let $q \in F \setminus \{0, 1\}$ be of finite order $e \in \mathbb{N}_{\geq 2}$, the ``quantum characteristic''. Let $r \in \mathbb{N}^*$ and $\kappa = (\kappa_0, \dots, \kappa_{r-1}) \in (\mathbb{Z}/e\mathbb{Z})^r$. The \emph{Ariki--Koike algebra} $\H^\kappa_n$, or \emph{cyclotomic Hecke algebra} of  $G(r,1,n)$, is the unitary associative $F$-algebra defined by the generators $S, T_1, \dots, T_{n-1}$ and the relations
\begin{align*}
\prod_{k = 0}^{r-1} (S - q^{\kappa_k}) &= 0,
\\
S T_1 S T_1 &= T_1 S T_1 S,
\\
S T_a &= T_a S, & &\text{for all } a \in \{2, \dots, n-1\},
\\
(T_a + 1)(T_a - q) &= 0, &&\text{for all } a \in \{1, \dots, n-1\},
\\
T_a T_b &= T_b T_a, & &\text{for all } a, b \in \{1, \dots, n-1\} \text{ with } \lvert a - b \rvert > 1,
\\
T_a T_{a+1} T_a &= T_{a+1} T_a T_{a+1}, & &\text{for all } a \in \{1, \dots, n-2\}
\end{align*}
(see~\cite{ariki-koike,broue-malle-rouquier}). Note that we may consider a more general version of the first relation (the \emph{cyclotomic} relation for the generator~$S$), but a Morita equivalence of Dipper and Mathas~\cite{dipper-mathas} ensures that it suffices to understand  this particular case where only powers of~$q$ are involved.

The algebra $\H_n^\kappa$ is a natural deformation of the group algebra of $G(r,1,n)$ and their representation theories are deeply linked. If $r = 1$, we recover the Hecke algebra $\H_n$ of $G(1,1,n) \simeq \mathfrak{S}_n$, also known as the Hecke algebra of type $A_{n-1}$. In this case, the situation is similar to the one of the symmetric group: if $\H_n$ is semisimple then its irreducible representations $\{\mathcal{D}^\lambda\}_\lambda$ are parametrised by the partitions of $n$, otherwise they are parametrised by the $e$-regular partitions of $n$ while the blocks of $\H_n$ are parametrised by the $e$-cores of $n$.
In general, if $\H_n^\kappa$ is semisimple then its irreducible representations are parametrised by the \emph{$r$-partitions} of $n$, that is, by the $r$-tuples $\tuple{\lambda} = (\lambda^{(0)}, \dots, \lambda^{(r-1)})$ of partitions with $\lvert \tuple{\lambda}\rvert \coloneqq \lvert \lambda^{(0)} \rvert + \dots + \lvert \lambda^{(r-1)} \rvert = n$.
If $\H_n^\kappa$ is not semisimple, its irreducible representations $\{\mathcal{D}^{\tuple{\lambda}}\}_{\tuple{\lambda}}$ can be indexed by a non-trivial generalisation of $e$-regular partitions, known as \emph{Kleshchev} $r$-partitions (see~\cite{ariki-mathas, ariki_classification}). Similarly, the naive generalisation of $e$-cores to $r$-partition, the \emph{$e$-multicores}, do not  parametrise in general  the blocks  of $\H_n^\kappa$. In fact, Lyle and Mathas~\cite{lyle-mathas} proved that the blocks of $\H_n^\kappa$ are parametrised by the multisets of $\kappa$-residues modulo $e$ of the $r$-partitions of $n$ (see~\textsection\ref{subsection:multipartitions}). We can identify this parametrising set with a subset $Q^\kappa_n$ of $Q^+ \coloneqq \mathbb{N}^{\mathbb{Z}/e\mathbb{Z}}$ and we denote by $\H_\alpha^\kappa$ the block corresponding to $\alpha \in Q^\kappa_n$. Moreover, to each $r$-partition $\tuple{\lambda}$ of $n$ we can associate an element $\alpha_\kappa(\tuple{\lambda}) \in Q^\kappa_n$. The blocks of $\H_n^\kappa$ partition the set of $r$-partitions of $n$ via the map $\tuple{\lambda} \mapsto \alpha_\kappa(\tuple{\lambda})$. We say that the block indexed by $\alpha \in Q^\kappa_n$  \emph{contains} the $r$-partition $\tuple{\lambda}$ if $\alpha_\kappa(\tuple{\lambda}) = \alpha$.

Now let $p \in \mathbb{N}^*$ dividing both $r$ and $e$, let $d \coloneqq \frac{r}{p}$ and $\eta \coloneqq \frac{e}{p}$ and assume that $\kappa$ is \emph{compatible} with $(d, \eta, p)$ (cf.~\eqref{equation:kappa_compatible}). The algebra $\H_n^\kappa$ has a natural subalgebra $\H_{p, n}^\kappa$ that is a Hecke algebra of the complex reflection group $G(r,p,n)$ (see~\cite{ariki_representation,broue-malle-rouquier} and also~\cite{rostam}, where we emphasise the connection between these two papers). The subalgebra $\H_{p, n}^\kappa \subseteq \H_n^\kappa$ is the subalgebra of fixed points of the automorphism $\sigma$ of order $p$ defined on the generators of $\H_n^\kappa$ by
\begin{equation}
\label{equation:introduction-sigma}
\begin{aligned}
\sigma(S) &= \zeta S,
\\
\sigma(T_a) &= T_a, &&\text{for all } a \in \{1, \dots, n-1\},
\end{aligned}
\end{equation}
where $\zeta \in F^\times$ has order $p$.
 The representation theory of $\H_{p, n}^\kappa$ can be studied using Clifford theory, see for instance \cite{ariki_representation,chlouveraki-jacon,genet-jacon,hu-mathas_decomposition}. Let $\tuple{\lambda}$ be a Kleshchev $r$-partition and let $\mathcal{D}^{\tuple{\lambda}}$ be the irreducible $\H_n^\kappa$-module indexed by $\tuple{\lambda}$. The restriction $\left.\mathcal{D}^{\tuple{\lambda}}\right\downarrow^{\H_n^\kappa}_{\H_{p, n}^\kappa}$ is isomorphic to a sum of irreducible $\H_{p, n}^\kappa$ modules. The number of irreducible $\H_{p, n}^\kappa$-modules that appear depends on the cardinality of the orbit $[\tuple{\lambda}]$ of $\tuple{\lambda} = \bigl(\lambda^{(0)}, \dots, \lambda^{(r-1)}\bigr)$ under the shift action defined by
\[
\prescript{\sigma}{}{\tuple{\lambda}} \coloneqq (\lambda^{(r-d)}, \dots, \lambda^{(r-1)}, \lambda^{(0)}, \dots, \lambda^{(r-d-1)}).
\]
A natural question is to determine the extreme cardinalities of the orbits under this action, and thus the extremal number of irreducible $\H_{p, n}^\kappa$-module that appear during the restriction process. The answer is an easy exercise when considering all $r$-partitions of $n$.

\begin{proposition}
\label{proposition:introduction_minmax}
Let  $\mathcal{C} \coloneqq \{\#[\tuple{\lambda}] : \tuple{\lambda} \text{ is an } r\text{-partition of } n\} \subseteq \mathbb{N}^*$. We have $\max \mathcal{C} = p$ and $\min \mathcal{C} = \frac{p}{\gcd(p, n)}$.
\end{proposition}

Already with this Proposition~\ref{proposition:introduction_minmax}, we can give some results about the representation theory of $\H_{p, n}^\kappa$, such as the number of ``Specht modules'' that appear in the restriction of Specht modules of $\H_n^\kappa$ to $\H_{p, n}^\kappa$ (as defined in \cite{hu-mathas_graded}). We can also prove that a natural basis of $\H_{p, n}^\kappa$ is not ``adapted'' cellular (cf. \textsection\ref{subsubsection:adapted_cellularity}). In order to give block-analogue answers, we  introduce a shift action on $Q^+$. More precisely, for any $\alpha \in Q^+$ we define $\sigma \cdot \alpha$ by shifting coordinates by $\eta$ and we write $[\alpha]$ for the orbit of $\alpha$. The subalgebra $\H_{[\alpha]}^\kappa \coloneqq \oplus_{\beta \in [\alpha]} \H_\beta^\kappa$ of $\H_n^\kappa$ is stable under $\sigma : \H_n^\kappa \to \H_n^\kappa$, and we denote by  $\H_{p, [\alpha]}^\kappa$ the subalgebra of fixed points.  The two shift actions that we have defined are compatible in the following way: if $\tuple{\lambda}$ is an $r$-partition then
\[
\alpha_\kappa(\prescript{\sigma}{}{\tuple{\lambda}}) = \sigma \cdot \alpha_\kappa(\tuple{\lambda})
\]
(see Lemma~\ref{lemme:sigma_alpha_lambda}).
Hence, if $\alpha \coloneqq \alpha_\kappa(\tuple{\lambda})$ we always have $\#[\tuple{\lambda}] \geq \#[\alpha]$. It is easy to see in small examples that  we may have a strict inequality. However, the main results of this paper, Theorem~\ref{theoreme:orbite_1} and Corollary~\ref{corollaire:orbite_gnrl}, prove that equality holds if we allow us to choose among all $r$-partitions $\tuple{\mu}$ with $\alpha_\kappa(\tuple{\mu}) = \alpha_\kappa(\tuple{\lambda})$. It leads to a more precise version of the ``$\min$ part'' of Proposition~\ref{proposition:introduction_minmax}.

\begin{theoreme}
\label{theorem:introduction_min_block}
Let $\tuple{\lambda}$ be an $r$-partition and let $\alpha \coloneqq \alpha_\kappa(\tuple{\lambda})$. There exists an $r$-partition $\tuple{\mu}$ with $\alpha_\kappa(\tuple{\mu}) = \alpha$ and $\#[\tuple{\mu}] = \#[\alpha]$.
\end{theoreme}

Wada~\cite{wada} proved a more precise version of the ``$\max$ part'' of Proposition~\ref{proposition:introduction_minmax}.
In order to classify the blocks of $\H_{p, n}^\kappa$,  Wada proved that there (almost) always exists an $r$-partition $\tuple{\mu}$ with $\alpha_\kappa(\tuple{\mu}) = \alpha$ and $\#[\tuple{\mu}] = p$. His proof uses the classification result of~\cite{lyle-mathas} and is very short. In contrast, the proof of Theorem~\ref{theorem:introduction_min_block} that we present here is quite long and we did not find a way to use~\cite{lyle-mathas}. At least, as in~\cite{lyle-mathas}, we use the \emph{abacus representation} of partitions.

Theorem~\ref{theorem:introduction_min_block} allows us to  give the block-analogues of the results for $\H_{p, n}^\kappa$ that we deduced from Proposition~\ref{proposition:introduction_minmax}, that is, the same results but for $\H_{p, [\alpha]}^\kappa$ instead of $\H_{p, n}^\kappa$. We can also deduce from Theorem~\ref{theorem:introduction_min_block} some consequences about the blocks of  $\H_n^\kappa$. We say that an $r$-partition $\tuple{\lambda}$ (respectively an element $\alpha \in Q^+$) is \emph{stuttering} if $\#[\tuple{\lambda}] = 1$ (resp. $\#[\alpha] = 1$).
By Theorem~\ref{theorem:introduction_min_block}, we know that the block indexed by a stuttering $\alpha \in Q^\kappa_n$ always contains a stuttering $r$-partition.

The paper is organised as follows. Section~\ref{section:combinatorics} is devoted to combinatorics. More specifically, in~\textsection\ref{subsection:partitions} we define partitions of integers and to each partition $\lambda$ we associate an element $\alpha(\lambda) \in Q^+ = \mathbb{N}^{\mathbb{Z}/e\mathbb{Z}}$. In~\textsection\ref{subsection:abaci} we recall the abacus representation of partitions. In~\textsection\ref{subsection:parametrisation}, to an $e$-core $\lambda$ we associate the \emph{$e$-abacus variable} $x = (x_0, \dots, x_{e-1}) \in \mathbb{Z}^e$. The main fact of this subsection is the equality
\[
\alpha(\lambda)_0 = \frac{1}{2}\sum_{i = 0}^{e-1} x_i^2
\]
(cf. Proposition~\ref{proposition:residu_0_chapeau}), which we recall from~\cite{fayers_weight}. In~\textsection\ref{subsection:multipartitions} we extend the previous definitions to multipartitions, so we can in~\textsection\ref{subsection:shifts} define the two shift maps $\tuple{\lambda} \mapsto \prescript{\sigma}{}{\tuple{\lambda}}$ and $\alpha \mapsto \sigma \cdot \alpha$ involved in the statement of our main results, Theorem~\ref{theoreme:orbite_1} and Corollary~\ref{corollaire:orbite_gnrl}. Theorem~\ref{theoreme:orbite_1} is the case $\#[\alpha] = 1$ of Theorem~\ref{theorem:introduction_min_block} and Corollary~\ref{corollaire:orbite_gnrl} is the general case.

Section~\ref{section:approx} is devoted to technical tools that we need to prove Theorem~\ref{theoreme:orbite_1}. The reader who wants to focus on the proof of Theorem~\ref{theoreme:orbite_1} may, in the first instance, skip this section. In~\textsection\ref{subsection:binary_matrices}, we study the existence of a chain of interchanges $\bigl(\begin{smallmatrix} 1&0\\0&1\end{smallmatrix}\bigr) \leftrightarrow \bigl(\begin{smallmatrix} 0&1\\1&0\end{smallmatrix}\bigr)$ in a family of binary matrices (Corollary~\ref{corollaire:interversions_somme_blocs}).
In~\textsection\ref{subsection:binary_averaging}, we recall a special case of a general theorem of Gale~\cite{gale} and Ryser~\cite{ryser} about the existence of a binary matrix with prescribed row and column sums. We apply the results of~\textsection\ref{subsection:binary_matrices} to impose extra conditions on block sums (Proposition~\ref{proposition:elt_base_canonique_bloc}).
Finally, we gathered in \textsection\ref{subsection:inegalites} some inequalities; in particular, Lemma~\ref{lemme:forte_convexite_general} is a special case of a Jensen's inequality for strongly convex functions and Lemma~\ref{lemme:erreur_compensent} is an application to an inequality involving the fractional part map.

In Section~\ref{section:demo}, we prove the main result, Theorem~\ref{theoreme:orbite_1}. After a preliminary step in \textsection\ref{subsection:using_shift_invariance}, we give in \textsection\ref{subsection:key_lemma} a key lemma (Lemma~\ref{lemme:inegalite_implique_conj}), which reduces the proof of Theorem~\ref{theoreme:orbite_1} to a (strongly) convex optimisation problem over the integers with linear constraints. We find in~\textsection\ref{subsection:tentative_naive} a partial solution, in~\textsection\ref{subsection:rectification} we use Proposition~\ref{proposition:elt_base_canonique_bloc} to find a solution and eventually in \textsection\ref{subsection:proof_main_theorem} we prove Theorem~\ref{theoreme:orbite_1}.

Finally, we give in Section~\ref{section:applications} two applications of Corollary~\ref{corollaire:orbite_gnrl}. The general idea is that we will have more precise results with Corollary~\ref{corollaire:orbite_gnrl} than with Proposition~\ref{proposition:introduction_minmax}.
We quickly recall in~\textsection\ref{subsection:cellular_algebras} the theory of cellular algebras of Graham and Lehrer~\cite{graham-lehrer}, the Ariki--Koike algebra $\H_n^\kappa$ and its blocks being particular cases. In \textsection\ref{subsection:cellularity_fixed_points_subalgebra} we use the map $\mu \coloneqq \sum_{j = 0}^{p-1} \sigma^j$ to construct a family of bases for $\H_{p, [\alpha]}^\kappa = \mu(\H_{[\alpha]}^\kappa)$ (Proposition~\ref{proposition:basis_stable}). We deduce in \textsection\ref{subsubsection:full_orbit} that $\H_{p, [\alpha]}^\kappa$ is a cellular algebra if $\#[\alpha] = p$, and  $\H_{p, n}^\kappa$ is cellular if $p$ and $n$ are coprime. Then, using Corollary~\ref{corollaire:orbite_gnrl}, we show that if $\#[\alpha] < p$ and $p$ is odd then the bases that we constructed for $\H_{p, [\alpha]}^\kappa$ are not \emph{adapted cellular} (see \textsection\ref{subsubsection:adapted_cellularity}).
Finally, in~\textsection\ref{subsection:restriction_specht}, we study the maximal number of ``Specht modules of $\H_{p, [\alpha]}^\kappa$'' (see~\cite{hu-mathas_graded}) that appear when restricting the Specht modules of $\H_{[\alpha]}^\kappa$ to $\H_{p, [\alpha]}^\kappa$.

\section{Combinatorics}
\label{section:combinatorics}

In this section, we recall standard definitions of combinatorics such as (multi)partitions and their associated abaci. We also introduce two \emph{shift} actions and then state our main result, Theorem~\ref{theoreme:orbite_1}.

Let $e \geq 2$ be an integer. We identify $\mathbb{Z}/e\mathbb{Z}$  with the set $\{0, \dots, e-1\}$. We will use $\mathbb{N} = \mathbb{Z}_{\geq 0}$ to denote the set of non-negative integers and $\mathbb{N}^* = \mathbb{Z}_{> 0}$.

\subsection{Partitions}
\label{subsection:partitions}

A \emph{partition} of $n$ is a non-increasing sequence of positive integers $\lambda = (\lambda_0, \dots, \lambda_{h-1})$ of sum~$n$.  We will write $\lvert \lambda\rvert \coloneqq n$ and $h(\lambda) \coloneqq h$. If $\lambda$ is a partition, we denote by $\Y(\lambda)$ its Young diagram, defined by:
\[
\Y(\lambda) \coloneqq \left\{(a, b) \in \mathbb{N}^2 : 0 \leq a \leq h(\lambda)-1 \text{ and }  0 \leq b \leq \lambda_a - 1\right\}.
\]

\begin{exemple}
We represent the Young diagram associated with the partition $(4,3,3,1)$ by
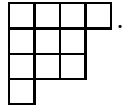
\[
\ydiagram{4,3,3,1}\, .
\]
\end{exemple}

We refer to the elements of $\mathbb{N} \times \mathbb{N}$ as \emph{nodes}. For instance, the elements of $\Y(\lambda)$ are nodes.  A node $\gamma \in \Y(\lambda)$ is \emph{removable} if $\Y(\lambda)\setminus \{\gamma\}$ is the Young diagram of a partition. Similarly, a node $\gamma \notin \Y(\lambda)$ is \emph{addable} if $\Y(\lambda)\cup \{\gamma\}$ is the Young diagram of a partition. A \emph{rim hook} of $\lambda$ is a subset of $\Y(\lambda)$ of the following form:
\[
r^{\lambda}_{(a, b)} \coloneqq  \left\{(a', b') \in \Y(\lambda) : a' \geq a,\, b' \geq b \text{ and } (a'+1, b'+1) \notin \Y(\lambda) \right\},
\]
where $(a, b) \in \Y(\lambda)$. We say that $r^{\lambda}_{(a, b)}$ is an \emph{$l$-rim hook} if it has cardinality $l$. Note that $1$-rim hooks are exactly removable nodes.
The \emph{hand} of a rim hook $r^{\lambda}_{(a, b)}$ is the node $(a, b') \in r^{\lambda}_{(a, b)}$ with maximal $b'$. 
The set $\Y(\lambda) \setminus r^{\lambda}_{(a, b)}$ is the Young diagram of a certain partition $\mu$, obtained by \emph{unwrapping} or \emph{removing} the rim hook $r^{\lambda}_{(a, b)}$ from $\lambda$.
Conversely, we say that $\lambda$ is obtained from~$\mu$ by \emph{wrapping on} or \emph{adding} the rim hook $r^{\lambda}_{(a, b)}$. We say that a partition $\lambda$ is an \emph{$e$-core}  if~$\lambda$ has no $e$-rim hook.

\begin{exemple}
\label{exemple:ruban}
We consider the partition $\lambda \coloneqq (3,2,2,1)$. An example of a $3$-rim hook is
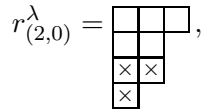
\[
r^\lambda_{(2,0)} =
\begin{ytableau}
{}&{}&{}
\\
{}&{}
\\
\times&\times
\\
\times
\end{ytableau}\, ,
\]
and a $4$-rim hook is for instance
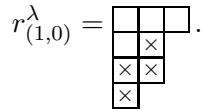
\[
r^\lambda_{(1,0)} = 
\begin{ytableau}
{} &{} & {}
\\
 {}&\times
\\
 \times& \times
\\ 
\times
\end{ytableau}\, .
\]
We can check that $\lambda$ has no $5$-rim hook so it is a $5$-core. We will see in \textsection\ref{subsection:parametrisation} how to use abaci to easily know whether a partition is an $e$-core.
\end{exemple}

 The \emph{residue}  of a node $\gamma = (a, b)$ is $\res(\gamma) \coloneqq b - a \pmod{e}$. For any $i \in \mathbb{Z}/e\mathbb{Z}$, an $i$-node is a node with residue $i$. If~$\lambda$ is a partition, we denote by $n^i(\lambda)$ the  multiplicity of $i$ in the multiset of  residues of all elements of $\Y(\lambda)$. Note that $\sum_{i = 0}^{e-1} n^i(\lambda) = \lvert\lambda\rvert$.

Let $Q$ be a free $\mathbb{Z}$-module of rank $e$ and let $\{\alpha_i\}_{i \in \mathbb{Z}/e\mathbb{Z}}$ be a basis. We have $Q = \oplus_{i=0}^{e-1} \mathbb{Z}\alpha_i$ and we define $Q^+ \coloneqq \oplus_{i=0}^{e-1} \mathbb{N}\alpha_i$. For any $\alpha \in Q$, we denote by $\lvert \alpha\rvert \in \mathbb{Z}$ the sum of its coordinates in the basis $\{\alpha_i\}_{i \in \mathbb{Z}/e\mathbb{Z}}$.
 If $\lambda$ is a partition we define
\[
\alpha(\lambda) \coloneqq \sum_{\gamma \in \Y(\lambda)} \alpha_{\res(\gamma)} = \sum_{i =0}^{e-1} n^i(\lambda) \alpha_i \in Q^+.
\]
Note that $\lvert \alpha(\lambda)\rvert = \lvert\lambda\rvert$.
More generally, if $\Gamma$ is any finite subset of $\mathbb{N}^2$  we will write $\alpha(\Gamma) \coloneqq \sum_{\gamma \in \Gamma} \alpha_{\res(\gamma)}$.
\begin{remarque}
\label{remarque:racine_ruban}
If $r^{\lambda}$ is an $l$-rim hook then $\alpha(r^{\lambda}) = \sum_{i = 0}^{l-1} \alpha_{i_0 + i}$ for some $i_0 \in \mathbb{Z}/e\mathbb{Z}$. In particular, if $r^{\lambda}$ is an $e$-rim hook then $\alpha(r^{\lambda}) = \sum_{i = 0}^{e-1} \alpha_i$.
\end{remarque}

Finally, if for $\alpha \in Q^+$ there exists a partition $\lambda$  such that $\alpha = \alpha(\lambda)$, we say that $\alpha \in Q^+$ is \emph{associated} with $\lambda$. For an arbitrary $\alpha \in Q^+$, there can exist two different partitions $\lambda \neq \mu$ such that $\alpha = \alpha(\lambda) = \alpha(\mu)$. However, if we restrict to $e$-cores then the map $\lambda \mapsto \alpha(\lambda)$ is one-to-one (see~\cite[2.7.41 Theorem]{james-kerber} or~\cite{lyle-mathas}). Hence, the following subset of $Q^+$:
\[
Q^* \coloneqq \left\{\alpha\in Q^+ : \alpha \text{ is associated with some } e \text{-core}\right\},
\]
is in bijection with the set of $e$-cores. The aim of \textsection\ref{subsection:parametrisation} is to explicitly construct a bijection between~$Q^*$ and~$\mathbb{Z}^{e-1}$.

\subsection{Abaci}
\label{subsection:abaci}

The abacus representation of a partition has been first introduced by 
James~\cite{james}. Here, we follow  the construction of \cite{lyle-mathas}. To a partition $\lambda = (\lambda_0, \dots, \lambda_{h-1})$ we associate the \emph{$\beta$-number}~$\beta(\lambda)$ defined as the sequence $(\lambda_{a-1} - a)_{a \geq 1}$, where $\lambda_{a-1} \coloneqq 0$ for any $a > h$. 

\begin{lemme}
\label{lemma:beta_number}
The $\beta$-number of a partition~$\lambda$ is strictly decreasing and satisfies $\beta(\lambda)_a = -a$ for all $a > h(\lambda)$. Conversely, if~$\beta = (\beta_a)_{a \geq 1}$ is a strictly decreasing sequence of integers with $\beta_a = -a$ for all $a \gg 1$ then~$\beta$ is the $\beta$-number of some partition.
\end{lemme}

The following result is well-known (see for instance~\cite[2.7.13 Lemma]{james-kerber}). 

\begin{lemme}
\label{lemme:transfert_bille}
Let $l \in \mathbb{N}^*$. A partition $\lambda$ has an $l$-rim hook if and only if there is an element $b \in \beta(\lambda)$ such that $b - l \notin \beta(\lambda)$. In that case, if $\mu$ is the partition that we obtain by removing this $l$-rim hook, then $\beta(\mu)$ is obtained by replacing $b$ by $b - l$ in $\beta(\lambda)$ and then sorting in decreasing order.
\end{lemme}

In particular, if $\mu$ is a partition and if $b \in \beta(\mu)$ and $l \in \mathbb{N}^*$ are such that $b + l \notin \beta(\mu)$, then replacing $b$ by $b+l$ in $\beta(\mu)$ and sorting in decreasing order is equivalent to adding an $l$-rim hook to $\mu$. Indeed,  by Lemma~\ref{lemma:beta_number} the sequence that we obtain from $\beta(\mu)$ is the $\beta$-number of a certain partition $\lambda$, and by Lemma~\ref{lemme:transfert_bille} the partition $\mu$ is obtained from $\lambda$ by unwrapping an $l$-rim hook. That is, the partition $\lambda$ is obtained from $\mu$ by wrapping on an $l$-rim hook (see also \cite[Lemma 5.26]{mathas}).

Lemma~\ref{lemme:transfert_bille} ensures that for any partition $\lambda$, there is a unique $e$-core $\overline{\lambda}$ that can be obtained by successively removing $e$-rim hooks.
We say that $\overline{\lambda}$ is the \emph{$e$-core of $\lambda$}, and the number of $e$-rim hooks that we remove to obtain~$\overline{\lambda}$ from~$\lambda$ is the \emph{$e$-weight} of~$\lambda$.
We now consider an abacus with $e$-runners, each runner being a horizontal copy of $\mathbb{Z}$ and displayed in the following way: the $0$th runner is on top and the origins of each copy of $\mathbb{Z}$ are aligned with respect to a vertical line.
We record the elements of $\beta(\lambda)$ on this abacus according to the following rule: there is a bead at position $j \in \mathbb{Z}$ on the runner $i \in \{0, \dots, e-1\}$ if and only if there exists $a \geq 1$ such that $\beta(\lambda)_a = i + je$. 
We say that this abacus is the \emph{$e$-abacus associated with $\lambda$}.  Note that, for any $i \in \{0, \dots, e-1\}$ and $j \gg 0$, on runner~$i$ there is a bead at position $-j$ (by Lemma~\ref{lemma:beta_number}) and a \emph{gap} (that is, no bead) at position $j$.

\begin{exemple}
\label{exemple:abaque}
We consider the partition $\lambda = (3,2,2,1)$ from Example~\ref{exemple:ruban}. Its $\beta$-number is $\beta(\lambda) = (2, 0, -1, -3, -5, \dots)$.
The associated $3$-abacus is
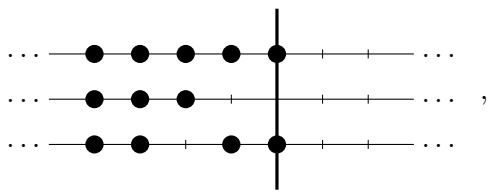
\abacus{3}{3,2,2,1}{,}
the associated $4$-abacus is
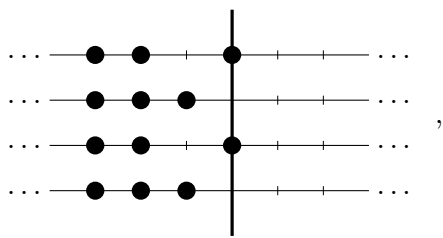
\abacus{4}{3,2,2,1}{,}
and the associated $5$-abacus is
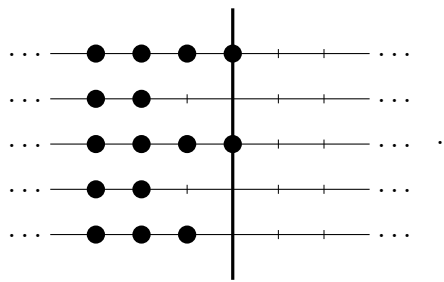
\abacus{5}{3,2,2,1}{.}
Recall that counting the number of gaps up each bead (continuing counting on the left starting from the $(e-1)$th runner when reaching the $0$th one) recovers the underlying partition.
\end{exemple}

Let $\lambda$ be a partition and let us consider its associated $e$-abacus. We give the abacus interpretation of Lemma~\ref{lemme:transfert_bille} in the two particular cases $l = 1$ and $l = e$. Note that for any $a \in \{1, \dots, h(\lambda)\}$, we have $\beta(\lambda)_a = i \pmod{e}$ if and only if $(a-1, \lambda_{a-1} - 1) \in \Y(\lambda)$ is an $i$-node.

\begin{corollaire}
\label{corollaire:transfert_bille_1}
\begin{itemize}
\item
We can move a bead on position $j \in \mathbb{Z}$ on runner $i \in \{0, \dots, e-1\}$ to the previously free position $j$ on runner $i-1$ (to the previously free position $j-1$ on runner $e-1$ if $i = 0$) if and only if $\lambda$ has a removable $i$-node.
\item We can move a bead on position $j$ on runner $i$ to the previously free position $j$ on runner $i+1$ (to the previously free position $j+1$ on runner $0$ if $i = e-1$) if and only if $\lambda$ has an addable $(i+1)$-node.
\end{itemize}
\end{corollaire}

\begin{corollaire}
\label{corollaire:transfert_bille_e}
\begin{itemize}
\item We can slide a bead on position $j$ on runner $i$ to the previously free position $j-1$ on the same runner if and only if $\lambda$ has an $e$-rim hook of hand residue $i$. Hence, the partition $\lambda$ is an $e$-core if and only if its associated $e$-abacus has no gap, that is, no bead has a gap on its left.
\item We can slide a bead on position $j$ on runner $i$ to the previously free position $j+1$ on the same runner if and only if $\lambda$ has an addable $e$-rim hook of hand residue $i$. Hence, we can always add an $e$-rim hook of hand residue $i$ to $\lambda$.
\end{itemize}
\end{corollaire}

\begin{exemple}
We consider the partition $\lambda = (3, 2, 2, 1)$ of Example~\ref{exemple:ruban}. Recall that we gave in Example~\ref{exemple:abaque} the $e$-abaci for $e \in \{3,4,5\}$. The $3$-abacus of $\lambda$ has only one gap thus $r^\lambda_{(2,0)}$ is the only $3$-rim hook that we can remove. The $4$-abacus of $\lambda$ has two gaps, corresponding to the two lonely beads on runners $0$ and $2$. Sliding left the bead on runner $2$ (respectively $0$) corresponds to removing the $4$-rim hook $r^\lambda_{(0,1)} = $ {\scriptsize\begin{ytableau}
{}&\times&{\color{blue}{\times}}
\\
{}&\times
\\
{}&\times
\\
{}
\end{ytableau}} (resp. $r^\lambda_{(1,0)} =$ {\scriptsize\begin{ytableau}
{}&{}&{}
\\
{}&{\color{red}{\times}}
\\
\times&\times
\\
\times
\end{ytableau}}). The hand residue, in blue (resp. red), matches since the multiset of residues is given by {\scriptsize \ytableaushort{01{\color{blue}{2}},2{\color{red}0},12,0}}. The $5$-abacus of $\lambda$ has no gap thus $\lambda$ is a $5$-core, as we saw in Example~\ref{exemple:ruban}.
\end{exemple}


\subsection{Parametrisation of \texorpdfstring{$Q^*$}{Q*}}
\label{subsection:parametrisation}

In this subsection, we will parametrise by $\mathbb{Z}^{e-1}$ the  set $Q^*$ of all $\alpha \in Q^+$ that are associated with $e$-cores.
Given an $e$-abacus associated with an $e$-core $\lambda$ and $i \in \{0, \dots, e-1\}$, let us write $x_i(\lambda) \in \mathbb{Z}$ for the position of the first gap on the runner $i$. We say that $x_0(\lambda), \dots, x_{e-1}(\lambda)$ are the \emph{parameters} of the $e$-abacus, or the \emph{$e$-abacus variables} of $\lambda$. We will also use the notation $x(\lambda) = \bigl(x_0(\lambda), \dots, x_{e-1}(\lambda)\bigr) \in \mathbb{Z}^e$. 

\begin{exemple}
We use \scalebox{\scalelozenge}{$\color{red}{\blacklozenge}$} to denote the position of each $x_i(\lambda)$. The $3$-abacus associated with the empty partition (which is a $3$-core indeed), of associated  $\beta$-number $(-1,-2,\dots)$, is
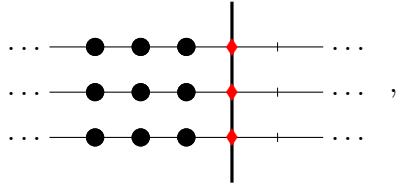
\abacuscore{3}{0}{,}
thus the associated parameters are $x_0(\emptyset) = x_1(\emptyset) = x_2(\emptyset) = 0$. As we saw in Example~\ref{exemple:abaque}, the $5$-abacus associated with the $5$-core $\lambda = (3,2,2,1)$ is
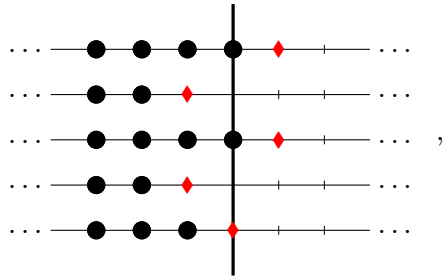
\abacuscore{5}{3,2,2,1}{,}
thus the associated parameters are
\[
x_0(\lambda) = x_2(\lambda) = 1, \qquad x_1(\lambda) = x_3(\lambda) = -1, \qquad x_4(\lambda) = 0.
\]
\end{exemple}

We have the following consequences of Corollary~\ref{corollaire:transfert_bille_1}.

\begin{lemme}
\label{lemme:difference_ni_xi}
Let $\lambda$ be an $e$-core. For all $i \in \{0, \dots, e-1\}$ we have $x_i(\lambda) = n^i(\lambda) - n^{i+1}(\lambda)$.
\end{lemme}

\begin{proof}
We construct the partition~$\lambda$ adding nodes one by one.  By Corollary~\ref{corollaire:transfert_bille_1}, the runner~$i$ is involved exactly when adding an $i$-node or an $(i+1)$-node, which corresponds to a bead going from runner~$i-1$ to~$i$ and runner~$i$ to~$i+1$ respectively. Thus, once the $e$-core~$\lambda$ is constructed exactly~$n^i(\lambda)$ (respectively~$n^{i+1}(\lambda)$) beads were added to (resp. removed from)  runner~$i$. We conclude since $x_i(\emptyset) = 0$.
\end{proof}

\begin{corollaire}
\label{corollaire:ni_no_sum_xj}
Let $\lambda$ be an $e$-core.
For all $i \in \{1, \dots, e-1\}$ we have $n^i(\lambda) = n^0(\lambda) - x_0(\lambda) - \dots - x_{i-1}(\lambda)$.
\end{corollaire}


\begin{proposition}
\label{proposition:xi_somme_nulle_correspondent_abaque}
Let $x_0, \dots, x_{e-1} \in \mathbb{Z}$. Then $x_0 + \dots + x_{e-1} = 0$ if and only if there is an $e$-core $\lambda$ such that $x_i = x_i(\lambda)$ for all $i \in \{0, \dots, e-1\}$.
\end{proposition}

\begin{proof}
If $\lambda$ is an $e$-core then Lemma~\ref{lemme:difference_ni_xi} ensures that $x_0(\lambda) + \dots + x_{e-1}(\lambda) = 0$. Now let $x_0, \dots, x_{e-1} \in \mathbb{Z}$ such that $x_0 + \dots + x_{e-1} = 0$ and consider the $e$-abacus~$A$ of parameters $x_0, \dots, x_{e-1}$. That is, the runner~$i$ is full of beads before position~$x_i - 1$ and has only gaps after position~$x_i$. For any $i, j \in \{0, \dots, e-1\}$, we consider the operation $t_{i,j}$ that moves the rightmost bead on runner~$i$ to the leftmost gap on runner~$j$. We have $t_{i,j}^{-1} = t_{j, i}$, and by Corollaries~\ref{corollaire:transfert_bille_1} and~\ref{corollaire:transfert_bille_e} the operation~$t_{i, j}$ preserves the set of $e$-abaci associated with $e$-cores. After the operation $t_{i, j}$, the parameter~$x_i$ (respectively~$x_j$) decreases (resp. increases) by~$1$ and thus the new parameters that we obtain still sum to~$0$. By induction on $\max_i x_i$, we can find a sequence $t_{i_1,j_1}, \dots, t_{i_k, j_k}$ of operations so that the parameters of the $e$-abacus $\widetilde{A} \coloneqq t_{i_k, j_k} \cdots t_{i_1, j_1}(A)$ are all zero. Hence, the $e$-abacus~$\widetilde{A}$ corresponds to the empty $e$-core, and we conclude since $A = t_{j_1, i_1}\cdots t_{j_k, i_k}(\widetilde{A})$.
\end{proof}


We thus have a bijection
\[
\{e\text{-cores}\} \overset{1:1}{\longleftrightarrow} \{(x_0, \dots, x_{e-1}) \in \mathbb{Z}^e : x_0 + \dots + x_{e-1} = 0\} \eqqcolon \mathbb{Z}^e_0.
\]

The function $n^0$ defined on the set of $e$-cores is a symmetric polynomial in $x_0, \dots, x_{e-1}$. Indeed, exchanging the runners $i$ and $i+1$ for any $i \in \{0, \dots, e-2\}$ only modifies the number of $(i+1)$-nodes (by Corollary~\ref{corollaire:transfert_bille_1}) and we conclude since the symmetric group $\mathfrak{S}(\{0, \dots, e-1\})$ is generated by the transpositions $(i, i+1)$ for all $i \in \{0, \dots, e-2\}$. We explicitly give this symmetric polynomial in the following proposition, where~$\lVert \cdot \rVert$ denotes the Euclidean norm on tuples of integers.

\begin{proposition}
\label{proposition:residu_0_chapeau}
Let $\lambda$ be an $e$-core. We have:
\[
n^0(\lambda) = \frac{1}{2}\lVert\tuplex{x}(\lambda)\rVert^2 = \frac{1}{2}\sum_{i = 0}^{e-1} x_i(\lambda)^2.
\]
\end{proposition}

\begin{proof}
Since $\lambda$ is an $e$-core, it has $e$-weight~$0$ and thus the result immediately follows from Lemma~\ref{lemme:difference_ni_xi} and~\cite[Proposition 2.1]{fayers_weight} (see also \cite[Bijection 2]{gks} and \cite[top of page 24]{olsson}).
\end{proof}

\begin{remarque}
Let $\lambda$ be an $e$-core. Using Corollary~\ref{corollaire:ni_no_sum_xj} and Proposition~\ref{proposition:residu_0_chapeau} we obtain
\[
n^i(\lambda) = \frac{1}{2}\lVert x(\lambda)\rVert^2 - x_0(\lambda) - \dots - x_{i-1}(\lambda),
\]
for all $i \in \{1, \dots, e-1\}$.
\end{remarque}

\begin{exemple}
We take $p = 2$ and $e = 4$. We consider the parameter $\tuplex{x} \coloneqq (2,-1,-1, 0) \in \mathbb{Z}^4_0$. The corresponding $4$-abacus is
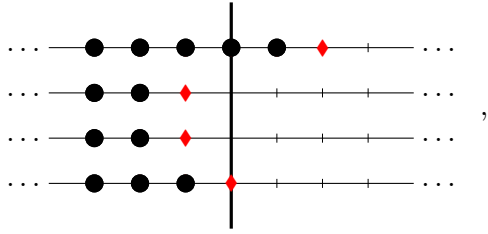
\abacuscore{4}{5,2,2}{,}
the $\beta$-number is then $(4,0,-1,-4,-5,\dots)$ and this corresponds to the $4$-core $\lambda = (5,2,2)$.
The multiset of residues is \ytableaushort{01230,30,23}{} and the number of $0$-nodes is  $3 = \frac{1}{2}(2^2+1^2+1^2+0^2)$.
\end{exemple}

\begin{exemple}
We take $p = e = 3$. We consider the parameter $\tuplex{x} \coloneqq (1,2, -3) \in \mathbb{Z}^3_0$. The corresponding $3$-abacus is:
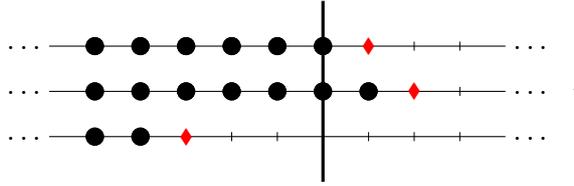
\abacuscore{3}{5,3,3,2,2,1,1}{,}
the $\beta$-number is then $(4,1,0,-2,-3,-5,-6,-8,-9,\dots)$ and this corresponds to the $4$-core $\lambda = (5,3,3,2,2,1,1)$. The multiset of residues is \ytableaushort{01201,201,120,01,20,1,0}{} and the number of $0$-nodes is  $7 = \frac{1}{2}(1^2+2^2+3^2)$.
\end{exemple}

\subsection{Multipartitions}
\label{subsection:multipartitions}

Let $d, \eta, p \in \mathbb{N}^*$ and assume that $e = \eta p$. We define $r \coloneqq dp$ and we identify $\mathbb{Z}/r\mathbb{Z}$ with the set $\{0, \dots, r-1\}$. Let $\kappa = (\kappa_0, \dots, \kappa_{r-1}) \in (\mathbb{Z}/e\mathbb{Z})^r$ be a multicharge.
An \emph{$r$-partition} (or \emph{multipartition}) of $n$ is an $r$-tuple $\tuple{\lambda} = (\lambda^{(0)}, \dots, \lambda^{(r-1)})$ of partitions such that $\lvert \tuple{\lambda} \rvert \coloneqq \lvert \lambda^{(0)}\rvert + \dots + \lvert \lambda^{(r-1)} \rvert = n$. We write $\tuple{\lambda} \in \P_n^\kappa$ if $\tuple{\lambda}$ is an $r$-partition of $n$.
We say that $\kappa$ is \emph{compatible} with $(d, \eta, p)$ when
\begin{equation}
\label{equation:kappa_compatible}
\kappa_{k + d} = \kappa_k + \eta, \qquad \text{for all } k  \in \mathbb{Z}/r\mathbb{Z}.
\end{equation}
Thus, the multicharge $\kappa$ is compatible with $(d, \eta, p)$ if and only if
\begin{equation}
\label{equation:forme_kappa}
\kappa = \bigl(\kappa_0, \dots, \kappa_{d-1}, \kappa_0 + \eta, \dots, \kappa_{d-1} + \eta, \dots\dots, \kappa_0 + (p-1)\eta, \dots, \kappa_{d-1} + (p-1)\eta\bigr).
\end{equation}

\begin{exemple}
\label{exemple:kappa_compatible}
If $d = 1$ and $\eta = p = 2$ (thus $e = 4$ and $r = 2$), the multicharge $\kappa \coloneqq (0, 2) \in (\mathbb{Z}/4\mathbb{Z})^2$ is compatible with $(d, \eta, p)$.
\end{exemple}

The Young diagram of an $r$-partition $\tuple{\lambda} = (\lambda^{(0)}, \dots, \lambda^{(r-1)})$ is the subset of $\mathbb{N}^3$ defined by
\[
\Y(\tuple{\lambda}) \coloneqq \bigcup_{c = 0}^{r-1} \left(\Y(\lambda^{(c)}) \times \{c\}\right).
\]
A \emph{node} is any element of $\mathbb{N} \times \mathbb{N} \times \{0, \dots, r-1\}$, for instance, any element of $\Y(\tuple{\lambda})$ is a node.
The \emph{$\kappa$-residue} of a node $\gamma = (a, b, c)$ is $\res_\kappa(\gamma) \coloneqq b - a + \kappa_c \pmod{e}$. For any $i \in \mathbb{Z}/e\mathbb{Z}$, we denote by $n^i_\kappa(\tuple{\lambda})$ its multiplicity in the multiset of $\kappa$-residues of all elements of $\Y(\tuple{\lambda})$. 
We also define
\[
\alpha_\kappa(\tuple{\lambda}) \coloneqq \sum_{\gamma \in \Y(\tuple{\lambda})} \alpha_{\res_\kappa(\gamma)} = \sum_{i = 0}^{e-1} n^i_\kappa(\tuple{\lambda})\alpha_i \in Q^+.
\]
We have $\lvert \alpha_\kappa(\tuple{\lambda})\rvert = \lvert\tuple{\lambda}\rvert$. These quantities $\alpha_\kappa(\tuple{\lambda})$ were used by Lyle and Mathas~\cite{lyle-mathas} to parametrise the blocks of $\H_n^\kappa$. More precisely, the \emph{cellularity} of the algebra $\H_n^\kappa$ allows to associate to each $r$-partition~$\tuple{\lambda}$ of~$n$ a block of $\H_n^\kappa$ (see Lemma~\ref{lemma:composition_factors} and~\textsection\ref{subsubsection:cellularity_AK}): the main result of~\cite{lyle-mathas} is that an $r$-partition~$\tuple{\mu}$ of~$n$ has the same associated block as~$\tuple{\lambda}$ if and only if $\alpha_\kappa(\tuple{\mu}) = \alpha_\kappa(\tuple{\lambda})$. In this case, we say that $\tuple{\lambda}$ and $\tuple{\mu}$ \emph{belong} to the same block of $\H_n^\kappa$.

Finally, if $\tuple{\lambda} = (\lambda^{(0)}, \dots, \lambda^{(r-1)})$ is an $r$-partition, its \emph{$e$-multicore} is the $r$-partition $\overline{\tuple{\lambda}} \coloneqq \bigl(\overline{\lambda^{(0)}}, \dots, \overline{\lambda^{(r-1)}}\bigr)$. We say that $\tuple{\lambda}$ is an \emph{$e$-multicore} if $\tuple{\lambda} = \overline{\tuple{\lambda}}$, that is, if each $\lambda^{(k)}$ for $k \in \{0, \dots, r-1\}$ is an $e$-core, in which case we write
\[
x^{(k)}(\tuple{\lambda}) \coloneqq x\bigl(\lambda^{(k)}\bigr) \in \mathbb{Z}^e_0,
\]
for the parameter of the $e$-abacus associated with the $e$-core $\lambda^{(k)}$. We write $x^{(k)}(\tuple{\lambda}) = \bigl(x^{(k)}_0(\tuple{\lambda}), \dots, x^{(k)}_{e-1}(\tuple{\lambda})\bigr)$, so that $x^{(k)}_i(\tuple{\lambda}) \coloneqq x_i\bigl(\lambda^{(k)}\bigr)$ for any $i \in \{0, \dots, e-1\}$.

\begin{remarque}
For ordinary partitions, which are $1$-partitions, we recover the definitions of \textsection\ref{subsection:partitions} if $\kappa = 0$. In particular, if $\lambda$ is a partition then $n^i(\lambda) = n^i_0(\lambda)$ for all $i \in \{0, \dots, e-1\}$ and $\alpha(\lambda) = \alpha_0(\lambda)$. Moreover, if $\lambda$ is an $e$-core then $x^{(0)}(\lambda) = x(\lambda)$.
\end{remarque}

\begin{remarque}
Contrary to~\cite{lyle-mathas}, we do not shift by $\kappa_k$ the definition of the $\beta$-number (and thus the $e$-abacus) for $\lambda^{(k)}$. While the convention used by~\cite{lyle-mathas} behaves well towards adding and removing $i$-nodes (more generally, concerning Lemma~\ref{lemme:transfert_bille}), the one we use is more adapted to detect when two components of $\tuple{\lambda}$ are equal (see~\textsection\ref{subsection:shifts}).
\end{remarque}

The next lemma is straightforward. 
\begin{lemme}
\label{lemma:partition_comb-n^i-delta}
Let $\lambda$ be a partition and $i, \delta, \kappa_* \in \mathbb{Z}/e\mathbb{Z}$.  We have
\[
n^i_{\kappa_* + \delta}(\lambda) = n^{i-\delta}_{\kappa_*}(\lambda).
\]
\end{lemme}



We now give a generalisation of Lemma~\ref{lemme:difference_ni_xi} and Proposition~\ref{proposition:residu_0_chapeau} in the setting of multipartitions.
Recall that we identify $\{0, \dots, e-1\}$ (respectively $\{0, \dots, r-1\}$) with $\mathbb{Z}/e\mathbb{Z}$ (resp. $\mathbb{Z}/r\mathbb{Z}$).
\begin{lemme}
\label{lemme:difference_ni_xi_multipart}
Let $\tuple{\lambda}$ be an $e$-multicore.  For all $i \in \{0, \dots, e-1\}$ we have
\[
n^i_\kappa(\tuple{\lambda}) - n^{i+1}_\kappa(\tuple{\lambda}) = \sum_{k = 0}^{r-1} x^{(k)}_{i-\kappa_k}(\tuple{\lambda})
.\]
\end{lemme}

\begin{proof}
Write $\tuple{\lambda} = \bigl(\lambda^{(0)}, \dots, \lambda^{(r-1)}\bigr)$ and let $i \in \{0, \dots, e-1\}$.
By Lemmas~\ref{lemme:difference_ni_xi} and~\ref{lemma:partition_comb-n^i-delta} we have
\begin{align*}
 n^i_\kappa(\tuple{\lambda}) - n^{i+1}_\kappa(\tuple{\lambda})
&=
\sum_{k = 0}^{r-1}
\left[n^i_{\kappa_k}(\lambda^{(k)}) - n^{i+1}_{\kappa_k}(\lambda^{(k)})\right]
\\
&=
\sum_{k = 0}^{r-1}\left[n^{i-\kappa_k}(\lambda^{(k)}) - n^{i+1-\kappa_k}(\lambda^{(k)})\right]
\\
&=
\sum_{k = 0}^{r-1} x_{i - \kappa_k}(\lambda^{(k)}).
\\
&=
\sum_{k = 0}^{r-1} x^{(k)}_{i - \kappa_k}(\tuple{\lambda}).
\end{align*}
\end{proof}

Finally, for any $i \in \{0, \dots, e-1\}$  define $L_i (\tuplex{x}) \coloneqq \sum_{i' = 0}^{i-1} x_{i'}$ for all $\tuplex{x} \in \mathbb{Z}^e$. By Corollary~\ref{corollaire:ni_no_sum_xj}, if $\tuple{\lambda} = \bigl(\lambda^{(0)}, \dots, \lambda^{(r-1)}\bigr)$ is an $e$-multicore we have
\[
n^0_\kappa(\tuple{\lambda})
=
\sum_{k = 0}^{r-1} n^0_{\kappa_k}(\lambda^{(k)})
=
\sum_{k= 0}^{r-1} n^{-\kappa_k}(\lambda^{(k)})
=
\sum_{k = 0}^{r-1} \left[n^0(\lambda^{(k)}) - L_{-\kappa_k}\bigl(\tuplex{x}^{(k)}(\tuple{\lambda})\bigr)\right].
\]
Hence, by Proposition~\ref{proposition:residu_0_chapeau},
\begin{equation}
\label{equation:multipartitions-L_i}
n^0_\kappa(\tuple{\lambda}) = \sum_{k = 0}^{r-1} \left[\frac{1}{2}\lVert\tuplex{x}^{(k)}(\tuple{\lambda})\rVert^2  - L_{-\kappa_k}\bigl(\tuplex{x}^{(k)}(\tuple{\lambda})\bigr)\right].
\end{equation}

\subsection{Shifts}
\label{subsection:shifts}

We are now ready to define our two shift maps.

\begin{definition}
\label{definition:shift_alpha}
Recall that $e$ is determined by $\eta$ and $p$.
For any $i \in \mathbb{Z}/e\mathbb{Z}$ we define $\sigma_{\eta, p}\cdot \alpha_i \coloneqq \alpha_{i + \eta} \in Q^+$, and we extend $\sigma_{\eta,p}$ to a $\mathbb{Z}$-linear map $Q \to Q$.
\end{definition}

\begin{definition}
\label{definition:shift_lambda}
Recall that $r$ is determined by $d$ and $p$.
If $\tuple{\lambda} = (\lambda^{(0)}, \dots, \lambda^{(r-1)})$ is an $r$-partition, we define
\[
\prescript{\sigma_{d, p}}{}{\tuple{\lambda}} \coloneqq (\lambda^{(r-d)}, \dots,  \lambda^{(r-1)}, \lambda^{(0)}, \dots, \lambda^{(r-d-1)}).
\]
\end{definition}

Note that $\sigma_{\eta, p}$ and $\sigma_{d, p}$ are the identity maps when $p = 1$ (thus $\eta = e$ and $d = r$).
For any $\alpha \in Q^+$, we denote by $\P_\alpha^\kappa$ the subset of $\P_n^\kappa$ given by $r$-partitions $\tuple{\lambda}$ such that $\alpha_\kappa(\tuple{\lambda}) = \alpha$.
The two shifts of Definitions~\ref{definition:shift_alpha} and~\ref{definition:shift_lambda} are compatible in the following way.

\begin{lemme}
\label{lemme:sigma_alpha_lambda}
Assume that the multicharge $\kappa$ is compatible with $(d, \eta, p)$.
If $\tuple{\lambda}$ is an $r$-partition then
$\alpha_\kappa(\prescript{\sigma_{d,p}}{}{\tuple{\lambda}}) = \sigma_{\eta, p} \cdot \alpha_\kappa(\tuple{\lambda})$. In other words, the map $\sigma_{d, p}$ induces a bijection between $\P_\alpha^\kappa$ and $\P_{\sigma_{\eta, p} \cdot\alpha}^\kappa$.
\end{lemme}

\begin{proof}
Recall that we are identifying $\mathbb{Z}/e\mathbb{Z}$ (respectively $\mathbb{Z}/r\mathbb{Z}$) with $\{0, \dots, e-1\}$ (resp. $\{0, \dots, r-1\}$). We write $\tuple{\lambda} = \bigl(\lambda^{(0)}, \dots, \lambda^{(r-1)}\bigr)$. Using the compatibility equation \eqref{equation:kappa_compatible} for the multicharge $\kappa$ and Lemma~\ref{lemma:partition_comb-n^i-delta} we have
\begin{align*}
\alpha_\kappa(\prescript{\sigma_{d, p}}{}{\tuple{\lambda}})
&= \alpha_\kappa\bigl(\lambda^{(r-d)}, \dots,  \lambda^{(r-1)}, \lambda^{(0)}, \dots, \lambda^{(r-d-1)}\bigr)
\\
&= \sum_{i = 0}^{e-1} n^i_\kappa\bigl(\lambda^{(r-d)}, \dots,  \lambda^{(r-1)}, \lambda^{(0)}, \dots, \lambda^{(r-d-1)}\bigr) \alpha_i
\\
&= \sum_{i = 0}^{e-1} \sum_{k = 0}^{r-1} n^i_{\kappa_k} \bigl(\lambda^{(r-d + k)}\bigr) \alpha_i
\\
&= \sum_{i = 0}^{e-1} \sum_{k = 0}^{r-1} n^i_{\kappa_k} \bigl(\lambda^{(k-d)}\bigr) \alpha_i
\\
&=\sum_{i = 0}^{e-1} \sum_{k = 0}^{r-1} n^i_{\kappa_{k + d}}\bigl(\lambda^{(k)}\bigr) \alpha_i
\\
&= \sum_{i= 0}^{e-1} \sum_{k = 0}^{r-1} n^i_{\kappa_k +\eta}\bigl(\lambda^{(k)}\bigr)\alpha_i
\\
&= \sum_{i= 0}^{e-1} \sum_{k = 0}^{r-1} n^{i-\eta}_{\kappa_k}\bigl(\lambda^{(k)}\bigr)\alpha_i
\\
&=  \sum_{i= 0}^{e-1} \sum_{k = 0}^{r-1} n^i_{\kappa_k}\bigl(\lambda^{(k)}\bigr)\alpha_{i +\eta}
\\
&= \sigma_{\eta, p} \cdot \sum_{i= 0}^{e-1} \sum_{k = 0}^{r-1} n^i_{\kappa_k}\bigl(\lambda^{(k)}\bigr)\alpha_i
\\
&= \sigma_{\eta, p} \cdot \sum_{i= 0}^{e-1}  n^i_\kappa\bigl(\lambda^{(0)}, \dots, \lambda^{(r-1)}\bigr)\alpha_i
\\
&= \sigma_{\eta, p}\cdot \alpha(\tuple{\lambda}),
\end{align*}
as desired. The second statement follows.
\end{proof}

\begin{lemme}
\label{lemme:sigma_puiss}
Let $p'$ be an integer that divides $p$ and assume that the multicharge~$\kappa$ is compatible with $(d, \eta, p)$. Then~$\kappa$ is compatible with $\bigl(p'd, p'\eta, \frac{p}{p'}\bigr)$ and
\begin{align*}
\sigma_{\eta, p}^{p'} &= \sigma_{p'\eta, \frac{p}{p'}}, &&\text{in } Q^+,
\\
\sigma_{d, p}^{p'} &= \sigma_{p'd, \frac{p}{p'}}, &&\text{on } r\text{-partitions}.
\end{align*}
\end{lemme}

\begin{proof}
 We have $r = dp = (p'd)\frac{p}{p'}$ and $e = p\eta = \frac{p}{p'}(p'\eta)$. Since~$\kappa$ is compatible with $(d, \eta, p)$ we have $\kappa_{i + d} = \kappa_i + \eta$ for all $i \in \mathbb{Z}/r\mathbb{Z}$ thus $\kappa_{i+p'd} = \kappa_i+p'\eta$, which means that~$\kappa$ is compatible with $\bigl(p'd, p'\eta, \frac{p}{p'}\bigr)$. We conclude applying Definitions~\ref{definition:shift_alpha} and~\ref{definition:shift_lambda}.
\end{proof}

We can now state the main theorem of the paper, which will be proved in Section~\ref{section:demo}.

\begin{theoreme}
\label{theoreme:orbite_1}
Let $\tuple{\lambda}$ be an $r$-partition and let $\alpha \coloneqq \alpha_\kappa(\tuple{\lambda}) \in Q^+$. Assume that $\kappa$ is compatible with  $(d, \eta, p)$.
If $\sigma_{\eta, p} \cdot \alpha = \alpha$ then there is an $r$-partition $\tuple{\mu} \in \P_\alpha^\kappa$ with $\prescript{\sigma_{d, p}}{}{\tuple{\mu}} = \tuple{\mu}$.
\end{theoreme}

We say that an $r$-partition $\tuple{\mu}$ as in Theorem~\ref{theoreme:orbite_1} is  \emph{stuttering}. Note that when $p = 1$, all $r$-partitions are stuttering. We will often drop the subscripts and only write $\sigma$ for $\sigma_{d, p}$ and $\sigma_{\eta, p}$ when the meaning is clear from the context.

\begin{exemple}
We consider the setting of Example~\ref{exemple:kappa_compatible} and the bipartition $\tuple{\lambda} \coloneqq \bigl((5,2,1),(1,1)\bigr)$. The multiset of $\kappa$-residues is
\[
\ytableaushort{01230,30,2}
\quad
\ytableaushort{2,1}\, ,
\]
thus $\alpha_\kappa(\tuple{\lambda}) = 3(\alpha_0 + \alpha_2) + 2(\alpha_1 + \alpha_3) \eqqcolon \alpha$. Hence, we have $\sigma \cdot \alpha = \alpha$ but $\prescript{\sigma}{}{\tuple{\lambda}} = \bigl((1,1),(5,2,1)\bigr) \neq \tuple{\lambda}$. We now consider  the partition $\mu \coloneqq (3,1,1)$. The residue multiset of the bipartition $(\mu, \mu)$ is
\[
\ytableaushort{012,3,2} \quad \ytableaushort{230,1,0}\, ,
\]
thus $\alpha_\kappa(\mu, \mu) = 3(\alpha_0 + \alpha_2) + 2(\alpha_1 + \alpha_3) = \alpha$. Hence, the stuttering bipartition $(\mu, \mu)$ is as in Theorem~\ref{theoreme:orbite_1}.
\end{exemple}

\begin{remarque}
\label{remarque:particular_cases}
Two particular cases of Theorem~\ref{theoreme:orbite_1} easily follow  from Lemma~\ref{lemme:sigma_alpha_lambda}. Let $\tuple{\lambda}$ be an $r$-partition and let $\alpha \coloneqq \alpha_\kappa(\tuple{\lambda})$.
\begin{enumerate}[$(i)$]
\item
\label{item-remarque:sigma_alpha(nu,nu)}
If $\prescript{\sigma}{}{\tuple{\lambda}} = \tuple{\lambda}$ then $\sigma \cdot \alpha = \alpha$ and there is nothing to prove.

\item If $\sigma \cdot \alpha = \alpha$ and if $\tuple{\lambda}$ is the only $r$-partition in $\P_\alpha^\kappa$ (\textit{e.g.} when the associated Ariki--Koike algebra is semisimple, see \cite{Ar_ssimple}) then $\prescript{\sigma}{}{\tuple{\lambda}} \in \P_{\sigma\cdot \alpha}^\kappa = \P_\alpha^\kappa$ thus $\prescript{\sigma}{}{\tuple{\lambda}} = \tuple{\lambda}$.
\end{enumerate}
\end{remarque}

Let us denote by $[\tuple{\lambda}]$ (respectively by  $[\alpha]$) the orbit of an $r$-partition $\tuple{\lambda}$ (resp. of $\alpha \in Q^+$) under the action of $\sigma$. We now state Theorem~\ref{theorem:introduction_min_block} from the introduction.

\begin{corollaire}
\label{corollaire:orbite_gnrl}
Assume that $\kappa$ is compatible with $(d, \eta, p$) and let $\alpha \in Q^+$ such that $\P_\alpha^\kappa$ is not empty. Then $\#[\alpha]$ is the smallest element of the set $\left\{\#[\tuple{\lambda}] : \tuple{\lambda} \in \P_\alpha^\kappa\right\}$. In other words, if $\tuple{\lambda}$ is an $r$-partition and $\alpha \coloneqq \alpha_\kappa(\tuple{\lambda})$, if $\sigma^j \cdot \alpha = \alpha$ for some $j \in \{0, \dots, p-1\}$ then there exists an $r$-partition $\tuple{\mu}$ such that $\alpha_\kappa(\tuple{\mu}) = \alpha$ and $\prescript{\sigma^j}{}{\tuple{\mu}} = \tuple{\mu}$.
\end{corollaire}

\begin{proof}
The second part of the statement is clear.
Let $\mathcal{C}$ be the set $\{\#[\tuple{\lambda}] : \tuple{\lambda} \in \P_\alpha^\kappa\}$ and let us prove that $\#[\alpha]$ is the smallest element of $\mathcal{C}$.
For each $\tuple{\lambda} \in \P_\alpha^\kappa$, by Lemma~\ref{lemme:sigma_alpha_lambda} we have $\alpha_\kappa(\prescript{\sigma}{}{\tuple{\lambda}}) = \sigma \cdot \alpha_\kappa(\tuple{\lambda})$ thus $\#[\tuple{\lambda}] \geq \#[\alpha]$, hence $\#[\alpha]$ is a lower bound of $\mathcal{C}$. To prove that it is the smallest element, it suffices to prove that there is an $r$-partition $\tuple{\mu} \in \P_\alpha^\kappa$ such that $\#[\tuple{\mu}] \leq \#[\alpha]$. Write $p' \coloneqq \#[\alpha]$. The integer $p'$ divides $p$ since $\sigma$ has order $p$. By Lemma~\ref{lemme:sigma_puiss}, we know that $\kappa$ is compatible with $(p'd, p'\eta, \frac{p}{p'})$. Moreover, we have $\sigma_{\eta, p}^{p'} \cdot \alpha = \alpha$ thus Lemma~\ref{lemme:sigma_puiss} also gives
\[
\sigma_{p' \eta, \frac{p}{p'}} \cdot \alpha = \alpha.
\]
Hence, by Theorem~\ref{theoreme:orbite_1} applied with $(p'd, p'\eta, \frac{p}{p'})$ we know that there is an $r$-partition $\tuple{\mu} \in \P_\alpha^\kappa$ such that
\[
\prescript{\sigma_{dp', \frac{p}{p'}}}{}{\tuple{\mu}} = \tuple{\mu},
\]
that is, by another application of Lemma~\ref{lemme:sigma_puiss},
\[
\prescript{\sigma^{p'}_{d, p}}{}{\tuple{\mu}} = \tuple{\mu}.
\]
Hence, we have $\#[\tuple{\mu}] \leq p'$ and we conclude since $p' = \#[\alpha]$.
\end{proof}

\begin{remarque}
\label{remark:main_thm}
We saw that Lyle and Mathas~\cite{lyle-mathas} proved that two $r$-partitions are in a same $\P^\kappa_\alpha$ if and only if they belong to the same block of $\H_n^\kappa$. Thus, Corollary~\ref{corollaire:orbite_gnrl} gives a little information about the $r$-partitions that belong to each block.
As we mentioned in the introduction, Wada~\cite{wada} proved that the maximum of the set $\{\#[\tuple{\lambda}] : \tuple{\lambda} \in \P_\alpha^\kappa\}$ of Corollary~\ref{corollaire:orbite_gnrl} is always $p$, provided that this set has at least two elements. His proof is very short but relies on the (non trivial) fact that if $\tuple{\lambda}$ and $\tuple{\mu}$ are in~$\P_\alpha^\kappa$ then they are \emph{Jantzen equivalent} (cf. \cite{lyle-mathas}). On the contrary, we did not find a way to use~\cite{lyle-mathas} to prove Theorem~\ref{theoreme:orbite_1}.
\end{remarque}

We conclude this section by a reduction step for our main theorem.  We assume that the multicharge $\kappa$ is compatible with $(d, \eta, p)$.
For any $\ell \in \{0, \dots, d-1\}$, we define the multicharge $\kappa^{(\ell)} \in (\mathbb{Z}/e\mathbb{Z})^p$ by
\begin{equation}
\label{equation:definition_kappa_ell}
\kappa^{(\ell)} \coloneqq
(\kappa_\ell, \kappa_{\ell+d}, \dots, \kappa_{\ell + (p-1)d}) = (\kappa_\ell, \kappa_\ell + \eta, \dots, \kappa_\ell + (p-1)\eta).
\end{equation}
We first need the following lemma.

\begin{lemme}
\label{lemme:ajout_rimhook}
Let $\ell \in \{0, \dots, d-1\}$, let $\lambda$ be a partition and let $\mu$ be a partition obtained from $\lambda$ by wrapping on an $\eta$-rim hook. We define the two $p$-partitions $\lambda^p$ and $\mu^p$ by $\lambda^p \coloneqq (\lambda, \dots, \lambda)$ and $\mu^p \coloneqq (\mu, \dots, \mu)$. If $\alpha \coloneqq \alpha_{\kappa^{(\ell)}}(\lambda^p)$ and $\beta \coloneqq \alpha_{\kappa^{(\ell)}}(\mu^p)$ then  $\beta = \alpha + \alpha_0 + \dots + \alpha_{e-1}$.
\end{lemme}

\begin{proof}
By Remark~\ref{remarque:racine_ruban}, we have $\alpha_{\kappa_\ell}(\mu) = \alpha_{\kappa_\ell}(\lambda) + \alpha_{i_0} + \dots + \alpha_{i_0 + \eta - 1}$ for some $i_0 \in \mathbb{Z}/e\mathbb{Z}$. Thus, for any $j \in \{0, \dots, p-1\}$ we have
\begin{align*}
\alpha_{\kappa_\ell + j\eta}(\mu)
&= \sigma^j \cdot \alpha_{\kappa_\ell}(\mu)
\\
&= \sigma^j \cdot \alpha_{\kappa_\ell}(\lambda) +  \sum_{i = 0}^{\eta-1} \sigma^j \cdot \alpha_{i_0 + i}
\\
&= \alpha_{\kappa_\ell + j\eta}(\lambda) +  \sum_{i = 0}^{\eta-1} \alpha_{i_0 + i+j\eta}.
\end{align*}
We obtain
\begin{align*}
\beta &= \alpha_{\kappa^{(\ell)}}(\mu^p)
\\
&= \sum_{j = 0}^{p-1} \alpha_{\kappa_\ell + j\eta}(\mu)
\\
& = \sum_{j = 0}^{p-1} \alpha_{\kappa_\ell + j\eta}(\lambda) + \sum_{j = 0}^{p-1} \sum_{i = 0}^{\eta-1} \alpha_{i_0 + i + j\eta}
\\
&= \alpha_{\kappa^{(\ell)}}(\lambda^p) + \alpha_0 + \dots + \alpha_{e-1}
\\
&= \alpha + \alpha_0 + \dots + \alpha_{e-1}.
\end{align*}
\end{proof}


\begin{proposition}
\label{proposition:reduction}
It suffices to prove Theorem~\ref{theoreme:orbite_1} for the $e$-multicores.
\end{proposition}

\begin{proof}
Let $\tuple{\lambda}$ be an $r$-partition such that $\sigma \cdot \alpha_\kappa(\tuple{\lambda}) = \alpha_\kappa(\tuple{\lambda})$ and let $\overline{\tuple{\lambda}}$ be its $e$-multicore.
By definition of the $e$-multicore and by Remark~\ref{remarque:racine_ruban}, we have $\alpha_\kappa(\tuple{\lambda}) = \alpha_\kappa(\overline{\tuple{\lambda}}) + w\sum_{i = 0}^{e-1} \alpha_i$ where $w \in \mathbb{N}$ is the number of $e$-rim hooks that we need to wrap on to obtain $\tuple{\lambda}$ from $\overline{\tuple{\lambda}}$.
Since $\alpha_\kappa(\tuple{\lambda})$ and $\sum_{i = 0}^{e-1} \alpha_i$ are both stable by $\sigma$, we have $ \sigma \cdot \alpha_\kappa(\overline{\tuple{\lambda}}) = \alpha_\kappa(\overline{\tuple{\lambda}})$. If Theorem~\ref{theoreme:orbite_1} is true for the $e$-multicore $\overline{\tuple{\lambda}}$, we can find a stuttering $r$-partition $\widetilde{\tuple{\mu}} = \prescript{\sigma}{}{\widetilde{\tuple{\mu}}}$ with $\alpha_\kappa(\widetilde{\tuple{\mu}}) = \alpha_\kappa(\overline{\tuple{\lambda}})$. Write $\widetilde{\tuple{\mu}} = (\widetilde{\mu}^{(0)}, \dots, \widetilde{\mu}^{(r-1)})$ and let $\mu^{(0)}$ be a partition obtained by wrapping on $w$ times an $\eta$-rim hook to $\widetilde{\mu}^{(0)}$. We define
\begin{align*}
\mu^{(jd)} &\coloneqq \mu^{(0)}, \qquad \text{for all } j \in \{1, \dots, p-1\},
\\
\mu^{(k)} &\coloneqq \widetilde{\mu}^{(k)}, \qquad \text{for all } k \in \{0, \dots, r-1\} \setminus \{0, d, \dots, (p-1)d\}.
\end{align*}
The $r$-partition $\tuple{\mu} \coloneqq (\mu^{(0)}, \dots, \mu^{(r-1)})$ satisfies $\tuple{\mu} = \prescript{\sigma}{}{\tuple{\mu}}$. Moreover, since $\widetilde{\mu}^{(0)} = \widetilde{\mu}^{(jd)}$ for all $j \in \{1, \dots, p-1\}$, we can apply $w$ times Lemma~\ref{lemme:ajout_rimhook} with $\ell \coloneqq 0$ starting from the $p$-partition
$\bigl(\widetilde{\mu}^{(0)}, \dots, \widetilde{\mu}^{(0)}\bigr)$. We obtain
\begin{align*}
\alpha_\kappa(\tuple{\mu})
&=
\alpha_{\kappa^{(0)}}\bigl(\mu^{(0)}, \dots, \mu^{(0)}\bigr) + \sum_{\ell = 1}^{d-1} \alpha_{\kappa^{(\ell)}}\bigl(\mu^{(\ell)}, \dots, \mu^{(\ell)}\bigr)
\\
&=
\alpha_{\kappa^{(0)}}\bigl(\widetilde{\mu}^{(0)}, \dots, \widetilde{\mu}^{(0)}\bigr) + w\sum_{i = 0}^{e-1}\alpha_i + \sum_{\ell = 1}^{d-1} \alpha_{\kappa^{(\ell)}}\bigl(\widetilde{\mu}^{(\ell)}, \dots, \widetilde{\mu}^{(\ell)}\bigr)
\\
&= 
\alpha_\kappa(\widetilde{\tuple{\mu}}) + w\sum_{i = 0}^{e-1}\alpha_i
\\
&=
\alpha_\kappa(\overline{\tuple{\lambda}}) + w\sum_{i = 0}^{e-1}\alpha_i
\\
&= \alpha_\kappa(\tuple{\lambda}).
\end{align*}
Hence, Theorem~\ref{theoreme:orbite_1} is proved for $\tuple{\lambda}$.
\end{proof}

\begin{remarque}
\label{remark:add_rim_hooks}
Since the $\eta$-rim hooks that we wrap on are arbitrary, the stuttering $r$-partition in Theorem~\ref{theoreme:orbite_1} is not unique in general. Moreover, using the same idea of wrapping on $\eta$-rim hooks we can easily prove Theorem~\ref{theoreme:orbite_1} in the particular case $\eta = 1$ (that is, $p = e$). Finally, if $\tuple{\lambda}$ and $\tuple{\mu}$ are as in Theorem~\ref{theoreme:orbite_1} and if $\tuple{\lambda}$ is an $e$-multicore, then $\tuple{\mu}$ is not necessarily an $e$-multicore.
\end{remarque}

\section{Binary tools and inequalities}
\label{section:approx}

In this section, we introduce two technical tools that we will need to prove Theorem~\ref{theoreme:orbite_1}. In~\textsection\ref{subsection:binary_matrices}, given a family of binary matrices satisfying some conditions, our aim is to prove that we can find a series of \emph{compatible} submatrices $\bigl(\begin{smallmatrix} 1&0\\0&1\end{smallmatrix}\bigr)$. We will need to study some particular cases (Lemma~\ref{lemme:interversion_bloc2} and Proposition~\ref{proposition:interversions_somme_blocs}) before stating the main result, Corollary~\ref{corollaire:interversions_somme_blocs}. We use this result to prove in~\textsection\ref{subsection:binary_averaging} the existence of a binary matrix with prescribed row, (partial) column and block sums. Finally, we will give \textsection\ref{subsection:inegalites} some inequalities. The first one will be reminiscent of the binary setting, and the others will use convexity.

We use $\lvert \cdot \rvert : \mathbb{R}^n \to \mathbb{R}$ to denote the sum of the coordinates (we warn the reader that we do not take the sum of the absolute values) and we write $\lVert\cdot\rVert$ for the Euclidean norm.

\subsection{Binary matrices}
\label{subsection:binary_matrices}

Given two matrices with entries in $\{0, 1\}$ whose row sums (respectively column sums) are pairwise equal, we can get from the one to the other by replacing submatrices $\bigl(\begin{smallmatrix}
1 & 0 \\ 0 & 1
\end{smallmatrix}\bigr)$ by $\bigl(\begin{smallmatrix}
0 & 1 \\ 1 & 0
\end{smallmatrix}\bigr)$ (cf. \cite{ryser}).
These interchanges do not change the row or column sums, however they may change block sums.  The results of this section, particularly Corollary~\ref{corollaire:interversion}, will be used to prove  Proposition~\ref{proposition:elt_base_canonique_bloc} in \textsection\ref{subsection:binary_averaging}, where we show the existence of a binary matrix with prescribed row, column and block sums. Note that Chernyak and Chernyak~\cite{chch} considered matrices with prescribed row, column and block sums, but they did not study the existence problem.

We call \emph{binary matrix} a matrix with entries in $\{0, 1\}$. If $M$ is an $m \times n$ binary matrix, we write $M_{\ell k}$ for its entry at $(\ell, k) \in \{1, \dots, m\} \times \{1, \dots, n\}$. We denote by $\gamma_{\ell, k}(M)$ the binary matrix that we obtain from $M$ by changing the entry $(\ell, k)$ to $1 - M_{\ell k}$.
 We write $R_\ell(M)$ for the $\ell$th row of $M$.   Note that if $\lvert M \rvert$ denotes the sum of the entries of $M$ then $\lvert M \rvert = \sum_\ell \lvert R_\ell(M)\rvert$. Finally, if the number of rows of $M$ is even, we will systematically write $M = \begin{pmatrix}
M^+ \\ M^-
\end{pmatrix}$ where $M^+$ and $M^-$ have the same size, and we define $\gamma_{\ell, k}^\pm(M) \coloneqq \begin{pmatrix}
\gamma_{\ell, k}(M^+)
\\
\gamma_{\ell, k}(M^-)
\end{pmatrix}
$.


\begin{definition}
\label{definition:interchange}
Let $A = \begin{pmatrix}
A^+ \\ A^-
\end{pmatrix}$ and $B = \begin{pmatrix}
B^+ \\ B^-
\end{pmatrix}$ be two binary matrices with the same even number of rows.
We say that the matrix $\bigl(\begin{smallmatrix}
1 & 0 \\ 0 & 1
\end{smallmatrix}\bigr)$ is a \emph{compatible} submatrix of $\bigl(\begin{array}{c|c}A & B\end{array}\bigr)$  if there exist $\ell, k, k'$ such that
\begin{align*}
A^+_{\ell k} &= 1, & B^+_{\ell k'} &= 0,
\\
A^-_{\ell k} &= 0, & B^-_{\ell k'} &= 1.
\end{align*}
In that case, we will write $A \compat_{\ell, k, k'} B$. We denote by $\gamma_{\ell, k, k'}(A, B) \coloneqq \bigl(\gamma_{\ell, k}^\pm(A), \gamma^\pm_{\ell, k'}(B)\bigr)$ the pair of binary matrices that we obtain if we replace the submatrix $\bigl(\begin{smallmatrix}
1 & 0 \\ 0 & 1
\end{smallmatrix}\bigr)$ by $\bigl(\begin{smallmatrix}
0 & 1 \\ 1 & 0
\end{smallmatrix}\bigr)$.
\end{definition}


%

\begin{exemple}
We consider the binary matrices $A = \begin{pmatrix}
A^+ \\ A^-
\end{pmatrix}$ and $B = \begin{pmatrix}
B^+ \\ B^-
\end{pmatrix}$ defined by
\begin{align*}
A^+ &\coloneqq \begin{pmatrix}
1 & \color{red}{1} \\ 0 & 0
\end{pmatrix},
&
B^+ &\coloneqq \begin{pmatrix}
1 & 0&\color{red}{0}
\\
0 & 1&0
\end{pmatrix},
\\
A^- &\coloneqq \begin{pmatrix}
1 & \color{red}{0}
\\
0 & 1
\end{pmatrix},
&
B^- &\coloneqq \begin{pmatrix}
1 & 0&\color{red}{1} \\ 1 &0& 1
\end{pmatrix}.
\end{align*}
The red entries prove that $A \compat_{1,2,3} B$. With $(\widetilde{A} , \widetilde{B}) \coloneqq \gamma_{1,2,3}(A, B)$, we have
\begin{align*}
\widetilde{A}^+ &\coloneqq \begin{pmatrix}
1 & \color{red}{0} \\ 0 & 0
\end{pmatrix},
&
\widetilde{B}^+ &\coloneqq \begin{pmatrix}
1 & 0&\color{red}{1}
\\
0 & 1&0
\end{pmatrix},
\\
\widetilde{A}^- &\coloneqq \begin{pmatrix}
1 & \color{red}{1}
\\
0 & 1
\end{pmatrix},
&
\widetilde{B}^- &\coloneqq \begin{pmatrix}
1 & 0&\color{red}{0} \\ 1 &0& 1
\end{pmatrix}.
\end{align*}
\end{exemple}

If $A$ and $B$ are two binary matrices with the same even number of rows,  the set of all pairs $\gamma_{\ell, k, k'}(A, B)$ where $\ell, k, k'$ are such that $A \compat_{\ell, k, k'} B$ is denoted by $\Gamma(A, B)$. Moreover, we will write $A \compat B$ if the set $\Gamma(A, B)$ is non-empty, that is, if there exist $\ell, k, k'$ such that $A \compat_{\ell, k, k'} B$. 

We can generalise these notations to a family $(A_i)_{1 \leq i \leq n}$  of binary matrices with the same even number of rows. Let $((\ell_i, k_i, k'_i))_{1 \leq i \leq n-1}$ be a family of triples such that
\[
A_i \compat_{\ell_i, k_i, k'_i} A_{i+1},
\]
for all $i \in \{1, \dots, n-1\}$. For any $i \in \{2, \dots, n-1\}$ we have
\[
A_{i-1} \compat_{\ell_{i-1}, k_{i-1}, k'_{i-1}} A_i \compat_{\ell_i, k_i, k'_i} A_{i+1},
\]
thus, according to Definition~\ref{definition:interchange},
\[
(\ell_{i-1}, k'_{i-1}) \neq (\ell_i, k_i).
\]
Hence, for all $i \in \{2, \dots, n-1\}$ we have
\begin{align*}
A_{i-1} &\compat_{\ell_{i-1}, k_{i-1}, k'_{i-1}} \gamma^\pm_{\ell_i, k_i}(A_i), 
\\
\gamma^\pm_{\ell_{i-1}, k'_{i-1}}(A_i) &\compat_{\ell_i, k_i, k'_i} A_{i+1},
\end{align*}
and
\begin{equation}
\label{equation:perform}
\gamma^\pm_{\ell_i, k_i}\bigl(\gamma^\pm_{\ell_{i-1}, k'_{i-1}}(A_i)\bigr) = \gamma^\pm_{\ell_{i-1}, k'_{i-1}}\bigl(\gamma^\pm_{\ell_i, k_i}(A_i)\bigr).
\end{equation}
We denote by $\gamma_{((\ell_i, k_i, k'_i))_{1 \leq i \leq n-1}}\left((A_i)_{1 \leq i \leq n}\right)$ the family $(\widetilde{A}_i)_{1 \leq i \leq n}$ defined by
\begin{align*}
\widetilde{A}_1 &\coloneqq \gamma^\pm_{\ell_1, k_1}(A_1),
\\
\widetilde{A}_i &\coloneqq \gamma^\pm_{\ell_i, k_i}\bigl(\gamma^\pm_{\ell_{i-1}, k'_{i-1}}(A_i)\bigr), \qquad \text{for all } i \in \{2, \dots, n-1\},
\\
\widetilde{A}_n &\coloneqq \gamma^\pm_{\ell_{n-1}, k'_{n-1}}(A_n).
\end{align*}
By~\eqref{equation:perform}, no choice has been made to define  $\widetilde{A}_i$ for $i \in \{2, \dots, n-1\}$.
Finally, we denote by $\Gamma(A_1, \dots, A_n)$ the set of all families $\gamma_{((\ell_i, k_i, k'_i))_{1 \leq i \leq n-1}}\bigl((A_i)_{1 \leq i \leq n}\bigr)$ where $\bigl((\ell_i, k_i, k'_i)\bigr)_{1 \leq i \leq n-1}$ is such that $A_1 \compat_{\ell_1, k_1, k'_1} \cdots \compat_{\ell_{n-1}, k_{n-1}, k'_{n-1}} A_n$, and we will write $A_1 \compat \cdots \compat A_n$ if $\Gamma(A_1, \dots, A_n)$ is non-empty.

The following properties are straightforward from the definition.

\begin{proposition}
\label{proposition:interversion}
Let $A$ and $B$ be two binary matrices with the same even number of rows such that $A \compat_{\ell, k, k'} B$. If $(\widetilde{A}, \widetilde{B}) \coloneqq \gamma_{\ell, k, k'}(A, B)$ then
\begin{align*}
\widetilde{A}^+_{\ell k} &= A^+_{\ell k} - 1,
&
\widetilde{B}^+_{\ell k'} &= B^+_{\ell k'} + 1,
\\
\widetilde{A}^-_{\ell k} &= A^-_{\ell k} + 1,
&
\widetilde{B}^-_{\ell k'} &= B^-_{\ell k'} - 1,
\end{align*}
the other entries being unchanged. Hence, the following equalities are satisfied:
\begin{align*}
 \widetilde{A}^+_{\ell k} + \widetilde{A}^-_{\ell k} &= A^+_{\ell k} + A^-_{\ell k},
&
 \widetilde{B}^+_{\ell k'} + \widetilde{B}^-_{\ell k'} &= B^+_{\ell k'} + B^-_{\ell k'},
\\
R_\ell(\widetilde{A}^+) + R_\ell(\widetilde{B}^+) &= R_\ell(A^+) + R_\ell(B^+),
&
R_\ell(\widetilde{A}^-) + R_\ell(\widetilde{B}^-) &= R_\ell(A^-) + R_\ell(B^-),
\\
\intertext{and}
\lvert \widetilde{A}^+\rvert  &= \lvert A^+\rvert  - 1,
&
\lvert \widetilde{B}^+\rvert  &= \lvert B^+\rvert  + 1,
\\
\lvert \widetilde{A}^-\rvert  &= \lvert A^-\rvert  + 1,
&
\lvert \widetilde{B}^-\rvert  &= \lvert B^-\rvert  - 1.
\end{align*}
As a consequence, if  $A \compat B \compat C$ and $(\widehat{A}, \widehat{B}, \widehat{C}) \in \Gamma(A, B, C)$ then $\lvert \widehat{B}^+\rvert  = \lvert B^+\rvert $ and $\lvert \widehat{B}^-\rvert  = \lvert B^-\rvert $.
\end{proposition}

\begin{corollaire}
\label{corollaire:interversion}
Let $(A_i)_{1 \leq i \leq n}$ be a family of binary matrices with the same even number of rows. Assume that $i_0, \dots, i_s$ are distinct integers such that $A_{i_0} \compat \dots \compat A_{i_s}$ and let $(\widetilde{A}_{i_0}, \dots, \widetilde{A}_{i_s}) \in \Gamma(A_{i_0}, \dots, A_{i_s})$.  Then
\begin{align*}
\lvert \widetilde{A}^+_{i_0} \rvert &= \lvert A^+_{i_0} \rvert - 1,
&
\lvert \widetilde{A}^+_{i_s} \rvert &= \lvert A^+_{i_s} \rvert + 1,
\\
\lvert \widetilde{A}^-_{i_0} \rvert &= \lvert A^-_{i_0} \rvert + 1,
&
\lvert \widetilde{A}^-_{i_s} \rvert &= \lvert A^-_{i_s} \rvert - 1,
\end{align*}
and for all $t \in \{1, \dots, s-1\}$ we have
\begin{align*}
\lvert \widetilde{A}^+_{i_t} \rvert &= \lvert A^+_{i_t}\rvert,
\\
\lvert \widetilde{A}^-_{i_t} \rvert &= \lvert A^-_{i_t}\rvert.
\end{align*}
\end{corollaire}

The following, easy to prove, lemma is very important in the proof of  Proposition~\ref{proposition:elt_base_canonique_bloc}.

\begin{lemme}
\label{lemme:interversion_bloc2}
Let $A$ and $B$ be  two binary matrices with the same even number of rows.
We assume that
\begin{align*}
\lvert R_\ell(A^+)\rvert  + \lvert R_\ell(B^+)\rvert  &= \lvert R_\ell(A^-)\rvert  + \lvert R_\ell(B^-)\rvert , \qquad \text{for all } \ell,
\\
\lvert A^+\rvert  &> \lvert A^-\rvert.
\end{align*}
Then
$A \compat B$.
\end{lemme}

\begin{proof}
Since $\lvert A^+\rvert  > \lvert A^-\rvert $, there is some $\ell$ such that $\lvert R_\ell(A^+)\rvert  > \lvert R_\ell(A^-)\rvert $. Since the matrices have their entries in $\{0, 1\}$, for all $k$ we have
\[
\begin{pmatrix}
A^+_{\ell k} \\ A^-_{\ell k}
\end{pmatrix} \in \left\{ \begin{pmatrix}0 \\ 0\end{pmatrix}, \begin{pmatrix}
1 \\ 0
\end{pmatrix}, \begin{pmatrix}
1 \\ 1
\end{pmatrix},
\begin{pmatrix}
0 \\ 1
\end{pmatrix}\right\}.
\]
Thus, there is some $k$ such that $\begin{pmatrix}
A_{\ell k}^+ \\ A_{\ell k}^-
\end{pmatrix}
= \begin{pmatrix} 1 \\ 0 \end{pmatrix}$.
Moreover, we have
\[
\lvert R_\ell(B^+)\rvert  = \lvert R_\ell(B^-)\rvert  + \bigl(\lvert R_\ell(A^-)\rvert  - \lvert R_\ell(A^+)\rvert \bigr) < \lvert R_\ell(B^-)\rvert .
\]
Again, we deduce that there is some $k'$ such that $\begin{pmatrix} B^+_{\ell k'} \\ B^-_{\ell k'}\end{pmatrix} = \begin{pmatrix} 0 \\ 1 \end{pmatrix}$. Finally, we have
\begin{align*}
A^+_{\ell k} &= 1, & B^+_{\ell k'} &= 0,
\\
 A^-_{\ell k} &= 0, & B^-_{\ell k'} &= 1,
\end{align*}
thus $A \compat B$.
\end{proof}

Let us now give a generalisation of Lemma~\ref{lemme:interversion_bloc2} to an arbitrary number of matrices.

\begin{proposition}
\label{proposition:interversions_somme_blocs}
Let $(A_i)_{1 \leq i \leq n}$ be a family of binary matrices with the same even number of rows. We assume that
\begin{align*}
\sum_{i=1}^n \lvert R_\ell(A^+_i)\rvert  &= \sum_{i=1}^n \lvert R_\ell(A^-_i)\rvert , \qquad \text{for all } \ell,
\\
\lvert A^+_1\rvert  &> \lvert A^-_1\rvert ,
\\
\lvert A^+_i\rvert  &\geq \lvert A^-_i\rvert , \qquad \text{for all } i \in \{2, \dots, n-1\}.
\end{align*}
Then there exists a sequence $1 < i_1, \dots, i_{s-1} < n$ of distinct integers such that
\[
A_1 \compat A_{i_1} \compat A_{i_2} \compat \dots \compat A_{i_{s-1}} \compat  A_n.
\]
\end{proposition}


\begin{proof}
We consider the following binary matrices with an even number of rows:
\[
B_1 \coloneqq \begin{pmatrix}
A_2 & A_3 & \cdots & A_{n-1} & A_n
\end{pmatrix}.
\]
For each $\ell$ we have $\lvert R_\ell(B^+_1)\rvert  = \sum_{i = 2}^n \lvert R_\ell(A_i^+)\rvert$ and $\lvert R_\ell(B^-_1)\rvert  = \sum_{i = 2}^n \lvert R_\ell(A_i^-)\rvert$.  Thus,
\begin{align*}
\lvert R_\ell(A_1^+)\rvert  + \lvert R_\ell(B_1^+)\rvert 
&= \sum_{i = 1}^n \lvert R_\ell(A_i^+)\rvert 
\\
&= \sum_{i = 1}^n \lvert R_\ell(A^-_i)\rvert 
\\
&= \lvert R_\ell(A^-_1)\rvert  + \sum_{i = 2}^n \lvert R_\ell(A^-_i)\rvert 
\\
\lvert R_\ell(A^+_1)\rvert  + \lvert R_\ell(B^+_1)\rvert 
&= \lvert R_\ell(A^-_1)\rvert  + \lvert R_\ell(B^-_1)\rvert .
\end{align*}
Since $\lvert A^+_1\rvert  > \lvert A^-_1\rvert $, we can apply Lemma~\ref{lemme:interversion_bloc2} to the matrices $A$ and $B_1$. Hence, if we define $I_1$ as the set  of integers $i \in \{2, \dots, n\}$ such that $A_1 \compat A_i$, then $I_1$ is not empty. If $n \in I_1$ then the proof is over, and otherwise we start an induction. Assume that for some integer $s$  we have some pairwise disjoint non-empty subsets $I_0 \coloneqq \{1\}, I_1, \dots, I_{s-1}$  of $\{1, \dots, n-1\}$ such that for all $t \in \{1, \dots, s-1\}$ we have
\[
\text{for all } i_t \in I_t, \text{ there exists } i_{t-1} \in I_{t-1}  \text{ such that } A_{i_{t-1}} \compat A_{i_t}.
\] 
In the following, we write $i \notin I_0 \cup \dots \cup I_{s-1}$ to mean $i \in \{1, \dots, n\}\setminus \bigl(I_0 \cup \dots \cup I_{s-1}\bigr)$.
We define the two following binary matrices with the same even number of rows:
\begin{align*}
\widehat{A}_s &\coloneqq \Bigl(A_i\Bigr)_{i \in I_0 \cup \dots \cup I_{s-1}},
\\
B_s &\coloneqq \Bigl(A_i\Bigr)_{i \notin I_0 \cup \dots \cup I_{s-1}}.
\end{align*}
Note that the matrix $B_s$ is not empty since $n \in \{1, \dots, n\}\setminus \bigl(I_0 \cup \dots \cup I_{s-1}\bigr)$.
For all $\ell$ we have
\begin{align*}
\lvert R_\ell(\widehat{A}_s^+)\rvert  &= \sum_{i \in I_0 \cup \dots \cup I_{s-1}} \lvert R_\ell(A^+_i)\rvert,
\\
\lvert R_\ell(\widehat{A}_s^-)\rvert  &= \sum_{i \in I_0 \cup \dots \cup I_{s-1}} \lvert R_\ell(A^-_i)\rvert,
\\
\intertext{and}
\lvert R_\ell(B_s^+)\rvert  &= \sum_{i \notin I_0 \cup \dots \cup I_{s-1}} \lvert R_\ell(A^+_i)\rvert,
\\
\lvert R_\ell(B_s^-)\rvert &= \sum_{i \notin I_0 \cup \dots \cup I_{s-1}} \lvert R_\ell(A^-_i)\rvert.
\end{align*}
Thus,
\[
\lvert R_\ell(\widehat{A}^+_s)\rvert  + \lvert R_\ell(B^+_s)\rvert  = \sum_{i = 1}^n \lvert R_\ell(A^+_i)\rvert  = \sum_{i = 1}^n \lvert R_\ell(A^-_i)\rvert  = \lvert R_\ell(\widehat{A}^-_s)\rvert  + \lvert R_\ell(B^-_s)\rvert.
\]
Furthermore, since $\lvert A_i^+ \rvert \geq \lvert A_i^-\rvert$ for all $i \in I_1 \cup \dots \cup I_{s-1} \subseteq \{2, \dots,n-1\}$ and $\lvert A_1^+\rvert > \lvert A_1^-\rvert$ we obtain
\begin{align*}
\lvert \widehat{A}^+_s\rvert 
&= \sum_{i \in I_1 \cup \dots \cup I_{s-1}} \lvert A^+_i\rvert  + \lvert A^+_1\rvert 
\\
&\geq \sum_{i \in I_1 \cup \dots \cup I_{s-1}} \lvert A^-_i\rvert  + \lvert A^+_1\rvert 
\\
&\geq \lvert \widehat{A}^-_s\rvert  - \lvert A^-_1\rvert  + \lvert A^+_1\rvert 
\\
\lvert \widehat{A}^+_s\rvert  &> \lvert \widehat{A}^-_s\rvert.
\end{align*}
As a consequence, we can apply Lemma~\ref{lemme:interversion_bloc2} to the matrices $\widehat{A}_s$ and $B_s$. Hence, the set $I_s$ of integers $i \in \{1, \dots, n\} \setminus (I_0 \cup \dots \cup I_{s-1})$ such that $A_{\widehat{\imath}} \compat A_i$ for some $\widehat{\imath} \in I_0 \cup \dots \cup I_{s-1}$ is non-empty. Moreover, by construction such an integer $\widehat{\imath}$ is necessary in $I_{s-1}$. We stop here if $n \in I_s$, and otherwise we continue the induction with $I_0, I_1, \dots, I_s$.

Since the sets that we construct are non-empty, pairwise disjoint and included in $\{1, \dots, n\}$, there is some integer $s$ such that $n \in I_s$. By construction, for any $t \in \{1, \dots, s\}$ if $i_t \in I_t$ then there exists $i_{t-1} \in I_{t-1}$ such that $A_{i_{t-1}} \compat A_{i_t}$.
Hence, starting with $n \in I_s$, since the sets $(I_t)_{0 \leq t \leq s}$ are pairwise disjoint and $I_0 = \{1\}$, we can find a sequence $1  < i_1, \dots, i_{s-1} <  n$ of distinct integers such that $A_1 \compat A_{i_1} \compat \dots \compat A_{i_{s-1}} \compat A_n$.
\end{proof}

\begin{corollaire}
\label{corollaire:interversions_somme_blocs}
Let $(A_i)_{1 \leq i \leq n}$ be a family of matrices with the same even number of rows. We assume that
\begin{gather*}
\sum_{i=1}^n \lvert R_\ell(A^+_i)\rvert = \sum_{i=1}^n \lvert R_\ell(A^-_i)\rvert, \qquad \text{for all } \ell,
\\
\lvert A^+_{i_0} \rvert > \lvert A^-_{i_0}\rvert, \qquad \text{for some } i_0 \in \{1, \dots, n\}.
\end{gather*}
Then there exists a sequence of distinct integers $i_1, \dots, i_s$  distinct from $i_0$ such that
\[
A_{i_0} \compat A_{i_1} \compat A_{i_2} \compat \dots \compat A_{i_{s-1}} \compat A_{i_s},
\]
with $\lvert A^+_{i_s} \rvert < \lvert A^-_{i_s} \rvert$.
\end{corollaire}

\begin{proof}
Let $m \in \{1, \dots, n-1\}$ be the number of $i \in \{1, \dots, n\}$ such that $\lvert A^+_i \rvert \geq \lvert A^-_i\rvert$. Let $(j_k)_{1 \leq k \leq n}$ be a reordering of $\{1, \dots, n\}$ with $j_1 = i_0$ such that
\begin{align*}
\lvert A^+_{j_k} \rvert &\geq \lvert A^-_{j_k}\rvert,
&
&\text{for all } k \in \{1, \dots, m\},
\\
\lvert A^+_{j_k} \rvert &< \lvert A^-_{j_k}\rvert,
&
&\text{for all } k \in \{m+1, \dots, n\}.
\end{align*}
 We define the following binary matrix with an even number of rows:
\[
A \coloneqq \begin{pmatrix}
A_{j_{m+1}} & \cdots & A_{j_n}
\end{pmatrix}.
\]
For all $\ell$ we have
\[
\sum_{k = 1}^m \lvert R_\ell(A_{j_k}^+) \rvert + \lvert R_\ell(A^+)\rvert = \sum_{k = 1}^m \lvert R_\ell(A^-_{j_k}) \rvert + \lvert R_\ell(A^-)\rvert. 
\] Hence, we can apply Proposition~\ref{proposition:interversions_somme_blocs} to the family $(A_{j_1}, \dots, A_{j_m}, A)$. We find a sequence $i_1, \dots, i_{s-1}$ of distinct elements of $\{j_2, \dots, j_m\}$ such that
\[
A_{j_1} = A_{i_0} \compat A_{i_1} \compat \dots \compat A_{i_{s-1}} \compat A.
\]
We conclude since $A_{i_{s-1}} \compat A$ implies that there exists $i_s \in \{j_{m+1}, \dots, j_n\}$ such that $A_{i_{s-1}} \compat A_{i_s}$.
\end{proof}

%

\subsection{Application to binary averaging}
\label{subsection:binary_averaging}

The following result is well-known.

\begin{lemme}
\label{lemme:elt_base_canonique}
Let $w_0, \dots, w_{n-1} \in \{0, \dots, p\}$. For all $i \in \{0, \dots, n-1\}$ we define $v_i \coloneqq \frac{w_i}{p}$ and we set $v \coloneqq (v_0, \dots, v_{n-1}) \in [0, 1]^n$. There exist some vectors $\epsilon^0, \dots, \epsilon^{p-1} \in \{0, 1\}^n$ such that
\[
v = \frac{1}{p}\sum_{j = 0}^{p-1} \epsilon^j.
\]
In particular,
\[
\frac{1}{p}\sum_{j = 0}^{p-1} \lvert \epsilon^j\rvert  = \frac{1}{p}\sum_{j = 0}^{p-1} \lVert\epsilon^j\rVert^2 = \lvert v\rvert .
\]
If in addition $\lvert v\rvert \in \mathbb{N}$ then for all $j \in \{0, \dots, p-1\}$ we can choose $\epsilon^j$ such that $\lvert \epsilon^j \rvert = \lVert \epsilon^j\rVert^2 = \lvert v\rvert$.
\end{lemme}

The last result is equivalent to the existence of a binary $p \times n$ matrix  with row sums $(\lvert v\rvert, \dots, \lvert v \rvert)$ and column sums $(w_0, \dots, w_{n-1})$. By a general result of \cite{gale,ryser}, we know that such a matrix exists, since  the conjugate  $(p, \dots, p)$ (with $\lvert v \rvert$ terms) of the partition $(\lvert v\rvert, \dots, \lvert v\rvert)$ dominates the partition $\widetilde{w}$ for the usual dominance order on partitions, where $\widetilde{w}$ is the partition obtained by rearranging the entries of $w$ in decreasing order. However, for the convenience of the reader we give a simplified proof for the particular setting of Lemma~\ref{lemme:elt_base_canonique}.

\begin{proof}
For any $i \in \{0, \dots, n-1\}$, we define the set
\[
W_i \coloneqq \{w_0 + \dots + w_{i-1} +1, \dots, w_0 + \dots + w_i\}.
\]
For any $j \in \{0, \dots, p-1\}$, we consider the element $\epsilon^j \coloneqq (\epsilon^j_0, \dots, \epsilon^j_{n-1}) \in \{0, 1\}^n$ defined by
\[
\epsilon^j_i \coloneqq \begin{cases}
1 & \text{if } W_i \text{ contains an element of residue } j \text{ modulo } p,
\\
0 & \text{otherwise,}
\end{cases}
\]
for any $i \in \{0, \dots, n-1\}$.
Since $W_i$ has cardinality $w_i$ and is given by at most $p$ successive integers, the set of residues modulo $p$ of the elements of $W_i$ has also cardinality $w_i$. Hence, there are exactly $w_i$ integers $\epsilon^j_i$ for all $j \in \{0, \dots, p-1\}$ that are equal to $1$. The other are $0$, thus
\[
\sum_{j = 0}^{p-1} \epsilon^j_i = w_i.
\]
The $i$th component of $\frac{1}{p}\sum_{j = 0}^{p-1} \epsilon^j$ is thus $\frac{w_i}{p} = v_i$ and we obtain
\[
\frac{1}{p}\sum_{j = 0}^{p-1} \epsilon^j = v.
\]
Since $\lvert \cdot\rvert$ is additive, we deduce that $\frac{1}{p}\sum_{j =0}^{p-1} \lvert \epsilon^j\rvert =\lvert v\rvert$. Moreover, since $\epsilon^j \in \{0, 1\}^n$ we have $\lvert \epsilon^j\rvert = \lVert\epsilon^j\rVert^2$ thus $\frac{1}{p}\sum_{j =0}^{p-1} \lVert\epsilon^j\rVert^2 = \lvert v\rvert$.

Now assume  that $\lvert v\rvert \in \mathbb{N}$. There are in the set $\{1, \dots, \lvert v\rvert p = \lvert w\rvert \}$ exactly $\lvert v\rvert $ integers  of residue $j$ modulo $p$ for each $j \in \{0, \dots, p-1\}$. Since $\{W_i\}_{i \in \{0, \dots, n-1\}}$ is a partition of $\{1, \dots, \lvert w\rvert \}$, we deduce that
\[
\sum_{i = 0}^{n-1} \epsilon^j_i = \#\Bigl\{\text{elements of } \{1, \dots, \lvert w\rvert \} \text{ of residue } j \text{ modulo } p\Bigr\} = \lvert v\rvert,
\]
for all $j \in \{0, \dots, p-1\}$. Hence $\lvert \epsilon^j\rvert  = \lvert v\rvert $ and we conclude.
\end{proof}

We will use Corollary~\ref{corollaire:interversions_somme_blocs} of \textsection\ref{subsection:binary_matrices} to generalise Lemma~\ref{lemme:elt_base_canonique}: see Proposition~\ref{proposition:elt_base_canonique_bloc}. Let us first give an easy lemma.

\begin{lemme}
\label{lemme:ecart}
Let $a_0, \dots, a_{p-1}$ be integers of sum a multiple of $p$. The following integer:
\[
m \coloneqq \max\left\{ a_j - a_{j'} : j, j' \in \{0, \dots, p-1\}\right\} \in \mathbb{N},
\]
satisfies $m = 0$ or $m \geq 2$.
\end{lemme}

\begin{proof}
Assume $m \leq 1$. Then, for all $j, j' \in \{0, \dots, p-1\}$ we have $\lvert a_j - a_{j'}\rvert  \leq 1$. If $j_0 \in \{0, \dots, p-1\}$ is such that $a_{j_0}$ is the minimum of $\{a_j\}_{j \in \{0, \dots, p-1\}}$ then for all $j \in \{0, \dots, p-1\}$, there exists $\epsilon_j \in \{0, 1\}$ such that $a_j = a_{j_0} + \epsilon_j$.  From the hypothesis, we know that $p a_{j_0} + \sum_{j = 0}^{p-1} \epsilon_j$ is a multiple of $p$, thus $\sum_{j = 0}^{p-1} \epsilon_j$ is a multiple of $p$. Since $\epsilon_{j_0} = 0$, we deduce that  $\epsilon_j = 0$ for all $j$. We conclude that $a_{j_0} = a_j$ for all $j \in \{0, \dots, p-1\}$ thus $m = 0$.
\end{proof}
%

We need to introduce some notation in order to state Proposition~\ref{proposition:elt_base_canonique_bloc}.
For any $\ell \in \{0, \dots, d-1\}$ and $i \in \{0, \dots, e-1\}$, let $w^{(\ell)}_i \in \{0, \dots, p\}$ and set $v^{(\ell)}_i \coloneqq \frac{w^{(\ell)}_i}{p}$. For each $\ell \in \{0, \dots, d-1\}$ we define
\[
v^{(\ell)} \coloneqq (v^{(\ell)}_0, \dots, v^{(\ell)}_{e-1}).
\]
We obtain a $d \times e$ matrix
\[
V \coloneqq \begin{pmatrix}v^{(0)}\\
\vdots\\ v^{(d-1)}\end{pmatrix}.
\]
We assume that for all $\ell \in \{0, \dots, d-1\}$ we have  $\lvert v^{(\ell)} \rvert \in \mathbb{N}$. Hence, for all $\ell \in \{0, \dots, d-1\}$ we can apply Lemma~\ref{lemme:elt_base_canonique} (with $n \coloneqq e$). We obtain some vectors $\epsilon^{j(\ell)} \in \{0, 1\}^e$ for all $j \in \{0, \dots, p-1\}$, such that
\begin{equation}
\label{equation:v_ell_moyenne_epsilon}
v^{(\ell)} = \frac{1}{p} \sum_{j = 0}^{p-1} \epsilon^{j(\ell)},
\end{equation}
and
\begin{equation}
\label{equation:longueur_epsilon_i_ell_v_ell}
\lvert \epsilon^{j(\ell)}\rvert = \lvert v^{(\ell)}\rvert.
\end{equation}
For all $j \in \{0, \dots, p-1\}$, define the following $d \times e$ matrix:
\begin{gather*}
E^j \coloneqq \begin{pmatrix} \epsilon^{j(0)} \\ \vdots \\ \epsilon^{j(d-1)}\end{pmatrix}.
\end{gather*}
Recall that $e$ is a multiple of $\eta$ (and $e = \eta p$). We write the matrix $V$ with $\eta$ blocks of the same size $V = \begin{pmatrix}
V^{[0]} & \cdots & V^{[\eta-1]}
\end{pmatrix}$, and we use the same block structure for the matrices $E^j = \begin{pmatrix}
E^{j[0]} & \cdots & E^{j[\eta-1]}
\end{pmatrix}$.
As a consequence of~\eqref{equation:v_ell_moyenne_epsilon}, we have
\begin{equation}
\label{equation:longueur_Vk_moyenne_Eik}
\lvert V^{[i]} \rvert = \frac{1}{p}\sum_{j = 0}^{p-1} \lvert E^{j[i]} \rvert,
\end{equation}
for all  $i \in \{0, \dots, \eta-1\}$.

\begin{proposition}
\label{proposition:elt_base_canonique_bloc}
We keep the previous notation. In addition to the hypotheses $\lvert v^{(\ell)}\rvert \in \mathbb{N}$ for all $\ell \in \{0, \dots, d-1\}$, assume that for all $i \in \{0, \dots,\eta-1\}$ we have $\lvert V^{[i]} \rvert \in \mathbb{N}$.
Then we can choose the vectors $\epsilon^{j(\ell)}$ for all $j \in \{0, \dots, p-1\}$ and $\ell \in \{0, \dots, d-1\}$ such that the previous properties~\eqref{equation:v_ell_moyenne_epsilon} and \eqref{equation:longueur_epsilon_i_ell_v_ell} still hold, together with
\begin{equation}
\label{equation:sum_blocks}
\lvert E^{j[i]} \rvert = \lvert V^{[i]}\rvert,
\end{equation}
for all $j \in \{0, \dots, p-1\}$ and $i \in \{0, \dots, \eta-1\}$.
\end{proposition}

\begin{exemple}
Take $p = 4$ and $d = 2$. With the following matrix:
\[
V \coloneqq \frac{1}{4}\left(\begin{array}{*{4}{c}|*{4}{c}}
1&2&2&1 & 2&3&0&1
\\
0&2&1&3 & 1&3&2&0
\end{array}\right) = \begin{pmatrix}
v^{(0)} \\ v^{(1)}
\end{pmatrix} = \left(\begin{array}{c|c}
V^{[0]} & V^{[1]}
\end{array}\right),
\]
we have $\lvert v^{(0)} \rvert = \lvert v^{(1)}\rvert = \lvert V^{[0]} \rvert = \lvert V^{[1]} \rvert = 3$. The vectors $\epsilon^{j(\ell)}$ constructed as in the proof of Lemma~\ref{lemme:elt_base_canonique} are the following:
\begin{align*}
\epsilon^{0(0)} &= (1, 0, 1, 0, 0, 1, 0, 0),
&
\epsilon^{0(1)} &= (0, 1, 0, 1, 0, 1, 0, 0),
\\
\epsilon^{1(0)} &= (0, 1, 0, 1, 0, 1, 0, 0),
&
\epsilon^{1(1)} &= (0, 1, 0, 1, 0, 1, 0, 0),
\\
\epsilon^{2(0)} &= (0, 1, 0, 0, 1, 1, 0, 0),
&
\epsilon^{2(1)} &= (0, 0, 1, 0, 1, 0, 1, 0),
\\
\epsilon^{3(0)} &= (0, 0, 1, 0, 1, 0, 0, 1),
&
\epsilon^{3(1)} &= (0, 0, 0, 1, 0, 1, 1, 0).
\end{align*}
Thus, we have
\begin{align*}
E^0 &= \left(\begin{array}{*{4}{c}|*{4}{c}}
\color{red}1&0&1&0 & \color{red}0&1&0&0
\\
0&1&0&1 & 0&1&0&0
\end{array}\right)
=
\begin{pmatrix}
\epsilon^{0(0)}
\\
\epsilon^{0(1)}
\end{pmatrix}
= \left(\begin{array}{c|c}
E^{0[0]} & E^{0[1]}
\end{array}\right),
\\
E^1 &= \left(\begin{array}{*{4}{c}|*{4}{c}}
0&\color{blue}1&0&1 & \color{blue}0&1&0&0
\\
0&1&0&1 & 0&1&0&0
\end{array}\right)
=
\begin{pmatrix}
\epsilon^{1(0)}
\\
\epsilon^{1(1)}
\end{pmatrix}
= \left(\begin{array}{c|c}
E^{1[0]} & E^{1[1]}
\end{array}\right),
\\
E^2 &= \left(\begin{array}{*{4}{c}|*{4}{c}}
\color{red}0&1&0&0 & \color{red}1&1&0&0
\\
0&0&1&0 & 1&0&1&0
\end{array}\right)
=
\begin{pmatrix}
\epsilon^{2(0)}
\\
\epsilon^{2(1)}
\end{pmatrix}
= \left(\begin{array}{c|c}
E^{2[0]} & E^{2[1]}
\end{array}\right),
\\
E^3 &= \left(\begin{array}{*{4}{c}|*{4}{c}}
0&\color{blue}0&1&0 & \color{blue}1&0&0&1
\\
0&0&0&1 & 0&1&1&0
\end{array}\right)
=
\begin{pmatrix}
\epsilon^{3(0)}
\\
\epsilon^{3(1)}
\end{pmatrix}
= \left(\begin{array}{c|c}
E^{3[0]} & E^{3[1]}
\end{array}\right).
\end{align*}
However, we have $\lvert E^{0[0]} \rvert = 4 \neq \lvert V^{[0]}\rvert$, thus these vectors $\epsilon^{j(\ell)}$ do not satisfy the condition~\eqref{equation:sum_blocks} of Proposition~\ref{proposition:elt_base_canonique_bloc}. 
Let us consider the two compatible submatrices indicated by the coloured entries. Define $A \coloneqq \begin{pmatrix}
E^{0[0]} \\ E^{2[0]}
\end{pmatrix}$ and $B \coloneqq \begin{pmatrix}
E^{0[1]} \\ E^{2[1]}
\end{pmatrix}$ (respectively $C \coloneqq \begin{pmatrix}
E^{1[0]} \\ E^{3[0]}
\end{pmatrix}$ and $D \coloneqq \begin{pmatrix}
E^{1[1]} \\ E^{3[1]}
\end{pmatrix}$) and set $(\widetilde{A}, \widetilde{B}) \coloneqq \gamma_{1,1,1}(A, B)$ (resp. $(\widetilde{C}, \widetilde{D}) \coloneqq \gamma_{1,2,1}(C, D)$).
We have
\begin{align*}
E^0 &= \left(\begin{array}{c|c}
A^+ & B^+
\end{array}\right),
\\
E^1 &= \left(\begin{array}{c|c}
C^+ & D^+
\end{array}\right),
\\
E^2 &= \left(\begin{array}{c|c}
A^- & B^-
\end{array}\right),
\\
E^3 &= \left(\begin{array}{c|c}
C^- & D^-
\end{array}\right),
\\
\intertext{and}
\widetilde{E}^0 &\coloneqq \left(\begin{array}{c|c}\widetilde{A}^+ & \widetilde{B}^+
\end{array}\right)
= \left(\begin{array}{*{4}{c}|*{4}{c}}
\color{red}{0}&0&1&0 & \color{red}{1}&1&0&0
\\
0&1&0&1 & 0&1&0&0
\end{array}\right),
\\
\widetilde{E}^1 &\coloneqq \left(\begin{array}{c|c}\widetilde{C}^+ & \widetilde{D}^+
\end{array}\right)= \left(\begin{array}{*{4}{c}|*{4}{c}}
0&\color{blue}{0}&0&1 & \color{blue}{1}&1&0&0
\\
0&1&0&1 & 0&1&0&0
\end{array}\right),
\\
\widetilde{E}^2 &\coloneqq \left(\begin{array}{c|c}\widetilde{A}^- & \widetilde{B}^-
\end{array}\right)= \left(\begin{array}{*{4}{c}|*{4}{c}}
\color{red}{1}&1&0&0 & \color{red}{0}&1&0&0
\\
0&0&1&0 & 1&0&1&0
\end{array}\right),
\\
\widetilde{E}^3 &\coloneqq \left(\begin{array}{c|c}\widetilde{C}^- & \widetilde{D}^-
\end{array}\right)= \left(\begin{array}{*{4}{c}|*{4}{c}}
0&\color{blue}{1}&1&0 & \color{blue}{0}&0&0&1
\\
0&0&0&1 & 0&1&1&0
\end{array}\right).
\end{align*}
The vectors $\widetilde{\epsilon}^{j(\ell)}$ defined for all $j \in \{0, \dots, 3\}$ and $\ell \in \{0, 1\}$ by $\widetilde{E}^j = \begin{pmatrix}
\widetilde{\epsilon}^{j(0)} \\ \widetilde{\epsilon}^{j(1)}
\end{pmatrix}$ satisfy \eqref{equation:v_ell_moyenne_epsilon} and \eqref{equation:longueur_epsilon_i_ell_v_ell}, together with the condition~\eqref{equation:sum_blocks} of Proposition~\ref{proposition:elt_base_canonique_bloc}. In general, the existence of such interchanges will be given by Corollary~\ref{corollaire:interversions_somme_blocs}.
\end{exemple}

The remaining part of this subsection is now devoted to the proof of Proposition~\ref{proposition:elt_base_canonique_bloc}.
First, note that the interchanges $\bigl(\begin{smallmatrix} 1&0\\0&1\end{smallmatrix}\bigr) \leftrightarrow \bigl(\begin{smallmatrix} 0&1\\1&0\end{smallmatrix}\bigr)$ that are compatible with the block decomposition
\begin{equation}
\label{equation:matrice_epsilon}
\begin{pmatrix}
E^0
\\
\vdots
\\
E^{p-1}
\end{pmatrix}
=
\left(
\begin{array}{c|c|c}
E^{0[0]} & \cdots & E^{0[\eta-1]}
\\
\hline
\vdots & \vdots & \vdots
\\
\hline
E^{(p-1)[0]} & \cdots & E^{(p-1)[\eta-1]}
\end{array}\right),
\end{equation}
do not affect properties~\eqref{equation:v_ell_moyenne_epsilon} and~\eqref{equation:longueur_epsilon_i_ell_v_ell}. However, these interchanges change the value of some $\lvert E^{j[i]} \rvert$, as described in Proposition~\ref{proposition:interversion}. Thus, it suffices to prove that there exists a sequence of compatible interchanges that modifies each $\lvert E^{j[i]}\rvert$ to $\lvert V^{[i]}\rvert$.
We endow $\mathbb{N}\times\mathbb{N}^*$ with the usual lexicographic order. We will use an induction on $(\Delta, N) \in \mathbb{N}\times\mathbb{N}^*$, where
\[
\Delta \coloneqq \max\left\{ \lvert E^{j[i]} \rvert - \lvert E^{j'[i]} \rvert : i \in \{0, \dots, \eta-1\}, j, j' \in \{0, \dots, p-1\} \right\} \in \mathbb{N},
\]
and
\[
N \coloneqq \#\left\{(i, j, j') \in \{0, \dots, \eta-1\} \times \{0, \dots, p-1\}^2 : \lvert E^{j[i]} \rvert - \lvert E^{j'[i]} \rvert = \Delta\right\} \in \mathbb{N}^*.
\]
Define
\begin{align*}
M &\coloneqq \max\left\{\lvert E^{j[i]} \rvert : i \in \{0, \dots, \eta-1\}, j \in \{0, \dots, p-1\}\right\},
\\
m &\coloneqq \min\left\{\lvert E^{j[i]} \rvert : i \in \{0, \dots, \eta-1\}, j \in \{0, \dots, p-1\}\right\},
\end{align*}
and
\begin{align*}
N_{\text{max}} &\coloneqq \#\left\{(i, j) \in \{0, \dots, \eta-1\} \times \{0, \dots, p-1\} : \lvert E^{j[i]}\rvert = M\right\},
\\
N_{\text{min}} &\coloneqq  \#\left\{(i, j) \in \{0, \dots, \eta-1\} \times \{0, \dots, p-1\} : \lvert E^{j[i]}\rvert = m\right\}.
\end{align*}
We have $\Delta = M - m$ and $N = N_{\text{max}} N_{\text{min}}$.
If $\Delta = 0$ then by \eqref{equation:longueur_Vk_moyenne_Eik} we have $\lvert E^{j[i]} \rvert = \lvert V^{[i]} \rvert$ for all $i, j$ so the proof is over.  Assume $\Delta \geq 1$  and let $i_0 \in \{0, \dots ,e-1\}$ and $j_0, j'_0 \in \{0, \dots, p-1\}$ such that $\lvert E^{j_0[i_0]} \rvert - \lvert E^{j'_0[i_0]} \rvert = \Delta$.
We now consider the matrix
\[
\begin{pmatrix}
E^{j_0} \\ E^{j'_0}
\end{pmatrix}
=
\begin{pmatrix}
E^{j_0[0]} & \cdots & E^{j_0[i_0]} & \cdots & E^{j_0[\eta-1]}
\\
E^{j'_0[0]} & \cdots & E^{j'_0[i_0]} & \cdots & E^{j'_0[\eta-1]}
\end{pmatrix},
\]
given by the $j_0$th and $j'_0$th block-rows of the matrix of \eqref{equation:matrice_epsilon}.
We consider the family $(A_i)_{0 \leq i \leq \eta-1}$ of matrices with the same even number of rows defined by
\[
A_i = \begin{pmatrix}
A^+_i \\ A^-_i
\end{pmatrix} \coloneqq \begin{pmatrix}
E^{j_0[i]} \\ E^{j'_0[i]}
\end{pmatrix},
\]
for all $i \in \{0, \dots, \eta-1\}$.
The hypotheses of Corollary~\ref{corollaire:interversions_somme_blocs} are satisfied, thanks to the definition of $i_0$ and \eqref{equation:longueur_epsilon_i_ell_v_ell} (note that $R_\ell(E^{j_0}) = \epsilon^{j_0(\ell)}$ and $R_\ell(E^{j'_0}) = \epsilon^{j'_0(\ell)}$). Hence, we can find a sequence of distinct integers $i_1, \dots, i_s$ distinct from $i_0$ with $\lvert A_{i_s}^+ \rvert < \lvert A_{i_s}^- \rvert$ and
\[
A_{i_0} \compat \cdots \compat A_{i_s}.
\]
Let $(\widetilde{A}_{i_0}, \dots,  \widetilde{A}_{i_s}) \in \Gamma(A_{i_0}, \dots, A_{i_s})$. By Corollary~\ref{corollaire:interversion}, we know that
\begin{equation}
\label{equation:proof_general_tilde_unchanged}
\begin{aligned}
\lvert \widetilde{A}^+_{i_t} \rvert &= \lvert A^+_{i_t}\rvert,
\\
\lvert \widetilde{A}^-_{i_t} \rvert &= \lvert A^-_{i_t} \rvert,
\end{aligned}
\end{equation}
for all $t \in \{1, \dots, s-1\}$. Moreover, we have
\begin{subequations}
\label{subequations:proof_general_tilde}
\begin{align}
\label{equation:proof_general_tilde_i0}
\lvert \widetilde{A}^+_{i_0} \rvert &= \lvert A^+_{i_0} \rvert - 1,
&
\lvert \widetilde{A}^-_{i_0} \rvert &= \lvert A^-_{i_0} \rvert + 1,
\\
\label{equation:proof_general_tilde_is}
\lvert \widetilde{A}^+_{i_s} \rvert &= \lvert A^+_{i_s} \rvert + 1,
&
\lvert \widetilde{A}^-_{i_s} \rvert &= \lvert A^-_{i_s} \rvert - 1.
\end{align}
\end{subequations}
We now want to evaluate the new values $\widetilde{\Delta}$ and $\widetilde{N}$ of $\Delta$ and $N$ that we obtain and prove that $(\widetilde{\Delta}, \widetilde{N})$ is strictly less than $(\Delta, N)$.
We have
\[
\widetilde{\Delta} = \max\left\{ \lvert \widetilde{E}^{j[i]} \rvert - \lvert \widetilde{E}^{j'[i]} \rvert : i \in \{0, \dots, \eta-1\}, j, j' \in \{0, \dots, p-1\} \right\} \in \mathbb{N},
\]
and
\[
\widetilde{N} = \#\left\{(i, j, j') \in \{0, \dots, \eta-1\} \times \{0, \dots, p-1\}^2 : \lvert \widetilde{E}^{j[i]} \rvert - \lvert \widetilde{E}^{j'[i]} \rvert = \widetilde{\Delta}\right\} \in \mathbb{N}^*,
\]
where
\[
\widetilde{E}^{j[i]} \coloneqq \begin{cases}
\widetilde{A}^+_{i_t} & \text{if } i = i_t \text{ for some } t \in \{0, \dots, s\} \text{ and } j = j_0,
\\
\widetilde{A}^-_{i_t} & \text{if } i = i_t \text{ for some } t \in \{0, \dots, s\} \text{ and } j = j'_0,
\\
E^{j[i]} & \text{otherwise}.
\end{cases}
\]
Moreover, with
\begin{align*}
\widetilde{M} &\coloneqq \max\left\{\lvert \widetilde{E}^{j[i]} \rvert : i \in \{0, \dots, \eta-1\}, j \in \{0, \dots, p-1\}\right\},
\\
\widetilde{m} &\coloneqq \min\left\{\lvert \widetilde{E}^{j[i]} \rvert : i \in \{0, \dots, \eta-1\}, j \in \{0, \dots, p-1\}\right\},
\end{align*}
and
\begin{align*}
\widetilde{N}_{\text{max}} &\coloneqq \#\left\{(i, j) \in \{0, \dots, \eta-1\} \times \{0, \dots, p-1\} : \lvert \widetilde{E}^{j[i]}\rvert = \widetilde{M}\right\},
\\
\widetilde{N}_{\text{min}} &\coloneqq  \#\left\{(i, j) \in \{0, \dots, \eta-1\} \times \{0, \dots, p-1\} : \lvert \widetilde{E}^{j[i]}\rvert = \widetilde{m}\right\},
\end{align*}
we have $\widetilde{\Delta} = \widetilde{M} - \widetilde{N}$ and $\widetilde{N} = \widetilde{N}_{\text{max}} \widetilde{N}_{\text{min}}$. Note that by~\eqref{equation:proof_general_tilde_unchanged}, for all $i \in \{0, \dots, \eta - 1\}$ and $j \in \{0, \dots, p-1\}$ we have
\begin{align}
\label{equation:proof_general_equality}
\lvert \widetilde{E}^{j[i]}\rvert &= \lvert E^{j[i]}\rvert,
&
&\text{if } i \notin \{i_0, i_s\} \text{ or } j \notin \{j_0, j'_0\}.
\end{align}
By the assumption  $\lvert V^{[i]} \rvert \in \mathbb{N}$ and \eqref{equation:longueur_Vk_moyenne_Eik}, thanks to Lemma~\ref{lemme:ecart} we know that $\Delta = \lvert A^+_{i_0} \rvert - \lvert A^-_{i_0} \rvert = M - m  \geq 2$.
Hence, by~\eqref{equation:proof_general_tilde_i0} we have
\begin{equation}
\label{equation:proof_general_tilde_i0_ineq}
m < \lvert \widetilde{A}_{i_0}^-\rvert \leq \lvert \widetilde{A}_{i_0}^+  \rvert < M.
\end{equation}
Furthermore, since $m \leq \lvert A_{i_s}^+ \rvert < \lvert A_{i_s}^- \rvert \leq M$, by~\eqref{equation:proof_general_tilde_is} we  have
\begin{subequations}
\label{subequations:proof_general_tilde_is_ineq}
\begin{align}
\label{equation:proof_general_tilde_is+_ineq}
m < \lvert \widetilde{A}_{i_s}^+  \rvert \leq \lvert A_{i_s}^- \rvert \leq M,
\\
\label{equation:proof_general_tilde_is-_ineq}
m \leq \lvert A_{i_s}^+ \rvert \leq \lvert \widetilde{A}_{i_s}^- \rvert < M.
\end{align}
\end{subequations}
Equations~\eqref{equation:proof_general_equality}, \eqref{equation:proof_general_tilde_i0_ineq} and \eqref{subequations:proof_general_tilde_is_ineq} prove that $\widetilde{M} \leq M$ and $\widetilde{m} \geq m$, thus $\widetilde{\Delta} \leq \Delta$. If $\widetilde{\Delta} < \Delta$ then $(\widetilde{\Delta}, \widetilde{N}) < (\Delta, N)$, thus we now assume that $\widetilde{\Delta} = \Delta$, that is, $\widetilde{M} = M$ and $\widetilde{m} = m$. By~\eqref{equation:proof_general_equality} we have
\begin{multline*}
N_{\text{max}} - \widetilde{N}_{\text{max}} = \#\left\{(i, j) \in \{i_0, i_s\} \times \{j_0, j'_0\} : \lvert E^{j[i]} \rvert = M\right\}
\\
- \#\left\{(i, j) \in \{i_0, i_s\} \times \{j_0, j'_0\} : \lvert \widetilde{E}^{j[i]} \rvert = M\right\}.
\end{multline*}
Thus,
\[
N_{\text{max}} - \widetilde{N}_{\text{max}} = 1 + \delta_{\lvert A_{i_s}^-\rvert, M} - \#\left\{(i, j) \in \{i_0, i_s\} \times \{j_0, j'_0\} : \lvert \widetilde{E}^{j[i]} \rvert = M\right\},
\]
where $\delta$ is the Kronecker symbol.
By~\eqref{equation:proof_general_tilde_i0_ineq} and \eqref{subequations:proof_general_tilde_is_ineq}, we obtain
\begin{equation}
\label{equation:proof_general_diff_N}
N_{\text{max}} - \widetilde{N}_{\text{max}}
=
1 + \delta_{\lvert A_{i_s}^-\rvert,  M} -  \delta_{\lvert \widetilde{A}_{i_s}^+\rvert, M}.
\end{equation}
By~\eqref{equation:proof_general_tilde_is+_ineq}, we know that
\[
\delta_{\lvert \widetilde{A}_{i_s}^+\rvert, M} \leq \delta_{\lvert A_{i_s}^-\rvert,  M},
\]
thus \eqref{equation:proof_general_diff_N} yields $N_{\text{max}} - \widetilde{N}_{\text{max}} \geq 1$. Similarly, we have $N_{\text{min}} - \widetilde{N}_{\text{min}} \geq 1$. Finally, we obtain $\widetilde{N} = \widetilde{N}_{\text{max}} \widetilde{N}_{\text{min}} < N_{\text{max}} N_{\text{min}} = N$ and thus $(\widetilde{\Delta}, \widetilde{N}) = (\Delta, \widetilde{N}) < (\Delta, N)$. By induction, this concludes the proof of Proposition~\ref{proposition:elt_base_canonique_bloc}.

\subsection{A few inequalities}
\label{subsection:inegalites}
%
%
%
%
%
%

We will prove some inequalities that we will use to prove  Theorem~\ref{theoreme:orbite_1}. The setting of the first one is reminiscent of Lemma~\ref{lemme:elt_base_canonique} and the following ones use convexity. Recall that $\lVert \cdot \rVert$ is the euclidean norm on $\mathbb{R}^n$ and denote by $\langle \cdot, \cdot\rangle$ the associated scalar product.

\begin{lemme}
\label{lemme:majoration_g_epsilon_i}
Let $n \in \mathbb{N}^*$ and $h : \mathbb{R}^n \to \mathbb{R}$ be a function such that $h - \frac{1}{2}\lVert \cdot \rVert ^2$ is affine.
Let $v \in \mathbb{R}^n$ and suppose that  $\epsilon^0, \dots, \epsilon^{p-1} \in \{0, 1\}^n$ satisfy $v = \frac{1}{p}\sum_{j = 0}^{p-1} \epsilon^j$ and $\lvert \epsilon^j \rvert = \lVert \epsilon^j \rVert^2 = \lvert v \rvert$ for all $j \in \{0, \dots, p-1\}$.
For any $a \in \mathbb{R}^n$ we have
\[
h(a + v) - \frac{1}{p}\sum_{j = 0}^{p-1} h(a + \epsilon^j) = \frac{\lVert v\rVert^2 - \lvert v\rvert }{2}.
\]
More specifically, there exists $j \in \{0, \dots, p-1\}$ (depending on $a$) such that
\[
h(a+\epsilon^j) \leq h(a+v) + \frac{\lvert v\rvert  - \lVert v\rVert^2}{2}.
\]
\end{lemme}

\begin{proof}
Denote by $\Delta \coloneqq h(a+v) - \frac{1}{p} \sum_{j = 0}^{p-1} h(a+\epsilon^j)$ the left-hand side of the equality. Note that the Hessian matrix of the second partial derivatives of $h$ is the identity matrix. More precisely, since $h$ is a degree $2$ polynomial, the Taylor formula reads
\[
h(a+w) = h(a) + \langle \nabla h(a), w\rangle + \frac{1}{2} \lVert w \rVert^2, \qquad \text{for all } w \in \mathbb{R}^n,
\]
where $\nabla h(a)$ denotes the gradient of $h$ at $a$.
Since $v = \frac{1}{p}\sum_{j = 0}^{p-1} \epsilon^j$, the quantity that defines $\Delta$ vanishes at the affine level, hence
\[
\Delta
= \frac{1}{2}\left(\lVert v \rVert^2 - \frac{1}{p} \sum_{j = 0}^{p-1} \lVert \epsilon^j \rVert^2\right).
\]
We conclude since  $\lVert \epsilon^j\rVert^2 = \lvert v\rvert$. The second assertion is straightforward.
\end{proof}

The next inequalities involve convexity.  The first one is a particular case of a Jensen's inequality for convex functions. The reader may refer to \cite[Theorem 4]{merentes-nikodem}; we include a proof for convenience.

\begin{lemme}
\label{lemme:forte_convexite_general}
Let $n \in \mathbb{N}^*$ and $m \in \mathbb{R}$. Let $h : \mathbb{R}^n \to \mathbb{R}$ such that $h - \frac{m}{2}\lVert\cdot\rVert^2$ is convex. For any $x_0, \dots, x_{p-1} \in \mathbb{R}^n$ we have
\[
h(\overline{x}) \leq \frac{1}{p}\sum_{j = 0}^{p-1} h(x_j) - \frac{m}{2p} \sum_{j = 0}^{p-1} \lVert x_j - \overline{x}\rVert^2,
\]
where $\overline{x} \coloneqq \frac{1}{p}\sum_{j = 0}^{p-1} x_j$.
\end{lemme}

\begin{proof}
Since $h - \frac{m}{2}\lVert \cdot\rVert^2$ is convex, we have
\[
h(\overline{x}) - \frac{m}{2}\lVert\overline{x}\rVert^2 \leq \frac{1}{p}\sum_{j = 0}^{p-1} h(x_j) - \frac{m}{2p}\sum_{j = 0}^{p-1} \lVert x_j\rVert^2.
\]
Thus,
\begin{align*}
h(\overline{x})
&\leq
\frac{1}{p}\sum_{j = 0}^{p-1} h(x_j) - \frac{m}{2p}\left[\sum_{j = 0}^{p-1}\lVert x_j\rVert^2 - p\lVert \overline{x}\rVert^2\right]
\\
&\leq
\frac{1}{p}\sum_{j = 0}^{p-1} h(x_j) - \frac{m}{2p}\left[\sum_{j = 0}^{p-1}\lVert x_j - \overline{x}\rVert^2 + 2\sum_{j = 0}^{p-1} \langle x_j, \overline{x}\rangle - 2p \lVert \overline{x}\rVert^2\right]
\\
&\leq
\frac{1}{p}\sum_{j = 0}^{p-1} h(x_j) - \frac{m}{2p}\left[\sum_{j = 0}^{p-1}\lVert x_j - \overline{x}\rVert^2 + 2p\langle \overline{x}, \overline{x}\rangle - 2p \lVert \overline{x}\rVert^2\right]
\\
&\leq
\frac{1}{p}\sum_{j = 0}^{p-1} h(x_j) - \frac{m}{2p}\sum_{j = 0}^{p-1}\lVert x_j - \overline{x}\rVert^2.
\end{align*}
\end{proof}

\begin{remarque}
The real number $m$ of Lemma~\ref{lemme:forte_convexite_general} is usually taken to be positive. In this case, the map $h$ is convex and we say that it is  \emph{$m$-strongly convex}. We have stated Lemma~\ref{lemme:forte_convexite_general} for a general $m$ since we will need it to be negative in the proof of Lemma~\ref{lemme:erreur_compensent}.
\end{remarque}

For any $x \in \mathbb{R}$, we denote by $\{x\} \in [0, 1[$ its fractional part. We have $\{x\} \coloneqq x - \lfloor x\rfloor$, where $\lfloor x\rfloor \in \mathbb{Z}$ is the greatest integer less than or equal to $x$.
\begin{lemme}
\label{lemme:erreur_compensent}
Let $x_0, \dots, x_{p-1} \in \mathbb{Z}$ be integers and let $\overline{x} \coloneqq \frac{1}{p}\sum_{j = 0}^{p-1} x_j$. With $v \coloneqq \{\overline{x}\}$ we have
\[
v - v^2 \leq \frac{1}{p}\sum_{j = 0}^{p-1} (x_j - \overline{x})^2.
\]
\end{lemme}

\begin{proof}
Let us consider the function $\phi : \mathbb{R} \to \mathbb{R}$ defined by $x \mapsto \{x\} - \{x\}^2 + x^2$. It is continuous on $\mathbb{R}\setminus\mathbb{Z}$, and in fact continuous on $\mathbb{R}$ since $\lim_{x \to n^-} \phi(x) = \lim_{x \to n^+} \phi(x) = n^2$ for any $n \in \mathbb{Z}$.
Moreover,
\[
\phi(x)
= x - \lfloor x\rfloor - (x^2 - 2\lfloor x \rfloor x + \lfloor x \rfloor^2) + x^2
=
(1 + 2\lfloor x \rfloor) x - \lfloor x \rfloor(1+\lfloor x \rfloor).
\]
Thus, the function $\phi$ is affine on each interval  $\mathopen{[}n, n+1\mathclose{[}$ for $n \in \mathbb{Z}$, with slope $2n+1$. Hence, the function $\phi$ is continuous with non-decreasing left derivative thus $\phi$ is convex.
Applying Lemma~\ref{lemme:forte_convexite_general} with $n \coloneqq 1$, $m \coloneqq -2$ and $h \coloneqq \{ \cdot\} - \{\cdot\}^2$ we obtain
\[
v - v^2 \leq \frac{1}{p}\sum_{j = 0}^{p-1}\left(\{x_j\} - \{x_j\}^2\right) + \frac{1}{p}\sum_{j = 0}^{p-1} (x_j - \overline{x})^2.
\]
For any $j \in \{0, \dots, p-1\}$ we have $x_j \in \mathbb{Z}$ thus $\{x_j\} = 0$ and we conclude.
\end{proof}

\section{Proof of the main theorem}
\label{section:demo}

We are now ready to prove Theorem~\ref{theoreme:orbite_1}, which we repeat here for the convenience of the reader.

\begin{theoreme:orbite_1}
Let $\tuple{\lambda}$ be an $r$-partition and let $\alpha \coloneqq \alpha_\kappa(\tuple{\lambda}) \in Q^+$. Assume that $\kappa$ is compatible with  $(d, \eta, p)$.
If $\sigma \cdot \alpha = \alpha$ then there is an $r$-partition $\tuple{\mu} \in \P_\alpha^\kappa$ with $\prescript{\sigma}{}{\tuple{\mu}} = \tuple{\mu}$.
\end{theoreme:orbite_1}

Let $\tuple{\lambda}$ be an $r$-partition and assume that the multicharge $\kappa \in (\mathbb{Z}/e\mathbb{Z})^r$ is compatible with $(d, \eta, p)$. Recalling the reduction step Proposition~\ref{proposition:reduction}, we assume that $\tuple{\lambda}$ is an $e$-multicore.
 We define
\begin{align*}
\alpha &\coloneqq \alpha_\kappa(\tuple{\lambda}),
\\
x^{(k)} &\coloneqq x^{(k)}(\tuple{\lambda}),
& \text{for all } k &\in \{0, \dots, r-1\},
\\
n^i &\coloneqq n^i_\kappa(\tuple{\lambda}),
& \text{for all } i &\in \{0, \dots, e-1\}.
\end{align*}
In the whole section, we assume that $\sigma \cdot \alpha = \alpha$.
There will be four steps in the proof, each step corresponding to one subsection. First, we will give an expression of $n^0$ in terms of the abacus variables $x^{(0)}, \dots, x^{(r-1)}$, which takes into account the $\sigma$-stability of $\alpha$. We will then give a key lemma, followed by a naive (but useful) attempt to prove the theorem. Finally, we will use the results of Section~\ref{section:approx} to conclude the proof.

\subsection{Using shift invariance}
\label{subsection:using_shift_invariance}

In this subsection, we will write $n^0$ in terms of $x^{(k)}_i$ for $k \in \{0, \dots, r-1\}$ and $i \in \{0, \dots, e-1\}$ (Lemma~\ref{lemme:difference_ni_xi_multipart}). The  difference with the equality of Lemma~\ref{lemme:difference_ni_xi_multipart} is that $\alpha$ is now assumed to be $\sigma$-stable, which will allow us to make the expression symmetric.
The map $(\mathbb{R}^e)^r \to \mathbb{R}$ that we obtain will be later used to apply the convexity results of \textsection\ref{subsection:inegalites}.

 Recall from \textsection\ref{subsection:multipartitions} that we have some linear forms $L_0, \dots, L_{e-1}$ that satisfy~\eqref{equation:multipartitions-L_i}:
\[
n^0 = \sum_{k = 0}^{r-1} \left[\frac{1}{2}\bigl\lVert\tuplex{x}^{(k)}\bigr\rVert^2  - L_{-\kappa_k}\bigl(\tuplex{x}^{(k)}\bigr)\right].
\]
Since $\sigma \cdot \alpha = \alpha$, for all $j_0 \in \{0, \dots, p-1\}$ we have $n^0_\kappa(\tuple{\lambda}) = n^0_\kappa(\prescript{\sigma^{-j_0}}{}{\tuple{\lambda}})$ by Lemma~\ref{lemme:sigma_alpha_lambda}.
 We deduce that
\begin{align*}
n^0
&=
\sum_{k = 0}^{r-1} \left[\frac{1}{2}\bigl\lVert\tuplex{x}^{(k+j_0 d)}\bigr\rVert^2 - L_{-\kappa_k}\bigl(\tuplex{x}^{(k+j_0 d)}\bigr)\right]
\\
&=
\sum_{k = 0}^{r-1} \left[\frac{1}{2}\bigl\lVert\tuplex{x}^{(k)}\bigr\rVert^2 - L_{-\kappa_{k-j_0 d}}\bigl(\tuplex{x}^{(k)}\bigr)\right].
\end{align*}
Averaging on $j_0 \in \{0, \dots, p-1\}$, we obtain
\[
n^0 = \sum_{k = 0}^{r-1} \left[ \frac{1}{2}\bigl\lVert\tuplex{x}^{(k)}\bigr\rVert^2 - \widetilde{L}_{\overline{k}}\bigl(\tuplex{x}^{(k)}\bigr)\right],
\]
where $\widetilde{L}_{\overline{k}}$ is a linear form that depends only on the residue $\overline{k} \in \{0, \dots, d-1\}$ of $k$ modulo $d$. Now, if for $\ell \in \{0, \dots, d-1\}$ we consider the map defined on $\mathbb{R}^e$ by
\begin{equation}
\label{equation:def_g}
g_\ell : \tuplex{x} \mapsto \frac{1}{2}\lVert \tuplex{x}\rVert^2 - \widetilde{L}_\ell(\tuplex{x}),
\end{equation}
we have
\begin{equation}
\label{equation:def_f}
n^0 = \sum_{\ell = 0}^{d-1}\sum_{j = 0}^{p-1} g_\ell\bigl(\tuplex{x}^{(\ell + jd)}\bigr) \eqqcolon f\bigl(\tuplex{x}^{(0)}, \dots, \tuplex{x}^{(r-1)}\bigr).
\end{equation}
The map $f : (\mathbb{R}^e)^r \to \mathbb{R}$ is of the form $f = \frac{1}{2}\lVert\cdot\rVert^2 - L$ where $L$ is a linear form. Moreover, define
\begin{equation}
\label{equation:def_feg}
\begin{aligned}
f^{\langle p \rangle}\bigl(\tuplex{x}^{(0)}, \dots, \tuplex{x}^{(d-1)}\bigr)
&\coloneqq \sum_{\ell = 0}^{d-1} g_\ell\bigl(\tuplex{x}^{(\ell)}\bigr)
\\
&=
\frac{1}{p} f\bigl(\tuplex{x}^{(0)}, \dots, \tuplex{x}^{(d-1)}, \dots, \tuplex{x}^{(0)},\dots, \tuplex{x}^{(d-1)}\bigr),
\end{aligned}
\end{equation}
where, in the expression $f\bigl(\tuplex{x}^{(0)}, \dots, \tuplex{x}^{(d-1)}, \dots, \tuplex{x}^{(0)},\dots, \tuplex{x}^{(d-1)}\bigr)$ the sequence $\tuplex{x}^{(0)}, \dots, \tuplex{x}^{(d-1)}$ is repeated $p$ times. Like $f$, the map $f^{\langle p \rangle} : (\mathbb{R}^e)^d \to \mathbb{R}$ is of the form $\frac{1}{2}\lVert \cdot \rVert^2 - L^{\langle p \rangle}$, where $L^{\langle p \rangle}$ is a linear form. Note that for all $j \in \{0, \dots, p-1\}$ we have
\[
f^{\langle p \rangle}\bigl(\tuplex{x}^{(jd)}, \dots, \tuplex{x}^{(d - 1 + jd)}\bigr) = \sum_{\ell = 0}^{d-1} g_\ell\bigl(\tuplex{x}^{(\ell + jd)}\bigr),
\]
hence, by~\eqref{equation:def_f} we deduce that
\[
f\bigl(\tuplex{x}^{(0)}, \dots, \tuplex{x}^{(r-1)}\bigr) = \sum_{j = 0}^{p-1}  f^{\langle p \rangle}\bigl(\tuplex{x}^{(jd)}, \dots, \tuplex{x}^{(d - 1 + jd)}\bigr).
\]


\subsection{Key lemma}
\label{subsection:key_lemma}

Lemma~\ref{lemme:inegalite_implique_conj} that we will give in this subsection is the key to our proof of Theorem~\ref{theoreme:orbite_1}. Recall that $\alpha = \alpha_\kappa(\tuple{\lambda})$ satisfies $\sigma \cdot \alpha = \alpha$.
For any $i \in \{0, \dots, \eta - 1\}$, define
\[
\delta_i \coloneqq n^i - n^{i+1}.
\]
The $\sigma$-stability of $\alpha$ implies that $\delta_i = n^{i+j_0\eta} - n^{i+j_0\eta+1}$ for all $j_0 \in \{0, \dots, p-1\}$.
We deduce from Lemma~\ref{lemme:difference_ni_xi_multipart} and the compatibility of $\kappa$ with $(d, \eta, p)$ (cf. \eqref{equation:kappa_compatible}) that
\begin{equation}
\label{equation:delta}
\delta_i
=
\sum_{k = 0}^{r-1} x^{(k)}_{i+j_0\eta-\kappa_k}
=
\sum_{\ell = 0}^{d-1} \sum_{j = 0}^{p-1} x^{(\ell + jd)}_{i +(j_0- j)\eta - \kappa_\ell},
\end{equation}
for all $j_0 \in \{0, \dots, p-1\}$.

As noted in Remark~\ref{remark:add_rim_hooks}, the stuttering $r$-partition $\tuple{\mu}$ of Theorem~\ref{theoreme:orbite_1}, which satisfies  $\alpha_\kappa(\tuple{\mu}) = \alpha$, is not necessarily an $e$-multicore. The following lemma shows that, to prove Theorem~\ref{theoreme:orbite_1}, it suffices to find a stuttering $e$-multicore $\tuple{\nu}$ such that $\alpha_\kappa(\tuple{\nu}) = \alpha - h(\alpha_0 + \dots + \alpha_{e-1})$ for some $h \in \mathbb{N}$. 

\begin{lemme}
\label{lemme:inegalite_implique_conj}
Suppose that $\tuplex{z}^{(0)}, \dots, \tuplex{z}^{(d-1)} \in \mathbb{Z}^e_0$ are such that
\begin{equation}
\label{key_lemma:leq}
p f^{\langle p \rangle}\bigl(\tuplex{z}^{(0)}, \dots, \tuplex{z}^{(d-1)}\bigr)
\leq 
f\bigl(\tuplex{x}^{(0)}, \dots, \tuplex{x}^{(r-1)}\bigr),
\end{equation}
and
\begin{equation}
\label{key_lemma:eq}
\sum_{\ell = 0}^{d-1} \sum_{j = 0}^{p-1} z^{(\ell)}_{i - j\eta - \kappa_\ell} = \delta_i,
\end{equation}
for all  $i \in \{0, \dots, \eta-1\}$.
Then Theorem~\ref{theoreme:orbite_1} is true for the $e$-multicore $\tuple{\lambda}$: we can find an $r$-partition $\tuple\mu$ such that $\alpha_\kappa(\tuple\mu) = \alpha$ and $\prescript{\sigma}{}{\tuple{\mu}} = \tuple\mu$.
\end{lemme}

\begin{proof}
For any $\ell \in \{0, \dots, d-1\}$ and $j \in \{1, \dots, p-1\}$, define $z^{(\ell + jd)} \coloneqq z^{(\ell)} \in \mathbb{Z}^e_0$. For each $k \in \{0, \dots, r-1\}$, let $\overline{\mu}^{(k)}$ be the $e$-core of parameter $z^{(k)}$. We obtain an $e$-multicore $\overline{\tuple\mu} = (\overline{\mu}^{(0)}, \dots, \overline{\mu}^{(r-1)})$ that satisfies $\prescript{\sigma}{}{\overline{\tuple{\mu}}} = \overline{\tuple{\mu}}$. For any $i \in \{0, \dots, e-1\}$, we define  $m^i \coloneqq n^i_\kappa(\overline{\tuple{\mu}})$.
Since $\kappa$ is compatible with $(d, \eta, p)$, we have $\sum_{\ell = 0}^{d-1} \sum_{j = 0}^{p-1} z^{(\ell)}_{i - j\eta - \kappa_\ell} = \sum_{k = 0}^{r-1} z^{(k)}_{i - \kappa_k}$. By Lemma~\ref{lemme:difference_ni_xi_multipart} and the assumption~\eqref{key_lemma:eq}, we deduce that
\[
m^i - m^{i+1} = \delta_i,
\]
for all $i \in \{0, \dots, \eta-1\}$. Hence, for all $i \in \{0, \dots, \eta-1\}$ we have $m^i - m^{i+1} = n^i - n^{i+1}$ thus
\begin{equation}
\label{equation:m^i_n^i}
m^0 - m^i = n^0 - n^i.
\end{equation}
The above equality is also true for any $i \in \{0, \dots, e-1\}$ since  $n^i = n^{i+\eta}$ and $m^i = m^{i+\eta}$  (by Lemma~\ref{lemme:sigma_alpha_lambda}).
Recalling the definition of $f$ (respectively $f^{\langle p \rangle}$) given at \eqref{equation:def_f} (resp. \eqref{equation:def_feg}), the assumption~\eqref{key_lemma:leq} implies
\[
m^0 \leq n^0.
\]
Hence, as in the proof of Proposition~\ref{proposition:reduction} we can construct an $r$-partition $\tuple{\mu} = \bigl(\mu^{(0)}, \dots, \mu^{(r-1)}\bigr)$ such that $\prescript{\sigma}{}{\tuple{\mu}} = \tuple{\mu}$ and:
\begin{itemize}
\item the partition $\mu^{(0)}$ 
is obtained by adding $n^0 - m^0$ times an $\eta$-rim hook to $\overline{\mu}^{(0)}$;
\item we have $\mu^{(j)} = \overline{\mu}^{(j)}$ for all $j \in \{1, \dots, d-1\}$.
\end{itemize}
Finally, by Lemma~\ref{lemme:ajout_rimhook} and \eqref{equation:m^i_n^i} we obtain
\begin{align*}
\alpha_\kappa(\tuple{\mu}) &= \alpha_\kappa(\overline{\tuple{\mu}}) + \bigl(n^0 - m^0\bigr)(\alpha_0 + \cdots + \alpha_{e-1})
\\
&= \sum_{i = 0}^{e-1} m^i \alpha_i + \sum_{i = 0}^{e-1} \bigl(n^0-m^0\bigr) \alpha_i
\\
&= \sum_{i = 0}^{e-1} \bigl(n^0 + m^i-m^0\bigr) \alpha_i
\\
&= \sum_{i = 0}^{e-1} n^i \alpha_i
\\
&= \alpha,
\end{align*}
thus we conclude.
\end{proof}


\subsection{Naive attempt}
\label{subsection:tentative_naive}

We will use the convexity of the map $f : (\mathbb{R}^e)^r \to \mathbb{R}$ to obtain some parameters $\widetilde{\tuplex{z}}^{(0)}, \dots, \widetilde{\tuplex{z}}^{(d-1)}$ that almost satisfy  the conditions of Lemma~\ref{lemme:inegalite_implique_conj}. These parameters will not necessarily be integers: we will fix this in the next section.

\begin{proposition}
\label{proposition:inegalite_forte_convexite}
For any $\ell \in \{0, \dots, d-1\}$, we define
\[
\widetilde{\tuplex{z}}^{(\ell)} \coloneqq \frac{1}{p}\sum_{j = 0}^{p-1} \tuplex{x}^{(\ell + jd)} \in \frac{1}{p}\mathbb{Z}^e.
\]
We have
\[
p f^{\langle p \rangle}\bigl(\widetilde{\tuplex{z}}^{(0)}, \dots, \widetilde{\tuplex{z}}^{(d-1)}\bigr)
\leq
f\bigl(\tuplex{x}^{(0)}, \dots, \tuplex{x}^{(r-1)}\bigr) - \frac{1}{2 }\sum_{\ell = 0}^{d-1} \sum_{j = 0}^{p-1} \bigl\lVert\tuplex{x}^{(\ell+jd)} - \widetilde{z}^{(\ell)}\bigr\rVert^2.
\]
\end{proposition}

\begin{proof}
Let $\ell \in \{0, \dots, d-1\}$ and let $k \in \{0, \dots, r-1\}$ be of residue  $\ell$ modulo $d$. Recall the definition of the map $g_\ell : \mathbb{R}^e \to \mathbb{R}$ given in~\eqref{equation:def_g}. The map $g_\ell - \frac{1}{2}\lVert\cdot\rVert^2$ is convex, thus by Lemma~\ref{lemme:forte_convexite_general} we deduce that
\[
g_\ell\bigl(\widetilde{\tuplex{z}}^{(\ell)}\bigr)
\leq \frac{1}{p}\sum_{j = 0}^{p-1} g_\ell\bigl(\tuplex{x}^{(\ell+jd)}\bigr) - \frac{1}{2p}\sum_{j = 0}^{p-1} \bigl\lVert\tuplex{x}^{(\ell+jd)}-\widetilde{z}^{(\ell)}\bigr\rVert^2.
\]
Summing over all $\ell \in \{0, \dots, d-1\}$, we obtain
\[
f^{\langle p \rangle}\bigl(\widetilde{\tuplex{z}}^{(0)}, \dots, \widetilde{\tuplex{z}}^{(d-1)}\bigr)
\leq 
\frac{1}{p} f\bigl(\tuplex{x}^{(0)}, \dots, \tuplex{x}^{(r-1)}\bigr)
 -  \frac{1}{2p}\sum_{\ell = 0}^{d-1}\sum_{j = 0}^{p-1} \bigl\lVert\tuplex{x}^{(\ell+jd)}-\widetilde{z}^{(\ell)}\bigr\rVert^2.
\]
Multiplying by $p$ gives the desired result.
\end{proof}

\begin{remarque}
The inequality of Proposition~\ref{proposition:inegalite_forte_convexite} is in fact an equality since $g_\ell - \frac{1}{2}\lVert \cdot \rVert^2$ is linear.
\end{remarque}

Let us now try to verify the hypotheses of Lemma~\ref{lemme:inegalite_implique_conj} with the parameters $\widetilde{\tuplex{z}}^{(0)}, \dots, \widetilde{\tuplex{z}}^{(d-1)} \in \frac{1}{p}\mathbb{Z}^e$ of Proposition~\ref{proposition:inegalite_forte_convexite}. First, for each $\ell \in \{0, \dots, d-1\}$ we have
\begin{equation}
\label{equation:ztilde_sum0}
\bigl\lvert \widetilde{\tuplex{z}}^{(\ell)}\bigr\rvert = \frac{1}{p}\sum_{j = 0}^{p-1} \bigl\lvert \tuplex{x}^{(\ell + jd)} \bigr\rvert = \frac{1}{p}\sum_{j = 0}^{p-1} 0 = 0.
\end{equation}
Moreover, since $\bigl\lVert\tuplex{x}^{(\ell+jd)} - \tuplex{x}^{(\ell+j'd)}\bigr\rVert \geq 0$ we deduce from the inequality of Proposition~\ref{proposition:inegalite_forte_convexite} that
\begin{equation}
\label{equation:naive_attempt-conv}
p f^{\langle p \rangle}\bigl(\widetilde{\tuplex{z}}^{(0)}, \dots, \widetilde{\tuplex{z}}^{(d-1)}\bigr) \leq f\bigl(\tuplex{x}^{(0)},\dots, \tuplex{x}^{(r-1)}\bigr).
\end{equation}
Finally, for each  $i \in \{0, \dots, \eta-1\}$ we have, using~\eqref{equation:delta},
\begin{align}
\sum_{\ell=0}^{d-1} \sum_{j = 0}^{p-1} \widetilde{z}^{(\ell)}_{i -j\eta - \kappa_\ell}
&=
\frac{1}{p}\sum_{\ell=0}^{d-1} \sum_{j = 0}^{p-1} \sum_{j' = 0}^{p-1} x^{(\ell+j'd)}_{i -j\eta - \kappa_\ell}
\notag
\\
&=
\frac{1}{p}\sum_{j = 0}^{p-1}\sum_{\ell=0}^{d-1} \sum_{j' = 0}^{p-1} x^{(\ell+j'd)}_{i +(\underbrace{\scriptstyle -j+j'}_{\eqqcolon j_0}-j')\eta - \kappa_\ell}
\notag
\\
&=
\frac{1}{p}\sum_{j_0 = 0}^{p-1} \left(\sum_{\ell=0}^{d-1} \sum_{j' = 0}^{p-1} x^{(\ell+j'd)}_{i +(j_0-j')\eta - \kappa_\ell}\right)
\notag
\\
&= \frac{1}{p}\sum_{j_0 = 0}^{p-1} \delta_i
\notag
\\
&=
\delta_i.
\label{equation:sum_ztilde_delta}
\end{align}
Hence, all hypotheses are satisfied but one: the parameters $\widetilde{z}^{(0)}, \dots, \widetilde{z}^{(d-1)} \in \frac{1}{p}\mathbb{Z}^e_0$ are not necessarily in $\mathbb{Z}^e_0$.

\subsection{Rectification of the parameters}
\label{subsection:rectification}

We will construct from  $\widetilde{\tuplex{z}}^{(0)}, \dots, \widetilde{\tuplex{z}}^{(d-1)} \in \frac{1}{p}\mathbb{Z}^e_0$ (defined in Proposition~\ref{proposition:inegalite_forte_convexite})  some elements $\tuplex{z}^{(0)}, \dots, \tuplex{z}^{(d-1)} \in \mathbb{Z}^e_0$ that satisfy all the assumptions of Lemma~\ref{lemme:inegalite_implique_conj}. To that end, we will approximate $\widetilde{\tuplex{z}}^{(0)}, \dots, \widetilde{\tuplex{z}}^{(d-1)}$ using Proposition~\ref{proposition:elt_base_canonique_bloc}, and we will control the value of $f\bigl(\tuplex{z}^{(0)}, \dots, \tuplex{z}^{(d-1)}\bigr)$ using Lemma~\ref{lemme:majoration_g_epsilon_i}.
The remaining of this subsection in now devoted to the proof of the following proposition.

\begin{proposition}
\label{proposition:elt_z}
There exist elements $\tuplex{z}^{(0)}, \dots, \tuplex{z}^{(d-1)} \in \mathbb{Z}^e_0$ such that
\[
\sum_{\ell = 0}^{d-1} \sum_{j = 0}^{p-1} z^{(\ell)}_{i-j\eta-\kappa_\ell} = \delta_i,
\]
for all $i \in \{0, \dots, \eta-1\}$ and
\[
f^{\langle p \rangle}\bigl(\tuplex{z}^{(0)}, \dots, \tuplex{z}^{(d-1)}\bigr)
\leq
f^{\langle p \rangle}\bigl(\widetilde{\tuplex{z}}^{(0)}, \dots,\widetilde{\tuplex{z}}^{(d-1)}\bigr)
+
\frac{1}{2p}\sum_{\ell = 0}^{d-1}\sum_{j = 0}^{p-1} \bigl\lVert \tuplex{x}^{(\ell+jd)} -\widetilde{z}^{(\ell)}\bigr\rVert^2.
\]
\end{proposition}

Let $\ell \in \{0, \dots, d-1\}$. Since $\widetilde{z}^{(\ell)} \in \frac{1}{p}\mathbb{Z}^e$, we know that for any $i \in \{0, \dots, \eta-1\}$ and $j \in \{0, \dots, p-1\}$ there exist unique elements $m^{(\ell)}_{j + ip} \in \mathbb{Z}$ and $w^{(\ell)}_{j + ip} \in \{0, \dots, p-1\}$  such that 
\begin{equation}
\label{equation:relabelling}
\widetilde{z}^{(\ell)}_{i - j\eta-\kappa_\ell} = m^{(\ell)}_{j+ip} + \frac{w^{(\ell)}_{j+ip}}{p}.
\end{equation}
The fractional part of $\widetilde{z}^{(\ell)}_{i - j\eta - \kappa_\ell}$ is
\begin{equation}
\label{equation:def_v}
\bigl\{\widetilde{z}^{(\ell)}_{i-j\eta-\kappa_\ell}\bigr\} = \frac{w^{(\ell)}_{j+ip}}{p}   \eqqcolon v^{(\ell)}_{j+ip}.
\end{equation}
For each $\ell \in \{0, \dots, d-1\}$, we have two $e$-tuples $m^{(\ell)} \coloneqq \bigl(m^{(\ell)}_0, \dots, m^{(\ell)}_{e-1}\bigr)$ and $v^{(\ell)} \coloneqq \bigl(v^{(\ell)}_0, \dots, v^{(\ell)}_{e-1}\bigr)$.  Let $\pi_\ell$ be the permutation of $\{0, \dots, e-1\}$ defined by
\[
\pi_\ell(j + ip) \coloneqq i - j\eta - \kappa_\ell,
\]
for all $i \in \{0, \dots, \eta-1\}$ and $j \in \{0, \dots, p-1\}$. Permuting the components of $e$-tuples according to $\pi_0, \dots, \pi_{d-1}$, we obtain a map $\widetilde{f}^{\langle p \rangle} : (\mathbb{R}^e)^d \to \mathbb{R}$ that  satisfies
\[
\widetilde{f}^{\langle p \rangle}\bigl(m^{(0)} + v^{(0)}, \dots, m^{(d-1)} + v^{(d-1)}\bigr) = f^{\langle p \rangle}\bigl(\widetilde{z}^{(0)}, \dots, \widetilde{z}^{(d-1)}\bigr).
\]
To match with the setting of \textsection\ref{subsection:binary_averaging}, we define the two following $d \times e$ matrices:
\begin{align*}
M &= \begin{pmatrix}
m^{(0)}
\\
\vdots
\\
m^{(d-1)}
\end{pmatrix},
&
V &= \begin{pmatrix}
v^{(0)}
\\
\vdots
\\
v^{(d-1)}
\end{pmatrix},
\end{align*}
so that
\begin{equation}
\label{equation:ftildep_MV}
\widetilde{f}^{\langle p \rangle}\bigl(M + V\bigr) = f^{\langle p \rangle}\bigl(\widetilde{\tuplex{z}}^{(0)}, \dots, \widetilde{\tuplex{z}}^{(d-1)}\bigr).
\end{equation}
Like $f^{\langle p \rangle}$, the map $\widetilde{f}^{\langle p \rangle}$ defined on the $d \times e$ matrices is of the form $\frac{1}{2}\lVert \cdot \rVert^2 - \widetilde{L}$ where $\lVert \cdot \rVert^2$ is the sum of the squares of the entries and $\widetilde{L}$ is a linear form.
We now write the matrix $V$ blockwise in the same fashion as for  Proposition~\ref{proposition:elt_base_canonique_bloc}. That is, 
\begin{gather*}
V = \begin{pmatrix}
v^{(0)} \\ \vdots \\ v^{(d-1)}
\end{pmatrix}
=
\begin{pmatrix}
V^{[0]} & \cdots & V^{[\eta-1]}
\end{pmatrix},
\end{gather*}
where
\[
V^{[i]} = \begin{pmatrix}
v^{(0)}_{ip} & \cdots & v^{(0)}_{p-1 + ip}
\\
\vdots & \vdots & \vdots
\\
v^{(d-1)}_{ip} & \cdots & v^{(d-1)}_{p-1+ip}
\end{pmatrix},
\]
for any $i \in \{0, \dots, \eta-1\}$.
We now check that $V$ satisfies the assumptions of Proposition~\ref{proposition:elt_base_canonique_bloc}.
First, for any $\ell \in \{0, \dots, d-1\}$ the element $v^{(\ell)}$ satisfies $\bigl\lvert v^{(\ell)} \bigr\rvert \geq 0$ since its entries are non-negative. Furthermore,
\begin{align*}
\bigl\lvert v^{(\ell)} \bigr\rvert
&=
\sum_{i = 0}^{\eta-1} \sum_{j = 0}^{p-1} v^{(\ell)}_{j + ip}
\\
&=
\sum_{i = 0}^{\eta-1} \sum_{j = 0}^{p-1} \bigl(\widetilde{z}^{(\ell)}_{i - j\eta - \kappa_\ell} - m^{(\ell)}_{j + ip}\bigr) \qquad \text{(by } \eqref{equation:relabelling}, \eqref{equation:def_v})
\\
&=
\sum_{i = 0}^{e-1} \widetilde{z}^{(\ell)}_i - \sum_{i = 0}^{\eta-1} \sum_{j = 0}^{p-1} m^{(\ell)}_{j + ip}
\\
&=
\bigl\lvert \widetilde{\tuplex{z}}^{(\ell)}\bigr\rvert -\sum_{i = 0}^{\eta-1} \sum_{j = 0}^{p-1} m^{(\ell)}_{j + ip}.
\end{align*}
Hence, we have $\bigr\lvert v^{(\ell)} \bigr\rvert \in \mathbb{Z}$ since $\bigl\lvert \widetilde{\tuplex{z}}^{(\ell)} \bigr\rvert = 0$ and $m^{(\ell)}_{j + ip} \in \mathbb{Z}$, thus $\bigl\lvert v^{(\ell)} \bigr\rvert \in \mathbb{N}$.
The same argument shows that $\bigl\lvert V^{[i]} \bigr\rvert \in \mathbb{N}$ for any  $i \in \{0, \dots, \eta-1\}$ since
\begin{align*}
\bigl\lvert V^{[i]} \bigr\rvert
&=
\sum_{\ell = 0}^{d-1} \sum_{j = 0}^{p-1} v^{(\ell)}_{j + ip}
\\
&=
\sum_{\ell = 0}^{d-1} \sum_{j = 0}^{p-1} \widetilde{z}^{(\ell)}_{i - j\eta - \kappa_\ell} - \sum_{\ell = 0}^{d-1} \sum_{j = 0}^{p-1} m^{(\ell)}_{j + ip}
\\
&=
\delta_i - \sum_{\ell = 0}^{d-1} \sum_{j = 0}^{p-1} m^{(\ell)}_{j + ip}.
\end{align*}

Thus, we can apply Proposition~\ref{proposition:elt_base_canonique_bloc}. There exist vectors $\epsilon^{j(\ell)} \in \{0, 1\}^e$ for all $j \in \{0, \dots, p-1\}$ and $\ell \in \{0, \dots, d-1\}$ such that
\begin{align}
\frac{1}{p} \sum_{j = 0}^{p-1} \epsilon^{j(\ell)} &= v^{(\ell)},
\notag
\\
\label{equation:proof_epsilonjl_vl}
\bigl\lvert \epsilon^{j(\ell)} \bigr\rvert &= \bigl\lvert v^{(\ell)}\bigr\rvert,
\\
\label{equation:proof_Eji_Vi}
\bigl\lvert E^{j[i]} \bigr\rvert &= \bigl\lvert V^{[i]}\bigr\rvert, \qquad \text{for all } i \in \{0, \dots, \eta-1\}.
\end{align}
In the above equality, the matrices $E^{j[i]}$ for any $i \in \{0, \dots, \eta-1\}$ are defined by the same block decomposition as $V$:
\[
E^j \coloneqq \begin{pmatrix}
\epsilon^{j(0)}
\\
\vdots
\\
\epsilon^{j(d-1)}
\end{pmatrix} = \begin{pmatrix}
E^{j[0]} & \cdots & E^{j[\eta-1]}
\end{pmatrix},
\]
in particular each $E^{j[i]}$ has size $d \times p$.
The map $\widetilde{f}^{\langle p \rangle}$ and the matrices $V$ and $E^j$ for all $j \in \{0, \dots, p-1\}$ satisfy the assumptions of  Lemma~\ref{lemme:majoration_g_epsilon_i}. Hence, there exists  $j_0 \in \{0, \dots, p-1\}$ such that
\[
\widetilde{f}^{\langle p \rangle}(M + E^{j_0}) \leq \widetilde{f}^{\langle p \rangle}(M + V) + \frac{\lvert V \rvert - \lVert V\rVert^2}{2}.
\]
For each $\ell \in \{0, \dots, d-1\}$,  define the vector $\tuplex{z}^{(\ell)}$ by permuting the coordinates of  $m^{(\ell)} + \epsilon^{j_0(\ell)}$ via $\pi_\ell$. We have
\[
f^{\langle p \rangle}\bigl(\tuplex{z}^{(0)}, \dots, \tuplex{z}^{(d-1)}\bigr) = \widetilde{f}^{\langle p \rangle}(M + E^{j_0}),
\]
thus, recalling~\eqref{equation:ftildep_MV},
\begin{equation}
\label{equation:preuve_z_ineg_v}
f^{\langle p \rangle}\bigl(\tuplex{z}^{(0)}, \dots, \tuplex{z}^{(d-1)}\bigr) \leq f^{\langle p \rangle}\bigl(\widetilde{\tuplex{z}}^{(0)}, \dots, \widetilde{\tuplex{z}}^{(d-1)}\bigr) + \frac{\lvert V \rvert - \lVert V \rVert^2}{2}.
\end{equation}
We now check that $\tuplex{z}^{(0)}, \dots, \tuplex{z}^{(d-1)}$ have the properties described in Proposition~\ref{proposition:elt_z}. First, for any $\ell \in \{0, \dots, d-1\}$ the vector $z^{(\ell)}$ is a permutation of $m^{(\ell)} + \epsilon^{j_0(\ell)}$. Since $m^{(\ell)} \in \mathbb{Z}^e$ and $\epsilon^{j_0(\ell)} \in \{0, 1\}^e$, we have $z^{(\ell)} \in \mathbb{Z}^e$. Moreover,
\begin{align*}
\bigl\lvert \tuplex{z}^{(\ell)} \bigr\rvert
&=
\bigl\lvert m^{(\ell)} \bigr\rvert + \bigl\lvert \epsilon^{j_0(\ell)}\bigr\rvert
\\
&=
\bigl\lvert m^{(\ell)} \bigr\rvert + \bigl\lvert v^{(\ell)}\bigr\rvert & &\text{(by } \eqref{equation:proof_epsilonjl_vl})
\\
&=
\bigl\lvert \widetilde{\tuplex{z}}^{(\ell)}\bigr\rvert
\\
&=
0 &&\text{(by } \eqref{equation:ztilde_sum0}),
\end{align*}
thus $\tuplex{z}^{(\ell)} \in \mathbb{Z}^e_0$.
The equality condition of Proposition~\ref{proposition:elt_z}  is satisfied, since for any $i \in \{0, \dots, \eta-1\}$ we have
\begin{align*}
\sum_{\ell = 0}^{d-1} \sum_{j = 0}^{p-1} z^{(\ell)}_{i - j\eta - \kappa_\ell}
&=
\sum_{\ell = 0}^{d-1} \sum_{j = 0}^{p-1} \left[m^{(\ell)}_{j + ip} + \epsilon^{j_0(\ell)}_{j + ip}\right]
\\
&=
\sum_{\ell = 0}^{d-1} \sum_{j = 0}^{p-1} m^{(\ell)}_{j + ip} + \bigl\lvert E^{j_0[i]}\bigr\rvert
\\
&=
\sum_{\ell = 0}^{d-1} \sum_{j = 0}^{p-1} m^{(\ell)}_{j + ip} + \bigl\lvert V^{[i]}\bigr\rvert &&\text{(by } \eqref{equation:proof_Eji_Vi})
\\
&=
\sum_{\ell = 0}^{d-1} \sum_{j = 0}^{p-1} \left[m^{(\ell)}_{j + ip} + v^{(\ell)}_{j + ip}\right]
\\
&=
\sum_{\ell = 0}^{d-1} \sum_{j = 0}^{p-1} \widetilde{z}^{(\ell)}_{i - j\eta - \kappa_\ell}
\\
&=
\delta_i.
\end{align*}

It remains to prove that the value of $f^{\langle p \rangle}\bigl(z^{(0)}, \dots, z^{(d-1)}\bigr)$ does not grow too much. We have
\[
\frac{\lvert V \rvert - \lVert V \rVert^2}{2} =
\frac{1}{2}
\sum_{\ell = 0}^{d-1} \left[\bigl\lvert v^{(\ell)} \bigr\rvert - \bigl\lVert v^{(\ell)} \bigr\rVert^2\right]
= \frac{1}{2}\sum_{\ell = 0}^{d-1} \sum_{i = 0}^{\eta-1} \sum_{j = 0}^{p-1} \left[v^{(\ell)}_{j + ip} - \bigl(v^{(\ell)}_{j + ip}\bigr)^2\right].
\]
Recall the definition of the vectors $\widetilde{z}^{(\ell)}$ for any $\ell \in \{0, \dots, d-1\}$ given in Proposition~\ref{proposition:inegalite_forte_convexite}.
Since for all $i \in \{0, \dots, \eta-1\}$ and $j \in \{0, \dots, p-1\}$, each $v^{(\ell)}_{j + ip}$ is the fractional part of
\[
\widetilde{z}^{(\ell)}_{i - j\eta - \kappa_\ell} = \frac{1}{p} \sum_{j' = 0}^{p-1} x^{(\ell + j'd)}_{i - j\eta - \kappa_\ell},
\]
and since each $x^{(\ell + j'd)}_{i - j\eta - \kappa_\ell}$ is an integer, we can apply Lemma~\ref{lemme:erreur_compensent}. We obtain
\[
v^{(\ell)}_{j + ip} - \bigl(v^{(\ell)}_{j + ip}\bigr)^2 \leq \frac{1}{p}\sum_{j' = 0}^{p-1}  \left(x^{(\ell + j'd)}_{i - j\eta - \kappa_\ell} - \widetilde{z}^{(\ell)}_{i - j\eta - \kappa_\ell}\right)^2,
\]
for all $i \in \{0, \dots, \eta-1\}$ and $j \in \{0, \dots, p-1\}$, thus
\[
\bigl\lvert v^{(\ell)} \bigr\rvert - \bigl\lVert v^{(\ell)} \bigr\rVert^2
\leq
\frac{1}{p}
\sum_{j' = 0}^{p-1} \bigl\lVert \tuplex{x}^{(\ell + j'd)} - \widetilde{z}^{(\ell)}\bigr\rVert^2
\]
It follows from \eqref{equation:preuve_z_ineg_v} that
\begin{align*}
f^{\langle p \rangle}\bigl(\tuplex{z}^{(0)}, \dots, \tuplex{z}^{(d-1)}\bigr)
&\leq f^{\langle p \rangle}\bigl(\widetilde{\tuplex{z}}^{(0)}, \dots, \widetilde{\tuplex{z}}^{(d-1)}\bigr) + \frac{1}{2}\sum_{\ell = 0}^{d-1}  \left[\bigl\lvert v^{(\ell)} \bigr\rvert - \bigl\lVert v^{(\ell)}\bigr\rVert^2\right]
\\
&\leq
f^{\langle p \rangle}\bigl(\widetilde{\tuplex{z}}^{(0)}, \dots, \widetilde{\tuplex{z}}^{(d-1)}\bigr)
+
\frac{1}{2p}\sum_{\ell = 0}^{d-1}\sum_{j = 0}^{p-1} \bigl\lVert \tuplex{x}^{(\ell + jd)} - \widetilde{z}^{(\ell)}\bigr\rVert^2,
\end{align*}
as desired.

\subsection{Proof of the main theorem}
\label{subsection:proof_main_theorem}

We now conclude the proof of Theorem~\ref{theoreme:orbite_1}.
Let $\tuplex{z}^{(0)}, \dots, \tuplex{z}^{(d-1)} \in \mathbb{Z}^e_0$ be as in Proposition~\ref{proposition:elt_z}.  They satisfy
\begin{equation}
\label{equation:subsectionproof_delta}
\sum_{\ell = 0}^{d-1} \sum_{j = 0}^{p-1} z^{(\ell)}_{i-j\eta-\kappa_\ell} = \delta_i,
\end{equation}
for all $i \in \{0, \dots, \eta-1\}$ and
\[
f^{\langle p \rangle}\bigl(\tuplex{z}^{(0)}, \dots, \tuplex{z}^{(d-1)}\bigr)
\leq
f^{\langle p \rangle}\bigl(\widetilde{\tuplex{z}}^{(0)}, \dots,\widetilde{\tuplex{z}}^{(d-1)}\bigr)
+
\frac{1}{2p}\sum_{\ell = 0}^{d-1}\sum_{j = 0}^{p-1} \bigl\lVert \tuplex{x}^{(\ell+jd)} -\widetilde{z}^{(\ell)}\bigr\rVert^2.
\]
Since, by Proposition~\ref{proposition:inegalite_forte_convexite}, we have
\[
p f^{\langle p \rangle}\bigl(\widetilde{\tuplex{z}}^{(0)}, \dots, \widetilde{\tuplex{z}}^{(d-1)}\bigr)
\leq
f\bigl(\tuplex{x}^{(0)}, \dots, \tuplex{x}^{(r-1)}\bigr) - \frac{1}{2 }\sum_{\ell = 0}^{d-1} \sum_{j = 0}^{p-1} \bigl\lVert\tuplex{x}^{(\ell+jd)} - \widetilde{z}^{(\ell)}\bigr\rVert^2,
\]
we obtain
\begin{multline*}
pf^{\langle p \rangle}\bigl(\tuplex{z}^{(0)}, \dots, \tuplex{z}^{(d-1)}\bigr)
\leq
f\bigl(\tuplex{x}^{(0)}, \dots, \tuplex{x}^{(r-1)}\bigr)
- \frac{1}{2 }\sum_{\ell = 0}^{d-1} \sum_{j = 0}^{p-1} \bigl\lVert\tuplex{x}^{(\ell+jd)} - \widetilde{z}^{(\ell)}\bigr\rVert^2
\\
+
\frac{1}{2}\sum_{\ell = 0}^{d-1}\sum_{j = 0}^{p-1} \bigl\lVert \tuplex{x}^{(\ell+jd)} -\widetilde{z}^{(\ell)}\bigr\rVert^2,
\end{multline*}
thus
\[
pf^{\langle p \rangle}\bigl(\tuplex{z}^{(0)}, \dots, \tuplex{z}^{(d-1)}\bigr)
\leq
f\bigl(\tuplex{x}^{(0)}, \dots, \tuplex{x}^{(r-1)}\bigr).
\]
\begin{remarque}
The error term $\frac{1}{2}\sum_{\ell = 0}^{d-1}\sum_{j = 0}^{p-1} \bigl\lVert \tuplex{x}^{(\ell+jd)} -\widetilde{z}^{(\ell)}\bigr\rVert^2$ vanished thanks the strong convexity inequality of Proposition~\ref{proposition:inegalite_forte_convexite}, the ``basic'' convexity inequality~\eqref{equation:naive_attempt-conv} being not accurate enough.
\end{remarque}

The above inequality, together with~\eqref{equation:subsectionproof_delta}, prove that the elements $\tuplex{z}^{(0)}, \dots, \tuplex{z}^{(d-1)} \in \mathbb{Z}^e_0$ satisfy the hypotheses of Lemma~\ref{lemme:inegalite_implique_conj}. Hence, Theorem~\ref{theoreme:orbite_1} is proved for the $e$-multicore $\tuple{\lambda}$. Recalling the reduction step from $r$-partitions to $e$-multicores, Proposition~\ref{proposition:reduction}, we conclude that Theorem~\ref{theoreme:orbite_1} is true for any $r$-partition.

\section{Applications}
\label{section:applications}

We assume that the multicharge $\kappa$ is compatible with $(d, \eta, p)$ (cf. \eqref{equation:kappa_compatible} and \eqref{equation:forme_kappa}). We present two applications of Theorem~\ref{theoreme:orbite_1} and Corollary~\ref{corollaire:orbite_gnrl}.
First, we will recall the definition of cellular algebras, as introduced by Graham and Lehrer~\cite{graham-lehrer}. The algebra $\H_n^\kappa$ and its blocks $\H_\alpha^\kappa$ for $\alpha \in Q^\kappa_n$ are examples of cellular algebras. We are interested in the fixed point subalgebras $\H_{p, [\alpha]}^\kappa$ (respectively $\H_{p, n}^\kappa$) of $\H_{[\alpha]}^\kappa$ (resp. $\H_n^\kappa$) for the algebra homomorphism $\sigma$. We can easily give bases for these algebras (cf. Proposition~\ref{proposition:basis_stable}). In \textsection\ref{subsubsection:full_orbit}, we prove that if $\#[\alpha]  = p$ (resp. if $p$ and $n$ are coprime) then $\H_{p, [\alpha]}^\kappa$ (resp. $\H_{p, n}^\kappa$) is cellular. Otherwise, using Corollary~\ref{corollaire:orbite_gnrl} we show that if in addition $p$ is odd then none of these bases of $\H_{p, [\alpha]}^\kappa$ are \emph{adapted} cellular (see \textsection\ref{subsubsection:adapted_cellularity}). Finally, in \textsection\ref{subsection:restriction_specht}  we will study the restriction of \emph{Specht modules} of $\H_{[\alpha]}^\kappa$. 

\subsection{Cellular algebras}
\label{subsection:cellular_algebras}

Let $A$ be an associative unitary finite-dimensional $F$-algebra.
A \emph{cellular datum} for the algebra $A$ is a triple $(\Lambda, \T, c)$ such that:
\begin{itemize}
\item the element $\Lambda = (\Lambda, \geq)$ is a finite partially ordered set;
\item for any $\lambda \in \Lambda$ we have an indexing set $\T(\lambda)$ and distinct elements $c^\lambda_\mathfrak{st}$ for all $\mathfrak{s}, \mathfrak{t} \in \T(\lambda)$ such that
\[
\left\{c^\lambda_\mathfrak{st} : \lambda \in \Lambda, \mathfrak{s}, \mathfrak{t} \in \T(\lambda)\right\},
\]
is a basis of $A$ as an $F$-module;
\item for any $\lambda \in \Lambda, \mathfrak{t} \in \T(\lambda)$ and $a \in A$, there exist scalars $r_\mathfrak{tv}(a) \in F$ such that for all $\mathfrak{s} \in \T(\lambda)$,
\[
c^\lambda_\mathfrak{st}a = \sum_{\mathfrak{v} \in \T(\lambda)} r_\mathfrak{tv}(a) c^\lambda_\mathfrak{sv} \pmod {A^{> \lambda}},
\]
where $A^{> \lambda}$ is the $F$-module spanned by $\{c^\mu_\mathfrak{ab} : \mu > \lambda \text{ and } \mathfrak{a}, \mathfrak{b} \in \T(\mu)\}$;
\item the $F$-linear map $* : A \to A$ determined by ${(c^\lambda_\mathfrak{st})}^* \coloneqq c^\lambda_\mathfrak{ts}$ for all $\lambda \in \Lambda$ and $\mathfrak{s}, \mathfrak{t} \in \T(\lambda)$ is an anti-automorphism of the algebra $A$.
\end{itemize}
We say that $A$ is a \emph{cellular algebra} if it has a cellular datum. We say that a basis $\mathcal{B}$ of $A$ is \emph{cellular} if it coincides with $\{c^\lambda_\mathfrak{st} : \lambda \in \Lambda, \mathfrak{s}, \mathfrak{t} \in \T(\lambda)\}$ where $(\Lambda, \T, c)$ is a cellular datum for $A$.

\begin{remarque}
\label{remark:cellular_algebra_dimension}
If $(\Lambda, \T, c)$ is a cellular datum for $A$ then
\[
\dim A = \sum_{\lambda \in \Lambda} \#\T(\lambda)^2.
\]
\end{remarque}

\begin{lemme}
\label{lemma:cellular_algebra_number_fixed_points}
Let $(\Lambda, \T, c)$ be a cellular datum of $A$ and let $*$ be the corresponding anti-automorphism. The cardinality of
\[
\left\{c^\lambda_\mathfrak{st} : \lambda \in \Lambda, \mathfrak{s}, \mathfrak{t} \in \T(\lambda), {(c^\lambda_\mathfrak{st})}^* = c^\lambda_\mathfrak{st}\right\},
\]
is $\sum_{\lambda \in \Lambda} \#\T(\lambda)$.
\end{lemme}

\begin{proof}
Since ${(c^\lambda_\mathfrak{st})}^* = c^\lambda_\mathfrak{ts}$, we have ${(c^\lambda_\mathfrak{st})}^* = c^\lambda_\mathfrak{st}$ if and only if $\mathfrak{s} = \mathfrak{t}$.
\end{proof}

%

Assume that $(\Lambda, \T, c)$ is a cellular datum for $A$. By~\cite{graham-lehrer}, for each $\lambda \in \Lambda$ we have an $A$-module $\mathcal{S}^\lambda$, called \emph{cell module}, endowed with a certain bilinear form $b_\lambda$ whose radical is an $A$-module. Moreover, if $\mathcal{D}^\lambda$ denotes the quotient of $\mathcal{S}^\lambda$ by the radical of $b_\lambda$, the set $\{\mathcal{D}^\lambda : \lambda \in \Lambda, \mathcal{D}^\lambda \neq \{0\}\}$ is a complete family of non-isomorphic irreducible $A$-modules. We conclude with the following lemma (see \cite[(3.9.8)]{graham-lehrer} or \cite[Corollary 2.22]{mathas}).  

\begin{lemme}
\label{lemma:composition_factors}
For any $\lambda \in \Lambda$, all the composition factors of the cell module~$\mathcal{S}^\lambda$ belong to the same block of the algebra~$A$.
\end{lemme}

Let $\lambda \in \Lambda$ and let $B$ be the block of~$A$ such that all the composition factors of~$\mathcal{S}^\lambda$ belong to~$B$. We say that~$\lambda$ \emph{belongs} to the block~$B$.

\subsection{Cellularity of the fixed point subalgebra}
\label{subsection:cellularity_fixed_points_subalgebra}

We will first give more definitions from combinatorics, and recall the existence of a particular cellular datum for $\H_n^\kappa$ and its blocks $\H_\alpha^\kappa$. Then, we will construct bases for the algebra $\H_{p, [\alpha]}^\kappa$ and study its cellularity. We will use the following notation:
\[
Q^\kappa_n \coloneqq \left\{\alpha \in Q^+ : \text{there exists } \tuple{\lambda} \in \P_n^\kappa \text{ such that } \alpha_\kappa(\tuple{\lambda}) = \alpha\right\}.
\]

\subsubsection{Tableaux}

Let $\tuple{\lambda} = \bigl(\lambda^{(0)}, \dots, \lambda^{(r-1)}\bigr)$ be an $r$-partition of $n$. Recall that we defined in \textsection\ref{subsection:partitions} and \textsection\ref{subsection:multipartitions} the Young diagram $\Y(\tuple{\lambda})$ of $\tuple{\lambda}$. A \emph{$\tuple{\lambda}$-tableau} is a bijection $\mathfrak{t} = \bigl(\mathfrak{t}^{(0)}, \dots, \mathfrak{t}^{(r-1)}\bigr) : \Y(\tuple{\lambda}) \to \{1, \dots, n\}$.
The \emph{$\kappa$-residue sequence} of a $\tuple{\lambda}$-tableau $\mathfrak{t}$ is the sequence
\[
\res_\kappa(\mathfrak{t}) \coloneqq \Bigl(\res_\kappa\bigl(\mathfrak{t}^{-1}(a)\bigr)\Bigr)_{a \in \{1, \dots, n\}}.
\]
A $\tuple{\lambda}$-tableau $\mathfrak{t} : \Y(\tuple{\lambda}) \to \{1, \dots, n\}$ is \emph{standard} if the value of $\mathfrak{t}$ increases along the rows and down the columns of $\Y(\tuple{\lambda})$. We denote by $\T(\tuple{\lambda})$ the set of standard $\tuple{\lambda}$-tableaux. 

\begin{exemple}
We take $r = p = 2$ and we consider the bipartition $\tuple{\lambda} \coloneqq \bigl((4,1), (1)\bigr)$. The map $\mathfrak{t} : \Y(\tuple{\lambda}) \to \{1, \dots, 6\}$ described by
\[
\ytableaushort{1546,2} \quad \ytableaushort{3}\, ,
\]
is a $\tuple{\lambda}$-tableau (we warn the reader that we represented in the same way the multiset of residues associated with a multipartition), but it is not standard. The tableau $\mathfrak{s} : \Y(\tuple{\lambda}) \to \{1, \dots, 6\}$ described by
\[
\ytableaushort{1456,2} \quad \ytableaushort{3}\, ,
\]
is standard. With $\kappa = (0, 2)$ and $e = 4 = 2\eta$, the residue sequence of $\mathfrak{s}$ is $\res_\kappa(\mathfrak{s}) = (0,3,2,1,2,3)$. 
\end{exemple}

Mimicking Definition~\ref{definition:shift_lambda}, we define the \emph{shift} of a $\tuple{\lambda}$-tableau $\mathfrak{t} = \bigl(\mathfrak{t}^{(0)}, \dots, \mathfrak{t}^{(r-1)}\bigr)$ by
\[
\prescript{\sigma}{}{\mathfrak{t}} \coloneqq \bigl(\mathfrak{t}^{(r-d)}, \dots, \mathfrak{t}^{(r-1)}, \mathfrak{t}^{(0)}, \dots, \mathfrak{t}^{(r-d+1)}\bigr),
\]
and we denote by $[\mathfrak{t}]$ the orbit of $\mathfrak{t}$ under the action of $\sigma$.
%
Note that $\prescript{\sigma}{}{\mathfrak{t}}$ is a $\prescript{\sigma}{}{\tuple{\lambda}}$-tableau, which is standard if $\mathfrak{t}$ is standard. In particular the set $\T[\tuple{\lambda}] \coloneqq \cup_{\tuple{\mu} \in [\tuple{\lambda}]} \T(\tuple{\mu})$ is stable under $\sigma$ and there is a well-defined equivalence relation $\sim$ on $\T[\tuple{\lambda}]$ generated by $\mathfrak{t} \sim \prescript{\sigma}{}{\mathfrak{t}}$. We write $\TT[\tuple{\lambda}] \coloneqq \T[\tuple{\lambda}] / {\sim}$ for the set of equivalence classes. 
Choose a lift $\phi : \TT[\tuple{\lambda}] \to \T[\tuple{\lambda}]$ of the canonical projection $\T[\tuple{\lambda}] \to \TT[\tuple{\lambda}]$. In other words, if $\mathfrak{t}$ is any standard $\tuple{\lambda}$-tableau then $\phi([\mathfrak{t}]) \in [\mathfrak{t}]$. For any $j \in \{0, \dots, p-1\}$, we define
\[
\T^\phi_j(\tuple{\lambda}) \coloneqq \left\{ \mathfrak{t} \in \T(\tuple{\lambda}) : \phi([\mathfrak{t}]) = \prescript{\sigma^j}{}{\mathfrak{t}}\right\}.
\]
Note that the set $\T^\phi_j(\tuple{\lambda})$ may be empty for some $j \in \{0, \dots, p-1\}$, but we have a partition $\T(\tuple{\lambda}
) = \sqcup_{j = 0}^{p-1} \T^\phi_j(\tuple{\lambda})$. Moreover:
\begin{equation}
\label{equation:tableaux:t_Tj}
\text{if } \mathfrak{t} \in \T^\phi_j(\tuple{\lambda}) \text{ then } \prescript{\sigma}{}{\mathfrak{t}} \in \T^\phi_{j-1}(\prescript{\sigma}{}{\tuple{\lambda}}).
\end{equation}
We have
\begin{equation}
\label{equation:card_T0crochet_Tcrochet}
\#\T_0^\phi[\tuple{\lambda}] = \frac{1}{p}\#\T[\tuple{\lambda}],
\end{equation}
where $\T_0^\phi[\tuple{\lambda}] \coloneqq \cup_{\tuple{\mu} \in [\tuple{\lambda}]} \T_0^\phi(\tuple{\mu}) = \bigl\{\mathfrak{t} \in \T[\tuple{\lambda}] : \phi([\mathfrak{t}]) = \mathfrak{t}\bigr\}$. In particular, the cardinality of $\T_0^\phi[\tuple{\lambda}]$ does not depend on $\phi$ and we may abuse notation by writing $\#\T_0[\tuple{\lambda}]$ instead of $\T_0^\phi[\tuple{\lambda}]$. Since $\#\T(\tuple{\lambda}) = \frac{1}{\#[\tuple{\lambda}]}\#\T[\tuple{\lambda}]$, we also deduce that
\begin{equation}
\label{equation:card_Tparenthese_T0crochet}
\#\T(\tuple{\lambda}) = \frac{p}{\#[\tuple{\lambda}]}\T_0^\phi[\tuple{\lambda}].
\end{equation}

\begin{exemple}
Recall that the multicharge $\kappa$ is compatible with $(d, \eta, p)$. For any $\mathfrak{t} \in \T[\tuple{\lambda}]$, the compatibility condition~\eqref{equation:forme_kappa} ensures that there exists a unique standard tableau $\widetilde{\phi}(\mathfrak{t}) \in [\mathfrak{t}]$ such that $1$ is in the image of the first $d$ components of $\widetilde{\phi}(\mathfrak{t})$, that is, such that there exists $c \in \{0, \dots, d-1\}$ with $\widetilde{\phi}(\mathfrak{t})\bigl((0,0,c)\bigr) = 1$. Note that when $d = 1$ (\textit{i.e.} when $r = p$), this condition is the same as $\res_\kappa\bigl(\widetilde{\phi}^{-1}(1)\bigr) = \kappa_0$. The map $\widetilde{\phi} : \T[\tuple{\lambda}] \to \T[\tuple{\lambda}]$ is constant on the equivalent classes of $\sim$. Thus, it factorises to a map $\phi : \TT[\tuple{\lambda}] \to \T[\tuple{\lambda}]$ that lifts the natural projection. In this case, for any $j \in \{0, \dots, p-1\}$ we have
\[
\T^\phi_j(\tuple{\lambda}) = \Bigl\{ \mathfrak{t} \in \T(\tuple{\lambda}) : \text{there exists } c \in \{(p-j)d, \dots, (p-j+1)d - 1\} \text{ such that } \mathfrak{t}\bigl((0,0,c)\bigr) = 1\Bigr\}.
\]
We will see in \textsection\ref{subsubsection:full_orbit} another example of a lift $\phi$ of the natural projection.
\end{exemple}

\begin{remarque}
Here, we chose $\phi$ to be a map $\TT[\tuple{\lambda}] \to \T[\tuple{\lambda}]$. If $\mathcal{P}$ is  any subset of $\P_n^\kappa/{\sim}$, the equivalence relation $\sim$ is also defined on $\cup_{[\tuple{\lambda}] \in \mathcal{P}} \T[\tuple{\lambda}]$ and the equivalence classes are in natural bijection with $\cup_{[\tuple{\lambda}] \in \mathcal{P}} \TT[\tuple{\lambda}]$. Thus, we can allow $\phi$ to be a lift $\cup_{[\tuple{\lambda}] \in \mathcal{P}} \TT[\tuple{\lambda}] \to \cup_{[\tuple{\lambda}] \in \mathcal{P}} \T[\tuple{\lambda}]$.
\end{remarque}

\subsubsection{Cellular datum for the Ariki--Koike algebra}
\label{subsubsection:cellularity_AK}

It is known that we can find a family
\begin{equation}
\label{equation:ungraded_cellular_basis}
\left\{c^{\tuple{\lambda}}_\mathfrak{st} : \tuple{\lambda} \in \P^\kappa_n \text{ and } \mathfrak{s}, \mathfrak{t} \in \T(\tuple{\lambda})\right\},
\end{equation}
that form a cellular basis of $\H^\kappa_n$  (cf. \cite{dipper-james-mathas}).
Now recall that we defined in~\eqref{equation:introduction-sigma} a particular algebra automorphism $\sigma : \H_n^\kappa \to \H_n^\kappa$ of order $p$. Let $\tuple{\eta}$ be the $n$-tuple $(\eta, \dots, \eta)$ considered as an element of $(\mathbb{Z}/e\mathbb{Z})^n$. By~\cite{brundan-kleshchev,rostam}, we know that the algebra $\H_n^\kappa$ is generated by some elements
\begin{align*}
&e(\tuple{i}), & &\text{for any } \tuple{i} \in (\mathbb{Z}/e\mathbb{Z})^n,
\\
&\psi_a, &&\text{for any } a \in \{1, \dots, n-1\},
\\
&y_a, &&\text{for any } a \in \{1, \dots, n\},
\end{align*}
the ``Khovanov--Lauda generators'', for which
\begin{subequations}
\label{equation:sigma_graded_generators}
\begin{align}
\label{equation:sigma_e(i)}
\sigma(e(\tuple{i})) &= e(\tuple{i} - \tuple{\eta}), & &\text{for any } \tuple{i} \in (\mathbb{Z}/e\mathbb{Z})^n,
\\
\sigma(\psi_a) &= \psi_a, & &\text{for any } a \in \{1, \dots, n-1\},
\\
\sigma(y_a) &= y_a, & &\text{for any } a \in \{1, \dots, n\}.
\end{align}
\end{subequations}
The elements $\{e(\tuple{i}) : \tuple{i} \in (\mathbb{Z}/e\mathbb{Z})^n\}$ form a complete system of orthogonal idempotents, that is,
\begin{subequations}
\begin{align}
\label{equation:e(i)_idempotent}
e(\tuple{i})^2 &= e(\tuple{i}), & &\text{for any } \tuple{i} \in (\mathbb{Z}/e\mathbb{Z})^n,
\\
\label{equation:e(i)e(j)}
e(\tuple{i})e(\tuple{j}) &= 0, &&\text{for any } \tuple{i} \neq \tuple{j} \in (\mathbb{Z}/e\mathbb{Z})^n,
\\
\label{equation:sum_e(i)}
\sum_{\tuple{i} \in (\mathbb{Z}/e\mathbb{Z})^n} e(\tuple{i}) &= 1.
\end{align}
\end{subequations}

Among the generators $e(\tuple{i})$ for any $\tuple{i} \in (\mathbb{Z}/e\mathbb{Z})^n$, we know exactly the ones that are non-zero (see~\cite[4.1. Lemma]{hu-mathas_graded}).

\begin{lemme}
\label{lemma:non_zero_idempotents}
For any $\tuple{i} \in (\mathbb{Z}/e\mathbb{Z})^n$, the idempotent $e(\tuple{i}) \in \H_n^\kappa$ is non-zero if and only if there exist $\tuple{\lambda} \in \P_n^\kappa$ and $\mathfrak{t} \in \T(\tuple{\lambda})$ such that $\tuple{i} = \res_\kappa(\mathfrak{t})$.
\end{lemme}

There is a well-defined algebra anti-automorphism $* : \H_n^\kappa \to \H_n^\kappa$, which we now fix, that is the identity on each Khovanov--Lauda generator (see~\cite[\textsection 5.1]{hu-mathas_graded}). We can find a cellular basis of $\H_n^\kappa$ of the form~\eqref{equation:ungraded_cellular_basis} such that the associated anti-automorphism is the map $*$, with the additional property
\begin{equation}
\label{equation:cellular_basis_idempotent}
c^{\tuple{\lambda}}_\mathfrak{st} \in e(\res_\kappa(\mathfrak{s})) \H_n^\kappa e(\res_\kappa(\mathfrak{t})),
\end{equation}
for all $\tuple{\lambda} \in \P_n^\kappa$ and $\mathfrak{s}, \mathfrak{t} \in \T(\tuple{\lambda})$ (see~\cite{hu-mathas_graded} and also~\cite{bowman}). Note that we recover the result of Lemma~\ref{lemma:non_zero_idempotents}.
We now fix such a cellular basis.

\begin{remarque}
\label{remark:application_AK_graded}
The cellular bases that are constructed in \cite{hu-mathas_graded,bowman} are \emph{graded} cellular bases: the algebra $\H_n^\kappa$ is $\mathbb{Z}$-graded (\cite{rouquier, brundan-kleshchev}) and there exists a map $\deg : \coprod_{\tuple{\lambda} \in \P^\kappa_n} \T(\tuple{\lambda})  \to \mathbb{Z}$ such that $c^{\tuple{\lambda}}_\mathfrak{st}$ is homogeneous of degree $\deg \mathfrak{s} + \deg \mathfrak{t}$. These graded cellular bases seem to be more adapted to $\sigma$ than the ungraded one of~\cite{dipper-james-mathas}: if $\H_n^\kappa$ is semisimple, we can prove that $\sigma$ permutes the elements of the graded basis but its action on the ungraded basis is more complicated.
\end{remarque}

The condition~\eqref{equation:cellular_basis_idempotent} allows us to give a more precise description of this cellular structure for $\H_n^\kappa$. 
 For any $\alpha \in Q^+$ with $\lvert\alpha\rvert = n$, denote by $I^\alpha$ the subset of $(\mathbb{Z}/e\mathbb{Z})^n$ given by the $n$-tuples $\tuple{i} \in (\mathbb{Z}/e\mathbb{Z})^n$ that have exactly $\alpha_i$ components equal to $i$ for any $i \in \{0, \dots, e-1\}$. The subalgebra
\[
\H^\kappa_\alpha \coloneqq \sum_{\tuple{i}, \tuple{j} \in I^\alpha} e(\tuple{i}) \H^\kappa_n e(\tuple{j}) \subseteq \H^\kappa_n,
\]
is a block of $\H^\kappa_n$ if $\alpha \in Q^\kappa_n$ and $\{0\}$ otherwise (see~\cite{lyle-mathas}). By~\eqref{equation:cellular_basis_idempotent}, when $\alpha \in Q^\kappa_n$ the algebra $\H_\alpha^\kappa$ is cellular, with cellular basis
\[
\left\{c^{\tuple{\lambda}}_\mathfrak{st} : \tuple{\lambda} \in \P^\kappa_\alpha \text{ and } \mathfrak{s}, \mathfrak{t} \in \T(\tuple{\lambda})\right\}
\]
(cf.~\cite[Corollary 5.12]{hu-mathas_graded}).

\subsubsection{Subalgebras of fixed points}
\label{subsubsection:subalgebras_fixed_points}

Recall from the introduction that we defined a subalgebra $\H_{p, n}^\kappa \subseteq \H_n^\kappa$ as the subalgebra of the fixed points of $\sigma : \H_n^\kappa \to \H_n^\kappa$. If $\mu : \H_n^\kappa \to \H_n^\kappa$ is the linear map defined by $\mu \coloneqq \sum_{j = 0}^{p-1} \sigma^j$, we have $\mu(\H^\kappa_n) = \H^\kappa_{p, n}$.

\begin{remarque}
We warn the reader that the map that we denoted by $\mu$ in \cite{rostam} is the map $\frac{1}{p}\mu$.
\end{remarque}

We now look at the blocks of $\H_n^\kappa$. Let $\alpha \in Q^\kappa_n$ and denote by $[\alpha]$ the orbit of $\alpha$ under the action of $\sigma$ (cf. Definition~\ref{definition:shift_alpha}). The subalgebra $\H_\alpha^\kappa \subseteq \H_n^\kappa$ is not necessarily stable under $\sigma$. Indeed, by~\eqref{equation:sigma_e(i)} we have
\begin{equation}
\label{equation:subalgebra-sigmaHalpha}
\sigma(\H_\alpha^\kappa) \subseteq \H_{\sigma \cdot\alpha}^\kappa.
\end{equation}
Hence, the smallest subalgebra of $\H_n^\kappa$ stable under $\sigma$ and containing $\H_\alpha^\kappa$ is
\[
\H_{[\alpha]}^\kappa \coloneqq \bigoplus_{\beta \in [\alpha]} \H_\beta^\kappa.
\]
Similarly, we define $\P_{[\alpha]}^\kappa \coloneqq \cup_{\beta \in [\alpha]} \P_\beta^\kappa$. Note that by Lemma~\ref{lemme:sigma_alpha_lambda} we have $[\tuple{\lambda}] \subseteq \P_{[\alpha]}^\kappa$. Hence, as for the tableaux, there is a well-defined equivalence relation $\sim$ on $\P_{[\alpha]}^\kappa$ generated by $\tuple{\lambda} \sim \prescript{\sigma}{}{\tuple{\lambda}}$. We write $\PP_{[\alpha]}^\kappa \coloneqq \P_{[\alpha]}^\kappa / {\sim}$ for the set of equivalence classes. As in \textsection\ref{subsubsection:cellularity_AK}, the algebra $\H_{[\alpha]}^\kappa$ is cellular, with cellular basis $\bigl\{c^{\tuple{\lambda}}_\mathfrak{st} : \tuple{\lambda} \in \P_{[\alpha]}^\kappa$ and $\mathfrak{s}, \mathfrak{t} \in \T(\tuple{\lambda})\bigr\}$.
As in the introduction, if $\H_{p, [\alpha]}^\kappa \subseteq \H_{[\alpha]}^\kappa$ denotes the subalgebra of fixed points of $\sigma$ then $\H_{p, [\alpha]}^\kappa = \mu(\H_{[\alpha]}^\kappa)$. 

\begin{proposition}
\label{proposition:basis_stable}
Let
\[
\phi : \bigcup_{[\tuple{\lambda}] \in \PP_{[\alpha]}^\kappa} \TT[\tuple{\lambda}] \longrightarrow \bigcup_{[\tuple{\lambda}] \in \PP_{[\alpha]}^\kappa} \T[\tuple{\lambda}],
\]
be a lift of the canonical projection. 
The family
\begin{equation}
\label{equation:basis_stable}
\left\{\mu(c^{\tuple{\lambda}}_\mathfrak{st}) : \tuple{\lambda} \in \P^\kappa_{[\alpha]}, \mathfrak{s} \in \T(\tuple\lambda), \mathfrak{t} \in \T^\phi_0(\tuple\lambda)\right\},
\end{equation}
is an $F$-basis of $\H^\kappa_{p, [\alpha]}$.
\end{proposition}

\begin{proof}
It suffices to prove that the family
\[
\left\{\sigma^j(c^{\tuple{\lambda}}_\mathfrak{st}) : j \in \{0, \dots, p-1\}, \tuple{\lambda} \in \P^\kappa_{[\alpha]}, \mathfrak{s} \in \T(\tuple{\lambda}), \mathfrak{t} \in \T_0^\phi(\tuple{\lambda})\right\},
\]
is an $F$-basis of $\H^\kappa_{[\alpha]}$. For any $j \in \{0, \dots, p-1\}$, define the idempotent
\[
e_j^\phi \coloneqq \sum_{\tuple{\lambda} \in \P^\kappa_{[\alpha]}} \sum_{\mathfrak{t} \in \T_j^\phi(\tuple{\lambda})} e\bigl(\res_\kappa(\mathfrak{t})\bigr).
\]
The family $\bigl\{c^{\tuple{\lambda}}_\mathfrak{st} : \tuple{\lambda} \in \P^\kappa_{[\alpha]}, \mathfrak{s} \in \T(\tuple{\lambda}), \mathfrak{t} \in \T_0^\phi(\tuple{\lambda})\bigr\}$ is an $F$-basis of $\H^\kappa_{[\alpha]} e_0^\phi$. Since $\kappa$ is compatible with $(d, \eta, p)$, for any $\tuple{\lambda} \in \P_{[\alpha]}^\kappa$ and any $\tuple{\lambda}$-tableau $\mathfrak{t}$ we have
\[
\res_\kappa\bigl(\prescript{\sigma}{}{\mathfrak{t}}\bigr) = \res_\kappa(\mathfrak{t}) + \tuple{\eta}.
\]
Using~\eqref{equation:tableaux:t_Tj}, we deduce that $\sigma(e_j^\phi) = e_{j+1}^\phi$ for all $j \in \{0, \dots, p-1\}$. Hence the family $\bigl\{\sigma^j(c^{\tuple{\lambda}}_\mathfrak{st}) : \tuple{\lambda} \in \P^\kappa_{[\alpha]}, \mathfrak{s} \in \T(\tuple{\lambda}), \mathfrak{t} \in \T_0^\phi(\tuple{\lambda})\bigr\}$ is an $F$-basis of $\H^\kappa_{[\alpha]} e_j$. By~\eqref{equation:sum_e(i)} and Lemma~\ref{lemma:non_zero_idempotents} we have $\sum_{j = 0}^{p-1} e_j^\phi = 1$ thus $\H_{[\alpha]}^\kappa = \oplus_{j = 0}^{p-1} \H_{[\alpha]}^\kappa e_j^\phi$ and we conclude.
\end{proof}

\begin{remarque}
Recall from Remark~\ref{remark:application_AK_graded} that the algebra $\H_{[\alpha]}^\kappa$ is $\mathbb{Z}$-graded. By~\cite{rostam}, the algebra $\H_{p, [\alpha]}^\kappa$ is also $\mathbb{Z}$-graded and the basis~\eqref{equation:basis_stable} is homogeneous.
\end{remarque}

We will prove the following partial alternative:
\begin{itemize}
\item if $\#[\alpha] = p$, the family~\eqref{equation:basis_stable} is a (graded) cellular basis of  $\H_{p, [\alpha]}^\kappa$, for a particular choice of lift $\phi$ (\textsection\ref{subsubsection:full_orbit});
\item if $\#[\alpha] < p$ and $p$ is odd, for any lift $\phi$ the family~\eqref{equation:basis_stable} is not an \emph{adapted} cellular basis of $\H_{p, [\alpha]}^\kappa$, in the sense of Definition~\ref{definition:*-cellularity} (\textsection\ref{subsubsection:adapted_cellularity}).
\end{itemize}

\subsubsection{Cellular basis in the full orbit case}
\label{subsubsection:full_orbit}
Let $\alpha \in Q^\kappa_n$ and assume that $\#[\alpha] = p$. By Lemma~\ref{lemme:sigma_alpha_lambda}, given $\tuple{\lambda} \in \P_{[\alpha]}^\kappa$ we know that for any $\mathfrak{t} \in \T(\tuple{\lambda})$ there is a unique standard tableau $\mathfrak{t}_\alpha \in [\mathfrak{t}]$ whose underlying $r$-partition is in $\P_\alpha^\kappa$. We have in fact $\mathfrak{t}_\alpha \in \T(\tuple{\lambda}_\alpha)$, where $\tuple{\lambda}_\alpha$ is the unique element of $[\tuple{\lambda}]$ that is in $\P_\alpha^\kappa$. We obtain a map
\[
\begin{array}{c|rcl}
\phi : &\displaystyle\bigcup_{[\tuple{\lambda}] \in \PP_{[\alpha]}^\kappa} \TT[\tuple{\lambda}] &\longrightarrow& \displaystyle\bigcup_{[\tuple{\lambda}] \in \PP_{[\alpha]}^\kappa} \T[\tuple{\lambda}]
\\
& \mathfrak{t} & \longmapsto & \mathfrak{t}_\alpha
\end{array},
\]
that lifts the natural projection. For any $\tuple{\lambda} \in \P_{[\alpha]}^\kappa$, we have
\[
\T_0^\phi(\tuple{\lambda}) = 
\begin{cases}
\T(\tuple{\lambda}), & \text{if } \tuple{\lambda} \in \P_\alpha^\kappa,
\\
\emptyset, & \text{otherwise}.
\end{cases}
\]
The basis~\eqref{equation:basis_stable} of $\H_{p, [\alpha]}^\kappa$ that we obtain is thus
\begin{equation}
\label{equation:basis_stable_orbite_pleine}
\left\{\mu(c^{\tuple{\lambda}}_\mathfrak{st}) : \tuple{\lambda} \in \P^\kappa_\alpha \text{ and } \mathfrak{s}, \mathfrak{t} \in \T(\tuple\lambda)\right\}.
\end{equation}
For any $\tuple{\lambda} \in \P_\alpha^\kappa$ and $\mathfrak{s}, \mathfrak{t} \in \T(\tuple{\lambda})$, we set $d_\mathfrak{st}^{\tuple{\lambda}} \coloneqq \mu(c^{\tuple{\lambda}}_\mathfrak{st})$.
Recall that $(\P_\alpha^\kappa, \T, c)$ is a cellular datum for $\H_\alpha^\kappa$.

\begin{proposition}
\label{proposition:cellularity_full_orbit}
Recall that $\#[\alpha] = p$. The triple $(\P_\alpha^\kappa, \T, d)$ is a cellular datum for $\H_{p, [\alpha]}^\kappa$.
\end{proposition}

\begin{proof}
It suffices to prove that $\mu$ commutes with $*$ and induces an algebra isomorphism between $\H_\alpha^\kappa$ and $\H_{p, [\alpha]}^\kappa$.
The first point is clear: indeed, since $*$ fixes each Khovanov--Lauda generator and by the action of $\sigma$ on these generators (cf.~\eqref{equation:sigma_graded_generators}) we know that $*$ and $\sigma$ commute.
Now, the restriction of $\mu$ to $\H_\alpha^\kappa$ is an algebra homomorphism. Indeed, for any $j \in \{1, \dots, p-1\}$ we have $\alpha \neq \sigma^j \cdot \alpha$ since $\#[\alpha] = p$, hence for any $h, h' \in \H_\alpha^\kappa$ we have $h\sigma^j(h') = 0$ (recall~\eqref{equation:e(i)e(j)} and~\eqref{equation:subalgebra-sigmaHalpha}). We conclude since by \eqref{equation:basis_stable_orbite_pleine}, $\mu$ sends a basis of $\H_\alpha^\kappa$ onto a basis of $\H_{p, [\alpha]}^\kappa$.
\end{proof}

\begin{corollaire}
If $p$ and $n$ are coprime then the algebra $\H_{p, n}^\kappa$ is cellular.
\end{corollaire}

\begin{proof}
Let us first prove that $\#[\beta] = p$ for all $\beta \in Q^+$ with $\lvert\beta\rvert = n$. If $\#[\beta] = p'$ then $p'$ divides $p$ and we can write
\[
\beta = \sum_{i = 0}^{p'\eta-1} \beta_i\bigl(\alpha_i + \alpha_{p'\eta + i} + \dots + \alpha_{(d-1)p'\eta + i}\bigr),
\]
where $d \coloneqq\frac{p}{p'}$ and $\beta_0, \dots, \beta_{p'\eta-1} \in \mathbb{N}$. We deduce that
\[
n = \lvert \beta \rvert = d \sum_{i = 0}^{p'\eta-1} \beta_i,
\]
hence $d$ divides $n$. But $d$ also divides $p$ thus $d = 1$ and $p' = p$ as desired. 
Hence, each subalgebra appearing in the following decomposition
\begin{equation}
\label{equation:full_orbit_dec_Hpn_Hpbeta}
\H_{p, n}^\kappa = \bigoplus_{[\beta] \in \mathfrak{Q}^\kappa_n} \H_{p, [\beta]}^\kappa,
\end{equation}
is cellular by Proposition~\ref{proposition:cellularity_full_orbit}, where $\mathfrak{Q}^\kappa_n$ is the quotient of $Q^\kappa_n$ by the equivalence relation $\sim$ generated by $\beta \sim \sigma\cdot\beta$ for all $\beta \in Q^\kappa_n$. We now easily check that $\H_{p, n}^\kappa$ is cellular, using the following fact:  for any $[\beta] \neq [\beta'] \in \mathfrak{Q}^\kappa_n$ we have $hh' = 0$ for all $h \in \H_{[\beta]}^\kappa$ and $h' \in \H_{[\beta']}^\kappa$ (cf. \eqref{equation:e(i)e(j)}).
\end{proof}

\subsubsection{Adapted cellularity}
\label{subsubsection:adapted_cellularity}

Let $\alpha \in Q^\kappa_n$  and let $\phi$ be as in Proposition~\ref{proposition:basis_stable}. By \eqref{equation:card_Tparenthese_T0crochet}, we have
\begin{align*}
\dim \H^\kappa_{p, [\alpha]}
&=
\sum_{\tuple{\lambda} \in \P_{[\alpha]}^\kappa} \bigl(\#\T(\tuple{\lambda})\bigr)\bigl(\#\T_0^\phi(\tuple{\lambda})\bigr)
\\
&=
\sum_{\tuple{\lambda} \in \P_{[\alpha]}^\kappa} \frac{p}{\#[\tuple{\lambda}]}\bigl(\#\T_0^\phi[\tuple{\lambda}]\bigr)\bigl(\#\T_0^\phi(\tuple{\lambda})\bigr)
\\
&=
\sum_{[\tuple{\lambda}] \in \PP_{[\alpha]}^\kappa} \frac{p}{\#[\tuple{\lambda}]}\bigl(\#\T_0^\phi[\tuple{\lambda}]\bigr) \sum_{\tuple{\mu} \in [\tuple{\lambda}]}\#\T_0^\phi(\tuple{\mu})
\\
&=
\sum_{[\tuple{\lambda}] \in \PP^\kappa_{[\alpha]}} \frac{p}{\#[\tuple{\lambda}]}\bigl(\#\T_0^\phi[\tuple{\lambda}]\bigr)^2.
\end{align*}
Recalling that $\#\T_0^\phi[\tuple{\lambda}]$ does not depend on $\phi$, we obtain
\begin{equation}
\label{equation:sum_dim_Hpn}
\dim \H_{p, [\alpha]}^\kappa = \sum_{[\tuple{\lambda}] \in \PP^\kappa_{[\alpha]}} \frac{p}{\#[\tuple{\lambda}]}\bigl(\#\T_0[\tuple{\lambda}]\bigr)^2.
\end{equation}

\begin{remarque}
With~\eqref{equation:card_Tparenthese_T0crochet} and Remark~\ref{remark:cellular_algebra_dimension} we obtain the equality $\dim\H_{p,[\alpha]}^\kappa = \frac{1}{p}\dim \H_{[\alpha]}^\kappa$.
\end{remarque}


Suppose that there exists a cellular datum $(\Lambda, \T, c)$ for $\H_{p, [\alpha]}^\kappa$. Remark~\ref{remark:cellular_algebra_dimension} and~\eqref{equation:sum_dim_Hpn} give two ways to write $\dim \H_{p, [\alpha]}^\kappa$ as a sum of squares:
\[
\dim \H_{p, [\alpha]}^\kappa
= \sum_{\lambda \in \Lambda} \#\T(\lambda)^2
= \sum_{[\tuple{\lambda}] \in \PP^\kappa_{[\alpha]}} \frac{p}{\#[\tuple{\lambda}]}\bigl(\#\T_0[\tuple{\lambda}]\bigr)^2. 
\]
These two sums have the same terms up to reordering if and only if for all $[\tuple{\lambda}] \in \PP_{[\alpha]}^\kappa$, there exist  $\lambda_{[\tuple{\lambda}], 1}, \dots, \lambda_{[\tuple{\lambda}], \frac{p}{\#[\tuple{\lambda}]}} \in \Lambda$ such that
\begin{subequations}
\label{subequations:adapted_cellularity}
\begin{align}
\#\T(\lambda_{[\tuple{\lambda}], j}) &= \#\T_0[\tuple{\lambda}],
&&
\text{for all } j \in \Bigl\{1, \dots, \frac{p}{\#[\tuple{\lambda}]}\Bigr\},
\end{align}
and
\begin{equation}
\left\{\lambda_{[\tuple{\lambda}], j} : [\tuple{\lambda}] \in \PP_{[\alpha]}^\kappa \text{ and } j \in \Bigl\{1, \dots, \frac{p}{\#[\tuple{\lambda}]}\Bigr\}\right\} = \Lambda.
\end{equation}
\end{subequations}
Recall that the anti-automorphism $* : \H_n^\kappa \to \H_n^\kappa$ was fixed in \textsection\ref{subsubsection:cellularity_AK}. 
\begin{definition}
\label{definition:*-cellularity}
Suppose that $(\Lambda, \T, c)$ is a cellular datum for $\H_{p, [\alpha]}^\kappa$. We say that $(\Lambda, \T, c)$ is an \emph{adapted} cellular datum if for all $[\tuple{\lambda}] \in \PP_n^\kappa$, there exist $\lambda_{[\tuple{\lambda}], 1}, \dots, \lambda_{[\tuple{\lambda}], \frac{p}{\#[\tuple{\lambda}]}} \in \Lambda$ such that the conditions~\eqref{subequations:adapted_cellularity} are satisfied, together with $\bigl(c^\lambda_\mathfrak{st}\bigr)^* = c^\lambda_\mathfrak{ts}$ for all $\lambda \in \Lambda$ and $\mathfrak{s}, \mathfrak{t} \in \T(\lambda)$. 
\end{definition}

We say that a basis $\mathcal{B}$ of $\H_{p, [\alpha]}^\kappa$ is \emph{adapted cellular} if there exists an adapted cellular datum $(\Lambda, \T, c)$ for $\H_{p, [\alpha]}^\kappa$ such that $\mathcal{B}$ coincides with $\bigl\{c^\lambda_\mathfrak{st} : \lambda \in \Lambda$ and $\mathfrak{s}, \mathfrak{t} \in \T(\lambda)\bigr\}$.

\begin{lemme}
\label{lemma:mu_star}
Let $\tuple{\lambda} \in \P^\kappa_n$ and $\mathfrak{s}, \mathfrak{t} \in \T(\tuple{\lambda})$. Then ${\mu(c^{\tuple{\lambda}}_\mathfrak{st})}^* = \mu(c^{\tuple{\lambda}}_\mathfrak{st})$ if and only if
\begin{align*}
\mathfrak{s} &= \mathfrak{t},  & &\text{if } p \text{ is odd},
\\
\mathfrak{s} &= \mathfrak{t} \text{ or } \sigma^{p/2}(c^{\tuple{\lambda}}_\mathfrak{st}) = c^{\tuple{\lambda}}_\mathfrak{ts}, &&\text{if } p \text{ is even}.
\end{align*}
\end{lemme}

\begin{proof}
Since $\mu$ and $*$ commute, we have
\[
{\mu(c^{\tuple{\lambda}}_\mathfrak{st})}^*
=
\mu(c^{\tuple{\lambda}}_\mathfrak{ts})
=
\sum_{j = 0}^{p-1} \sigma^j(c^{\tuple{\lambda}}_\mathfrak{ts}).
\]
Thus, if ${\mu(c^{\tuple{\lambda}}_\mathfrak{st})}^* = \mu(c^{\tuple{\lambda}}_\mathfrak{st})$ then
\[
\sum_{j = 0}^{p-1} \sigma^j(c^{\tuple{\lambda}}_\mathfrak{ts}) = \sum_{j = 0}^{p-1} \sigma^j(c^{\tuple{\lambda}}_\mathfrak{st}).
\]
By~\eqref{equation:sigma_e(i)}, \eqref{equation:e(i)_idempotent}, \eqref{equation:e(i)e(j)} and~\eqref{equation:cellular_basis_idempotent}, we deduce that there exists $j \in \{0, \dots, p-1\}$ such that
\begin{equation}
\label{equation:c_st_sigma}
c^{\tuple{\lambda}}_\mathfrak{ts} = \sigma^j(c^{\tuple{\lambda}}_\mathfrak{st}).
\end{equation}
Since $\sigma$ and $*$ commute, we obtain
\[
c^{\tuple{\lambda}}_\mathfrak{st} = \sigma^j(c^{\tuple{\lambda}}_\mathfrak{ts}),
\]
thus,
\[
\sigma^j(c^{\tuple{\lambda}}_\mathfrak{st}) = \sigma^{2j}(c^{\tuple{\lambda}}_\mathfrak{ts}).
\]
Combining with~\eqref{equation:c_st_sigma}, we obtain
\[
c^{\tuple{\lambda}}_\mathfrak{ts} = \sigma^{2j}(c^{\tuple{\lambda}}_\mathfrak{ts}).
\]
By~\eqref{equation:sigma_e(i)}, \eqref{equation:e(i)_idempotent} and~\eqref{equation:e(i)e(j)} and since $\tuple{\eta} \in (\mathbb{Z}/e\mathbb{Z})^n$ has order $p$, this equality implies that $2j \in \{0, p\}$.  If $p$ is odd then $j = 0$ and~\eqref{equation:c_st_sigma} yield $c^{\tuple{\lambda}}_\mathfrak{ts} = c^{\tuple{\lambda}}_\mathfrak{st}$ thus $\mathfrak{s} = \mathfrak{t}$. If $p$ is even then $j \in \{0, \frac{p}{2}\}$ and similarly we conclude using~\eqref{equation:c_st_sigma}. The converse is straightforward.
\end{proof}

Given the result of \textsection\ref{subsubsection:full_orbit}, it seems natural to look for a cellular basis for $\H_{p, [\alpha]}^\kappa$ of the form~\eqref{equation:basis_stable}. The following proposition uses Corollary~\ref{corollaire:orbite_gnrl} to give a partial answer to this problem.

\begin{proposition}
\label{proposition:[alpha]<p_adapted_cellular}
If $\#[\alpha] < p$ and $p$ is odd then the basis~\eqref{equation:basis_stable}
\[
\left\{\mu(c^{\tuple{\lambda}}_\mathfrak{st}) : \tuple{\lambda} \in \P^\kappa_{[\alpha]}, \mathfrak{s} \in \T(\tuple\lambda), \mathfrak{t} \in \T^\phi_0(\tuple\lambda)\right\},
\]
of $\H_{p, [\alpha]}^\kappa$ is not adapted cellular.
\end{proposition}

\begin{proof}
Let $N$ be the cardinality of
\[
\left\{\mu(c^{\tuple{\lambda}}_\mathfrak{st}) :  \tuple{\lambda} \in \P_{[\alpha]}^\kappa, \mathfrak{s} \in \T(\tuple{\lambda}), \mathfrak{t} \in \T_0^\phi(\tuple{\lambda}), {\mu(c^{\tuple{\lambda}}_\mathfrak{st})}^* = \mu(c^{\tuple{\lambda}}_\mathfrak{st})\right\}.
\]
Assume that the basis~\eqref{equation:basis_stable} is adapted cellular with associated cellular datum $(\Lambda, \T, c)$. Lemma~\ref{lemma:cellular_algebra_number_fixed_points} yields, with the notation of Definition~\ref{definition:*-cellularity},
\begin{align*}
N
&=
\sum_{\lambda \in \Lambda} \#\T(\lambda)
\\
&=
\sum_{[\tuple{\lambda}] \in \PP_{[\alpha]}^\kappa} \sum_{j = 1}^{\frac{p}{\#[\tuple{\lambda}]}} \#\T(\lambda_{[\tuple{\lambda}], j})
\\
&=
\sum_{[\tuple{\lambda}] \in \PP_{[\alpha]}^\kappa} \sum_{j = 1}^{\frac{p}{\#[\tuple{\lambda}]}} \#\T_0[\tuple{\lambda}]
\\
&=
\sum_{[\tuple{\lambda}] \in \PP_{[\alpha]}^\kappa} \frac{p}{\#[\tuple{\lambda}]} \#\T_0[\tuple{\lambda}].
\end{align*}
We have $\frac{p}{\#[\tuple{\lambda}]} \geq 1$ for all $[\tuple{\lambda}] \in \PP_{[\alpha]}^\kappa$. Moreover,  since $\#[\alpha] < p$ we know by Corollary~\ref{corollaire:orbite_gnrl} that there exists $[\tuple{\lambda}] \in \PP_{[\alpha]}^\kappa$ such that $\frac{p}{\#[\tuple{\lambda}]} > 1$. Thus, we obtain
\begin{equation}
\label{equation:proof_[alpha]<p}
N > \sum_{[\tuple{\lambda}] \in \PP_{[\alpha]}^\kappa}  \#\T_0[\tuple{\lambda}].
\end{equation}
But now $p$ is odd, thus by Lemma~\ref{lemma:mu_star} we know that
\[
\mu{(c^{\tuple{\lambda}}_\mathfrak{st})}^* = c^{\tuple{\lambda}}_\mathfrak{st} \iff \mathfrak{s} = \mathfrak{t},
\]
for all $\tuple{\lambda} \in \P_{[\alpha]}^\kappa$, $\mathfrak{s} \in \T(\tuple{\lambda})$ and $\mathfrak{t} \in \T_0^\phi(\tuple{\lambda})$. Hence, the only elements of the basis~\eqref{equation:basis_stable} that are fixed by the $*$ anti-automorphism are the $\mu(c^{\tuple{\lambda}}_\mathfrak{ss})$ for all  $\tuple{\lambda} \in \P_{[\alpha]}^\kappa$ and $\mathfrak{s} \in \T_0^\phi(\tuple{\lambda})$. We obtain
\[
N = \sum_{\tuple{\lambda} \in \P_{[\alpha]}^\kappa}  \#\T_0^\phi(\tuple{\lambda})
=  \sum_{\tuple{\lambda} \in \P_{[\alpha]}^\kappa}  \#\T_0(\tuple{\lambda})
= \sum_{[\tuple{\lambda}] \in \PP_{[\alpha]}^\kappa}  \#\T_0[\tuple{\lambda}],
\]
which contradicts~\eqref{equation:proof_[alpha]<p}.
\end{proof}

\begin{remarque}
We can also define an adapted cellularity for $\H_{p, n}^\kappa$, similarly to Definition~\ref{definition:*-cellularity}. Using Proposition~\ref{proposition:introduction_minmax}, we can show that if $p$ and $n$ are not coprime and $p$ is odd, then the basis of $\H_{p, n}^\kappa$ that we obtain from~\eqref{equation:basis_stable} and~\eqref{equation:full_orbit_dec_Hpn_Hpbeta} is not adapted cellular. Note that, under these conditions, there can exist an $\alpha \in Q^\kappa_n$ with $\#[\alpha]  = p$, so that the subalgebra $\H_{p, [\alpha]}^\kappa$ is cellular (cf. \textsection\ref{subsubsection:full_orbit}). This explains why we are dealing with $\H_{p, [\alpha]}^\kappa$ and not only with $\H_{p, n}^\kappa$.
\end{remarque}

\subsection{Restriction of Specht modules}
\label{subsection:restriction_specht}

Since we have a cellular datum $(\P_n^\kappa, \T, c)$ for the algebra $\H_n^\kappa$, we have a collection of cell modules $\{\mathcal{S}^{\tuple{\lambda}} : \tuple{\lambda} \in \P_n^\kappa\}$. In this case, the cell modules are called \emph{Specht modules}.
The algebra $\H_{p, n}^\kappa$ is not known to be cellular in general, but Hu and Mathas~\cite{hu-mathas_decomposition} defined what they also called \emph{Specht modules} for $\H_{p, n}^\kappa$. It is a family
\[
\left\{
\mathcal{S}^{\tuple{\lambda}}_j : j \in \{0, \dots, \frac{p}{\#[\tuple{\lambda}]} - 1\}\right\},
\]
of $\H_{p, n}^\kappa$-modules with
\begin{equation}
\label{equation:application_specht}
\left.\mathcal{S}^{\tuple{\lambda}}\right\downarrow^{\H_n^\kappa}_{\H_{p, n}^\kappa} \simeq \mathcal{S}^{\tuple{\lambda}}_0 \oplus \dots \oplus \mathcal{S}^{\tuple{\lambda}}_{\frac{p}{\#[\tuple{\lambda}]} - 1},
\end{equation}
for any $\tuple{\lambda} \in \P_n^\kappa$, where $\left.\mathcal{S}^{\tuple{\lambda}}\right\downarrow^{\H_n^\kappa}_{\H_{p, n}^\kappa}$ denotes the restriction of the $\H_n^\kappa$-module $\mathcal{S}^{\tuple{\lambda}}$ to an $\H_{p, n}^\kappa$-module. For any $\tuple{\lambda} \in \P_n^\kappa$ and $j, j' \in \{0, \dots, \frac{p}{\#[\tuple{\lambda}]} - 1\}$, the $\H_{p, n}^\kappa$-modules $\mathcal{S}^{\tuple{\lambda}}_j$ and $\mathcal{S}^{\tuple{\lambda}}_{j'}$ are isomorphic up to a twist of the action of $\H_{p, n}^\kappa$. The purpose of the name ``Specht module'' is that each irreducible $\H_{p, n}^\kappa$-module is isomorphic to the head of a $\mathcal{S}^{\tuple{\lambda}}_j$.

By Proposition~\ref{proposition:introduction_minmax}, we know that the maximal number of summands in~\eqref{equation:application_specht} is $\gcd(p, n)$ when we restrict a Specht module of $\H_n^\kappa$ and that this bound is reached. Our result Corollary~\ref{corollaire:orbite_gnrl} refines this result.

\begin{proposition}
For any $\alpha \in Q^\kappa_n$, the maximal number of summands in~\eqref{equation:application_specht} is $\frac{p}{\#[\alpha]}$ and this bound is reached, when we restrict a Specht module $\mathcal{S}^{\tuple{\lambda}}$ with $\tuple{\lambda} \in \P_{[\alpha]}^\kappa$, that is, when we restrict a Specht module of $\H_{[\alpha]}^\kappa$. 
\end{proposition}

\paragraph{Acknowledgements} I am thankful to Maria Chlouveraki and Nicolas Jacon for their numerous corrections and suggestions about the presentation of the paper. I also thank an anonymous referee for pointing out the reference~\cite{fayers_weight} for Proposition~\ref{proposition:residu_0_chapeau}.

\end{document}